\documentclass[11pt]{article}

\usepackage{amsmath,amssymb,amsthm,dsfont,mathrsfs,stmaryrd}
\usepackage[T1]{fontenc}
\usepackage[utf8]{inputenc}
\usepackage{titlesec,titletoc}

\usepackage{etoolbox}

\makeatletter
\patchcmd{\ttlh@hang}{\parindent\z@}{\parindent\z@\leavevmode}{}{}
\patchcmd{\ttlh@hang}{\noindent}{}{}{}
\makeatother

\usepackage[all,cmtip]{xy}
\usepackage[fleqn,tbtags]{mathtools}
\usepackage{mathabx}
\usepackage{url}
\usepackage[numbers]{natbib}
\usepackage{tikz,tikz-cd}
\usepackage{enumerate}
\usepackage{hyperref}
\usetikzlibrary{calc,fit}

\theoremstyle{plain}
\newtheorem{theo}{Theorem}[subsection]
\newtheorem{prop}[theo]{Proposition}
\newtheorem{cor}[theo]{Corollary}

\newtheorem{lem}[theo]{Lemma}

\theoremstyle{definition}
\newtheorem{defi}[theo]{Definition}
\newtheorem{ex}[theo]{Example}

\theoremstyle{remark}
\newtheorem{rem}[theo]{Remark}

\DeclareMathOperator{\id}{Id}

\newcommand{\Uqhatg}{\mathcal{U}_{q}(\hat{\mathfrak{g}})}

\newcommand{\C}{\mathscr{C}}
\newcommand{\qt}{(q,t)}
\newcommand{\g}{\mathfrak{g}}
\newcommand{\Z}{\mathbb{Z}}
\newcommand{\N}{\mathbb{N}}

\newcommand{\Yt}{\mathcal{Y}_t}

\newcommand{\F}{\mathcal{F}}

\newcommand{\I}{\hat{I}}
\newcommand{\mI}{\hat{I}^-}
\newcommand{\It}{\hat{I}^-_{\leq 2}}
\newcommand{\Gm}{G^-}
\newcommand{\M}{\mathcal{M}}
\newcommand{\Mm}{\mathcal{M}^-}
\newcommand{\CZ}{\mathscr{C}_\mathbb{Z}}
\newcommand{\CZm}{\mathscr{C}_\mathbb{Z}^-}
\newcommand{\qm}{\chi_q^-}
\newcommand{\A}{ \mathcal{A}}
\newcommand{\Y}{\mathcal{Y}}
\newcommand{\cO}{\mathcal{O}}

\numberwithin{equation}{section}
\author{Léa Bittmann}
\title{A quantum cluster algebra approach to representations of simply-laced quantum affine algebras}

\begin{document}

\raggedbottom

\tikzset{->-/.style={decoration={
  markings,
  mark=at position .5 with {\arrow{>}}},postaction={decorate}}}

\maketitle

\begin{abstract}
We establish a quantum cluster algebra structure on the quantum Grothendieck ring of a certain monoidal subcategory of the category of finite-dimensional representations of a simply-laced quantum affine algebra. Moreover, the $(q,t)$-characters of certain irreducible representations, among which fundamental representations, are obtained as quantum cluster variables. This approach gives a new algorithm to compute these $(q,t)$-characters. As an application, we prove that the quantum Grothendieck ring of a larger category of representations of the Borel subalgebra of the quantum affine algebra, defined in a previous work as a quantum cluster algebra, contains indeed the well-known quantum Grothendieck ring of the category of finite-dimensional representations. Finally, we display our algorithm on a concrete example.
\end{abstract}

\setcounter{tocdepth}{1}
\tableofcontents


	\section{Introduction}

Finite-dimensional representations of quantum affine algebras have been classified by Chari-Pressley \citep{CP95} with a quantum analog of Cartan's highest weight classification of finite-dimensional representations of simple Lie algebras. Combining this classification with the notations from Frenkel-Reshetikhin $q$-character \citep{QCRQAA} theory, one gets the following.

Let $\g$ be a finite-dimensional simple Lie algebra, and $\Uqhatg$ be the quantum affine algebra. Irreducible finite-dimensional representations of $\Uqhatg$ are indexed by monomials in the infinite set of variables $\{Y_{i,a}\}_{i\in I, a\in \mathbb{C}^\times}$, where $I=\{1, \ldots,n\}$ are the vertices of the Dynkin diagram of $\g$. For such a monomial $m$, the corresponding simple $\Uqhatg$-module is denoted by $L(m)$. If the monomial is just one term $m=Y_{i,a}$, the corresponding simple module $L(Y_{i,a})$ is called a \emph{fundamental module}. Chari-Pressley classification result also implies that every simple module can be obtained as a subquotient of a tensor product of such fundamental modules. 

This classification is a major result in the way of obtaining information of the finite-dimensional representations of quantum affine algebras. However, it gives limited information regarding the structure of the modules in themselves. For that purpose, Frenkel and Reshetikhin have developed a theory of $q$-characters, giving the decomposition of the modules into generalized eigenspaces for the action of a large commutative subalgebra of $\Uqhatg$. Frenkel-Mukhin established an algorithm to compute those $q$-characters \citep{CqC}. This algorithm is guaranteed to work on fundamental modules, but not on all irreducible $\Uqhatg$-modules \citep{oFMA}.

Then, when $\g$ is of simply-laced type, Nakajima \citep{tAqC} used the input from geometry, and more precisely perverse sheaves on quiver varieties, to construct $t$-deformations of these $q$-characters, called $\qt$-characters, as elements of a \emph{quantum Grothendieck ring}. He introduced a second base for the Grothendieck ring of the category of finite-dimensional representations of $\Uqhatg$, also indexed by the monomials in the variables $\{Y_{i,a}\}_{i\in I, a\in \mathbb{C}^\times}$, formed by the \emph{standard modules}. Geometrically, these standard modules correspond to constant sheaves, but algebraically, for each monomial $m$, the standard module $M(m)$ can be seen as the tensor product of the fundamental modules corresponding to each of the factors in $m$, in a particular order (see also \citep{SMQAA}).

He first used a $t$-deformed version of Frenkel-Mukhin's algorithm to compute $\qt$-characters for the fundamental modules, then extended the $\qt$-characters to the standard modules, denoted $[M(m)]_t$. Next, he defined $\qt$-characters for the simple modules as some unique family of elements $[L(m)]_t$ of the quantum Grothendieck ring satisfying some invariant property, as well as having a decomposition of the form
\begin{equation}\label{inverseKL}
[L(m)]_t = [M(m)]_t + \sum_{m'<m}Q_{m',m}(t)[M(m')]_t,
\end{equation}
where $<$ is a partial order on the set of Laurent monomials in the variables $\{Y_{i,a}\}_{i\in I, a\in \mathbb{C}^\times}$, defined by Nakajima, and $Q_{m',m}(t)\in \Z[t^{\pm 1}]$ is a Laurent polynomial. 

Nakajima then showed that these $\qt$-characters were indeed $t$-deformations of the $q$-characters, in the sense that the evaluation of the $\qt$-characters at $t=1$ recovers the $q$-characters. Finally, inverting the unitriangular decomposition (\ref{inverseKL}), one gets an algorithm, of the Kazhdan-Lusztig type, to compute the $\qt$-characters, and so the $q$-characters of all simple finite-dimensional $\Uqhatg$-modules.

This algorithm is theoretically computable, but as noted in \citep{NE8}, trying to compute it in reality can easily exceed the size of computer memory available. The first step of the algorithm is to compute the $\qt$-characters of the fundamental representations, and for example, for $\g$ of type $E_8$, the 5th fundamental representation requires 120Go of memory to compute. 

In \citep{CAQAA} Hernandez and Leclerc introduced a new point of view on representations of quantum affine algebras, using the theory of cluster algebras that was developed by Fomin and Zelevinsky in the early 2000's \citep{CA1}, \citep{CA2}, \citep{CA3},\citep{CA4}. In \citep{ACAA} they established a new algorithm to compute $q$-characters of a particular class of irreducible modules, called \emph{Kirillov-Reshetikhin modules}, which include the fundamental modules, using the cluster algebra structure of the Grothendieck ring of a subcategory of the category of finite-dimensional $\Uqhatg$-modules. The picture is completed when put into the broader context of the category $\cO^+$ of representations of $\Uqhatg$, introduced by Hernandez-Jimbo in \citep{ARDRF}. In \citep{CABS}, Hernandez and Leclerc showed that the Grothendieck ring of this category, which contains the finite-dimensional representations, is isomorphic to a cluster algebra built on an infinite quiver, while explicitly giving the identification. 

In a previous work \citep{LEA2}, the author defined the quantum Grothendieck ring for this category $\cO^+$ of representations as a quantum cluster algebra, as defined by Berenstein and Zelevinsky \citep{qCA}. However, the question of whether this quantum Grothendieck ring contained the quantum Grothendieck of the category of finite-dimensional representation, as used by Nakajima, was only proven in type $A$, and remained conjectural for other types. 

In this Chapter we propose to show that, when $\g$ is of simply-laced type, the quantum Grothendieck ring of a certain monoidal subcategory of the category of finite-dimensional $\Uqhatg$-modules has a quantum cluster algebra structure (Proposition~\ref{propfinale}). The proof relies heavily on a family of relations satisfied by the $\qt$-characters of the Kirillov-Reshetikhin modules called \emph{quantum $T$-systems} proved in \citep{tAqCKR}. These relations are $t$-deformations of the $T$-systems relations, first stated in \citep{FRSLM}. These relations have not been generalized to non-simply-laced types, except for type $B_n$ in \citep{HO19}. This is the main reason why the results of this chapter are limited to $ADE$ types. 
This quantum cluster algebra approach gives a new algorithm to compute the $\qt$-characters of the Kirillov-Reshetikhin modules, and in particular of the fundamental modules (see Proposition~\ref{propyirm}). This algorithm seems more efficient, at least in terms of number of steps, than the Frenkel-Mukhin algorithm (see Remark~\ref{remcomplexity}). 

For certain subcategories of the category of finite-dimensional $\Uqhatg$-modules generated by a finite number of fundamental modules, Qin obtained in \citep{QIN} in a different context results similar to some results whose direct proofs are given here (see Remark~\ref{remQin}). In this present work, we give explicit sequences of mutations to obtain 
$\qt$-characters of fundamental modules.

Next, we use this new result to prove a conjecture that was stated by the author in  the aforementionned work \citep{LEA2}. This previous work dealt with a category $\cO^+$ of representations of the Borel subalgebra of the quantum affine algebra, which contains the finite-dimensional $\Uqhatg$-modules. The quantum Grothendieck ring of this category was defined as a quantum cluster algebra, and it was conjectured that this ring contained the quantum Grothendieck ring of the category of finite-dimensional representations. Here, we show that the quantum cluster algebra considered in  \citep{LEA2} can be seen as a twisted version (in the sense of \citep{GRA2}) of the quantum cluster algebra occurring in the finite-dimensional case (see Proposition~\ref{propisoquant}). As an application, the $\qt$-characters of the fundamental modules are obtained as quantum cluster variables in the quantum Grothendieck ring of the category $\cO^+$ (Proposition~\ref{propfinale}), and the inclusion of quantum Grothendieck rings conjectured in  \citep{LEA2} follows naturally (Theorem~\ref{theoconj2} and Corollary~\ref{corconj}). 

Note that these results extend the algorithm to compute $\qt$-characters of some simple modules in the category $\cO^+$. However, for this category of representations, the question of defining analogs of standard modules remains open. The author tackled this question in a previous work \citep{LEA}, and gave a complete answer when the underlying simple Lie algebra is $\g=\mathfrak{sl}_2$. This work is also a partial answer to the first point of Nakajima's "to do" list from \citep{QVCA}.

The author would also like to note that in type $A$, parallel results to the ones presented here were proven in \citep{Bolor}, via a different approach. In this work, $\qt$-characters of Kirillov-Reshetikhin modules are also obtained as quantum cluster variables in some quantum cluster algebra, the method uses a generalization of the tableaux-sum notations introduced by Nakajima in \citep{NtqAD}. 

Finally, we use this algorithm to explicitly compute, when $\g$ is of type $D_4$, the $\qt$-character of the fundamental representation at the trivalent node. 

This paper is organized as follows. In Sections \ref{sectCartan} and \ref{sectFDRQAA} we recall notations and results regarding finite-dimensional representations of quantum affine algebras. In Section \ref{sectqGRs}, we recall results regarding the $t$-deformation of Grothendieck rings, such as $\qt$-characters and quantum $T$-systems. In Section \ref{sectqCAstructure} we prove the existence of a quantum cluster algebra $\A_t$ with $t$-commutations relations coherent with the framework of $\qt$-characters. Then, in Section \ref{sectqCAqGR} we prove that this quantum cluster algebra is isomorphic to the quantum Grothendieck ring of a certain monoidal subcategory of the category of finite-dimensional $\Uqhatg$-modules; in this process, we established an algorithm to compute $\qt$-characters of Kirillov-Reshetikhin modules. Section \ref{sectapp} is devoted to the category $\cO^+$, and to the proof of the inclusion Conjecture of \citep{LEA2}. Finally, the explicit computation mentioned just above is done in Section \ref{sectexplD4}.
	
	\subsection*{Acknowledgements}
	
I would like to thank my PhD advisor, David Hernandez, under whose supervision most of the work presented here was carried out. I thank him for his constant support, for his time, and for the numerous discussions. I also thank Bernhard Keller for sharing with me the lastest update of his quiver mutation applet, as well as instructions for the new \emph{highest weight} feature.

	\section{Cartan data and quantum Cartan data}\label{sectCartan}

		\subsection{Root data}\label{sectdata}	
		
Let us fix some notations for the rest of the paper. Let the $\g$ be a simple Lie algebra of rank $n$ and of type $A$,$D$ or $E$. This restriction is necessary as one of the main arguments of the proof is the quantum $T$-systems, which have only been proven for these types as yet. Let $\gamma$ be the Dynkin diagram of $\g$ and let $I:=\{ 1, \ldots, n \}$ be the set of vertices of $\gamma$. 

The \emph{Cartan matrix} of $\g$ is the $n\times n$ matrix $C$ such that 
\begin{equation*}
C_{i,j}=\left\lbrace \begin{array}{rl}
				2 & \text{ if }  i=j,\\
				-1 & \text{ if }  i \sim j \quad\text{( } i \text{ and } j \text{ are adjacent vertices of } \gamma  \text{ ) },\\
				0 & \text{ otherwise.}
\end{array}\right.
\end{equation*}

Let us denote by $(\alpha_{i})_{i\in I}$ the simple roots of $\mathfrak{g}$, $(\alpha_{i}^{\vee})_{i\in I}$ the simple coroots and $(\omega_{i})_{i\in I}$ the fundamental weights. We will use the usual lattices $Q= \bigoplus_{i\in I}\mathbb{Z}\alpha_{i}$, $Q^{+}= \bigoplus_{i\in I}\mathbb{N}\alpha_{i}$ and $P= \bigoplus_{i\in I}\mathbb{Z}\omega_{i}$. Let $P_{\mathbb{Q}}=P\otimes \mathbb{Q}$, endowed with the partial ordering : $\omega \leq \omega'$ if and only if $\omega' - \omega \in Q^{+}$. 

The Dynkin diagram of $\g$ is numbered as in \citep{IDLA}, and let $a_1,a_2,\ldots,a_n$ be the Kac labels ($a_0=1$).

Let $h$ be the (dual) Coxeter number of $\g$:
\begin{equation}\label{tabh}
\begin{array}{|c|c|c|c|c|c|}
\hline
\g & A_n & D_n & E_6 & E_7 & E_8\\
\hline
h & n+1 & 2n-2 & 12 & 18 & 30\\
\hline
\end{array}
\end{equation}

		\subsection{Quantum Cartan matrix}\label{sectQCM}
		
Let $z$ be an indeterminate.

\begin{defi}
The \textit{quantum Cartan matrix} of $\g$ is the matrix $C(z)$ with entries,
\begin{equation*}
C_{ij}(z)=\left\lbrace \begin{array}{cl}
				z+z^{-1} & \text{ if } i=j,\\
				-1 & \text{ if } i \sim j,\\
				0 & \text{ otherwise.}
\end{array}\right.
\end{equation*}
\end{defi}
\begin{rem}
The evaluation $C(1)$ is the Cartan matrix of $\mathfrak{g}$. As $\det(C)\neq 0$, then $\det(C(z))\neq 0$ and we can define $\tilde{C}(z)$, the inverse of the matrix $C(z)$. The entries of the matrix $\tilde{C}(z)$ belong to $\mathbb{Q}(z)$. 
\end{rem}

One can write 
\begin{equation*}
C(z) = (z+z^{-1})\id - A,
\end{equation*}
where $A$ is the adjacency matrix of $\gamma$. Hence,
\begin{equation*}
\tilde{C}(z) = \sum_{m=0}^{+\infty}(z+z^{-1})^{-m-1}A^{m}.
\end{equation*}
Therefore, we can write the entries of $\tilde{C}(z)$ as power series in $z$. For all $i,j\in I$,
\begin{equation}\label{qCm}
\tilde{C}_{i j}(z) = \sum_{m=1}^{+\infty}\tilde{C}_{i,j}(m)z^{m}\quad \in \mathbb{Z}[[z]].
\end{equation}
\begin{ex}\begin{enumerate}[(i)]\label{exsl2-1}
	\item For $\g=\mathfrak{sl}_{2}$, one has 
\begin{equation}
\tilde{C}_{1 1} = \sum_{n=0}^{+\infty}(-1)^{n}z^{2n+1} = z-z^{3}+z^{5}-z^{7}+z^{9}-z^{11} + \cdots \\
\end{equation}
		\item For $\g=\mathfrak{sl}_{3}$, one has 
\begin{align*}
\tilde{C}_{ii} &= z-z^{5}+z^{7}-z^{11}+z^{13} + \cdots, \quad 1 \leq i \leq 2\\
\tilde{C}_{ij} &= z^{2}-z^{4}+z^{8}-z^{10}+z^{14} + \cdots, \quad 1 \leq i \neq j \leq 2.
\end{align*}
\end{enumerate}
\end{ex}

We will need the following lemma (see \citep[Lemma 3.2.4]{LEA2}):
\begin{lem}\label{lemCtildeC}
For all $(i,j)\in I^2$, 
\begin{align*}
\tilde{C}_{ij}(m-1) + \tilde{C}_{ij}(m+1) - \sum_{k \sim j}\tilde{C}_{ik}(m) & = 0, \quad \forall m \geq 1, \\
\tilde{C}_{ij}(1) & = \delta_{i,j}.
\end{align*}
\end{lem}

Let us extend the functions $ \tilde{C}_{ij}$ to symmetrical functions on $\Z$
\begin{equation}
{\bf C}_{i,j}(m) := \tilde{C}_{ij}(m) + \tilde{C}_{ij}(-m)  \quad (m\in\Z),
\end{equation}
with the usual convention $\tilde{C}_{ij}(m)=0$ if $m\leq 0$.

Then Lemma \ref{lemCtildeC} translates as:

\begin{equation}\label{Clem}
{\bf C}_{ij}(m-1) + {\bf C}_{ij}(m+1) - \sum_{k \sim j}{\bf C}_{ik}(m) = \left\lbrace \begin{array}{ll}
2\delta_{i,j} & \text{ if } m=0,\\
0 & \text{ otherwise.}
\end{array} \right.
\end{equation}

		\subsection{Height function}\label{sectheightfun}
		
As $\g$ is simply-laced, its Dynkin diagram $\gamma$ is a bipartite graph. There is partition $I=I_0 \sqcup I_1$ such that every edge in $\gamma$ connects a vertex of $I_0$ to a vertex of $I_1$. 

\begin{defi}
Define, for all $i\in I$,
\begin{equation}
\xi_i = \left\lbrace \begin{array}{ll}
0 & \text{ if } i \in I_0\\
1 & \text{ if } i \in I_1
\end{array}\right.
\end{equation}
The map $\xi : I\to \{0,1\}$ is called a \emph{height function} on $\gamma$.
\end{defi}

\begin{rem}
In more generality, every function $\xi : I \to \Z$ satisfying 
\begin{equation*}
\xi(j) = \xi(i) \pm 1, \text{ when } j\sim i
\end{equation*}
is a height function on $\gamma$. It defines an orientation of the Dynkin diagram $\gamma$: 
\begin{equation*}
i\to j \text{ if } \xi(j) = i+1.
\end{equation*}
Our particular choice of height function defines a \emph{sink-source} orientation.
\end{rem}

\begin{ex}
If $\g$ if of type $D_5$, then $\gamma$ is 
\begin{minipage}[c]{12em}
\begin{center}
\begin{tikzpicture}[scale=0.8]
\node[circle,draw=black, fill=white] (0) at (0,0) {};
\node[circle,draw=black, fill=white] (1) at (1,0) {};
\node[circle,draw=black, fill=white] (2) at (2,0) {};
\node[circle,draw=black, fill=white] (3) at (3,0) {};
\node[circle,draw=black, fill=white] (4) at (2,0.8) {};
\draw (0) -- (1);
\draw (1) -- (2);
\draw (2) -- (3);
\draw (2) -- (4);
\end{tikzpicture}
\end{center}
\end{minipage}

and if we fix $\xi_1 = 0$, then
\begin{align*}
\xi_1 = 1, \quad \xi_3 & = 1,\\
\xi_2 = 0, \quad \xi_4 & = 0.
\end{align*}
\end{ex}

From now on, we fix such a height function $\xi$.

We will also use the notation:
\begin{equation}\label{epsilon}
\epsilon_i : =(-1)^{\xi_i} \quad \in \{ \pm 1\} \quad (i\in I).
\end{equation}

		\subsection{Semi-infinite quiver}\label{sectinftyquiver}

Let us define an infinite quiver $\Gamma$ as in \citep{ACAA}. 

First, let
\begin{equation}
\hat{I}:= \bigcup_{i\in I}\left(i, 2\Z+\xi_i\right).
\end{equation}

Let $\Gamma$ be the quiver with vertex set $\I$ and arrows
\begin{equation}
\left( (i,r) \rightarrow (j,s) \right) \Longleftrightarrow \left( C_{i,j} \neq 0 \text{ and } s= r + C_{i,j} \right).
\end{equation}

\begin{ex}\label{exsl4}
For $\g=\mathfrak{sl}_{4}$, one choice of $\I$ is
\begin{equation*}
\hat{I} =  (1,2\Z) \cup  (2,2\Z+1) \cup (3,2\Z),
\end{equation*}
and $\Gamma$ is the following:
\[
\xymatrix@R=0.5em@C=1.5em{
	& \vdots & \\
\vdots & (2,1) \ar[dl]\ar[dr]\ar[u] & \vdots \\
(1,0)\ar[dr]\ar[u] & & (3,0) \ar[dl]\ar[u] \\
& (2,-1) \ar[dl]\ar[dr]\ar[uu] & \\
(1,-2)\ar[dr]\ar[uu] & & (3,-2)\ar[dl]\ar[uu] \\
& (2,-3) \ar[dl]\ar[dr]\ar[uu] & \\
(1,-4)\ar[uu] &  \vdots \ar[u] & (3,-4)\ar[uu] \\
\vdots \ar[u] &   	& \vdots \ar[u]
}
\]
\end{ex}
	
\begin{defi}
Let
\begin{equation*}
\hat{I}^- := \hat{I}\cap \left( I\times \Z_{\leq 0}\right).
\end{equation*}
And define $G^-$ to be the semi-infinite subquiver of $\Gamma$ of vertex set $\mI$.

Let 
\begin{equation}\label{defJ}
\hat{J} : = \left( I\times \Z \right) \setminus \I.
\end{equation}
\end{defi}

\begin{ex}\label{exsl42}
Following Example \ref{exsl4}, for $\g=\mathfrak{sl}_{4}$, with the same choice of height function, one has
\begin{equation*}
\mI =  (1,2\Z_{\leq 0}) \cup  (2,2\Z_{\leq 0}-1) \cup (3,2\Z_{\leq 0}),
\end{equation*}
and $G^-$ is the following:
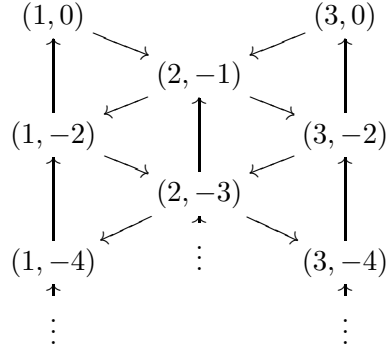
\[
\xymatrix@R=0.5em@C=1.5em{
 (1,0)\ar[dr] & & (3,0) \ar[dl] \\
& (2,-1) \ar[dl]\ar[dr] & \\
(1,-2)\ar[dr]\ar[uu] & & (3,-2)\ar[dl]\ar[uu] \\
& (2,-3) \ar[dl]\ar[dr]\ar[uu] & \\
(1,-4)\ar[uu] &  \vdots \ar[u] & (3,-4)\ar[uu] \\
\vdots \ar[u] &   	& \vdots \ar[u]
}
\]
\end{ex}

Finally, we recall a useful notation from \citep{ACAA}. For $(i,r)\in\mI$, define
\begin{equation}\label{kir}
k_{i,r}:= \frac{-r+\xi_i}{2}.
\end{equation}

The vertex $(i,r)$ is the $k_{i,r}$th vertex in its column in $G^-$, starting at the top.

	\section{Finite-dimensional representations of quantum affine algebras}\label{sectFDRQAA}

In this section, we recall the notations and different results regarding quantum affine algebras and finite-dimensional representations of quantum affine algebras.

			\subsection{Quantum affine algebra}
		
Let $\hat{\g}$ be the untwisted affine Lie algebra corresponding to $\g$.

Fix an nonzero complex number $q$, which is not a root of unity, and $h\in\mathbb{C}$ such that $q=e^{h}$. Then for all $r\in \mathbb{Q}, q^{r}:=e^{rh}$. Since $q$ is not a root of unity, for $r,s\in \mathbb{Q}$, we have $q^{r}=q^{s}$ if and only if $r=s$.

Let $\Uqhatg$ be the \emph{quantum enveloping algebra} of the Lie algebra $\hat{\g}$ (see \citep{GQG}), it is a $\mathbb{C}$-Hopf algebra.

	\subsection{Finite-dimensional representations}

Let $\C$ be the category of all (type 1) finite-dimensional $\Uqhatg$-modules. As $\Uqhatg$ is a Hopf algebra, $\C$ is a tensor category. The simple modules in $\C$ have been classified by Chari-Pressley (\citep{GQG}), in terms of Drinfeld polynomials. 

The simple finite-dimensional $\Uqhatg$-modules are indexed by the monomials in the infinite set of variables $(Y_{i,a})_{i\in I, a \in \mathbb{C}^\times}$, called \emph{dominant monomials} (\citep{QCRQAA}). For such a monomial $m$, let $L(m)$ denote the corresponding simple $\Uqhatg$-module.

We define the following sets of dominant monomials:
\begin{equation*}
\mathcal{M}= \left\lbrace \prod_{\text{finite}}Y_{i,q^r}^{n_{i,r}} \mid (i,r)\in\I , n_{i,r} \in \Z_{\geq 0}, n_{i,r}=0 \text{ except for a finite number of } (i,r) \right\rbrace,
\end{equation*}
\begin{equation*}
\mathcal{M}^-= \left\lbrace \prod_{\text{finite}}Y_{i,q^r}^{n_{i,r}} \mid (i,r)\in\mI ,n_{i,r} \in \Z_{\geq 0}, n_{i,r}=0 \text{ except for a finite number of } (i,r) \right\rbrace.
\end{equation*}

\begin{defi}
Let $\C_\Z$ be the full subcategory of $\C$ of objects whose composition factors are of the form $L(m)$, with $m\in\M$.

Let $\C_\Z^-$ be the full subcategory of $\C$ of objects whose composition factors are of the form $L(m)$, with $m\in\Mm$.
\end{defi}

The category $\CZ$ is a tensor category and $\CZm$ is a monoidal category (\citep[5.2.4]{CAQAA} and \citep[Proposition 3.10]{ACAA}).

\begin{rem}
Every simple object in $\C$ can be written as a tensor product of simple objects which are essentially in $\C_\Z$ (see \citep[Section 3.7]{CAQAA}). Thus, the description of the simple objects of $\C$ reduces to the description of the simple objects of $\C_\Z$.
\end{rem}

Let us introduce some particular irreducible finite-dimensional representations.
\begin{defi}
For all $(i,r)\in\I$, $V_{i,r}:=L(Y_{i,r})$ is called a \emph{fundamental module}.

For all $(i,r)\in\I, k\in \Z_{>0}$, let
\begin{equation}
m_{k,r}^{(i)}:=Y_{i,r}Y_{i,r+2}\cdots Y_{i,r+2k-2},
\end{equation}
the corresponding irreducible module $L(m_{k,r}^{(i)})$ is called a \emph{Kirillov-Reshetikhin module}, or \emph{KR-module}, and denote by
\begin{equation}
W_{k,r}^{(i)} := L(m_{k,r}^{(i)}).
\end{equation}
\end{defi}

Note that fundamental module are particular KR-modules, for $k=1, m_{1,r}^{(i)}= Y_{i,r}$.


		\subsection{$q$-characters and truncated $q$-characters}\label{sectqcharac}
		
Frenkel and Reshetikhin introduced in \citep{QCRQAA} an injective ring morphism, called the \emph{$q$-character} morphism, on the Grothendieck ring $K_0(\C)$ of the category $\C$:
\begin{equation}
\chi_q : K_0(\C)  \to  \Z\left[Y_{i,a}^{\pm 1} \mid i\in I,a\in \mathbb{C}^\times \right].
\end{equation}
Moreover, the $q$-character $\chi_q(V)$ of a $\Uqhatg$-module $V$ gives information about the decomposition into Jordan subspaces for the action of a large commutative subalgebra of $\Uqhatg$.

Here, as we restrict ourselves to the study of the category $\CZ$, the $q$-character will only involve variables $Y_{i,q^r}^{\pm 1}$, for $(i,r)\in\I$. Hence, for simplicity of notation we denote them by: $Y_{i,r}:=Y_{i,q^r}$. The $q$-character we are interested in is the injective ring morphism:
\begin{equation}
\chi_q : K_0(\CZ)  \to  \mathcal{Y}:=\Z\left[Y_{i,r}^{\pm 1} \mid (i,r)\in \I\right].
\end{equation}
We use the usual notation \citep{QCRQAA},
\begin{equation}\label{A}
A_{i,r} := Y_{i,r-1}Y_{i,r+1}\left( \prod_{j\sim i} Y_{j,r,}\right)^{-1}.
\end{equation}
for all $(i,r)\in \hat{J}$ (see (\ref{defJ})).
Note that $A_{i,r}$ is a Laurent monomial in the variables $Y_{j,s}$, with $(j,s)\in\I$. The monomials $A_{i,r}$ are analogs of the simple roots.

\begin{prop}\label{prop1+}\citep[Theorem 4.1]{CqC}
For $m$ a dominant monomial in $\M$, the $q$-character of the finite-dimensional irreducible representation $L(m)$ is of the form
\begin{equation}
\chi_q(L(m)) = m \left( 1 + \sum_p M_p \right),
\end{equation}
where $M_p$ is a monomial in the variables $A_{i,r}^{-1}$, with $(i,r)\in\hat{J}$.
\end{prop}

Let us recall Nakajima's partial order on monomials. For $m$ and $m'$ Laurent monomials in $\Y$, 
\begin{equation}\label{orderN}
m\leq m' \Longleftrightarrow m'm^{-1} \text{ is a product of } A_{i,r}, \text{ with } (i,r)\in\I.
\end{equation}

\begin{rem}
Note that Proposition~\ref{prop1+} can be translated as follows: for all dominant monomials $m$, the monomials occurring in the $q$-character of the finite-dimensional irreducible representation $L(m)$ are lower than $m$, for Nakajima's partial order.
\end{rem}

We also recall the \emph{truncated $q$-characters} from \citep{CAQAA}. For $m$ a monomial in $\Mm$, the $q$-character $\chi_q(L(m))$ may contain Laurent monomials in which variables $Y_{i,r}$, with $(i,r)\in\I\setminus \mI$ occur. Let $\chi_q^-(L(m))$ be the Laurent polynomial obtained from $\chi_q(L(m))$ by removing any such Laurent monomial. By definition
\begin{equation}
\qm(L(m)) \in \Z\left[ Y_{i,r}^{\pm 1} \mid (i,r)\in \mI
 \right].
\end{equation}

\begin{ex}
For $\g$ of type $A_2$, one has
\begin{align*}
\qm(L(Y_{1,0})) & = Y_{1,0},\\
\qm(L(Y_{1,-2})) & = Y_{1,-2} + Y_{1,0}^{-1}Y_{2,-1},\\
\qm(L(Y_{1,-4})) & = Y_{1,-4} + Y_{1,-2}^{-1}Y_{2,-3} + Y_{2,-1}^{-1} = \chi_q(L(Y_{1,-4})).
\end{align*}
\end{ex}

\begin{prop}\citep[Proposition 3.10]{ACAA}
The assignment $[L(m)]\mapsto \qm(L(m))$ extends to an injective ring homomorphism 
\begin{equation}\label{chim}
K_0(\CZm) \to \Z\left[ Y_{i,r}^{\pm 1} \mid (i,r)\in \mI \right].
\end{equation} 
\end{prop}
As such, all simple modules in $\CZm$ are identified with their isoclasses through the truncated $q$-character morphism.

		\subsection{Cluster algebra structure}\label{sectCAs}
		
One of the main ingredient we want to use in this work is the \emph{cluster algebra} structure of the Grothendieck ring of the category $\CZm$. 

Consider the cluster algebra $\mathcal{A}:=\mathcal{A}(\pmb u, \Gm)$, with initial seed $(\pmb u, \Gm)$, where 
\begin{itemize}
	\item[•] $\pmb u$ are initial cluster variables indexed by $\mI$, $\pmb u = \left\lbrace u_{i,r} \mid (i,r) \in\mI \right\rbrace$,
	\item[•] $\Gm$ is the semi-infinite quiver with vertex set $\mI$ defined in the previous section.
\end{itemize}
Consider the identification, for all $(i,r)\in\mI$,
\begin{equation}\label{uY}
u_{i,r} \equiv \prod_{\substack{k\geq 0 \\ r+2k\leq 0}}Y_{i,r+2k}.
\end{equation}
This identification makes sense, as these monomials are algebraically independent.

From the \emph{Laurent phenomenon}, we know that all cluster variables of $\A$ are Laurent polynomials in the variables $u_{i,r}$. Thus, via the identification (\ref{uY}), $\mathcal{A}$ is seen as a subring of $\Z\left[ Y_{i,r}^{\pm 1} \mid (i,r)\in \mI \right]$. 

\begin{theo}\citep[Theorem 5.1]{ACAA}\label{theoclusterstr}
The injective ring homomorphism $\qm$ is an isomorphism between the Grothendieck ring of the category $\CZm$ and the cluster algebra $\A$, after identification (\ref{uY}):
\begin{equation}
\qm : K_0(\CZm) \xrightarrow{\sim} \A.
\end{equation}
\end{theo}

Moreover, truncated $q$-characters of Kirillov-Reshetikhin modules can be obtained as cluster variables, via the identification of initial seed (\ref{uY}), it is the main result of \citep{ACAA}.

\newpage

	\section{Quantum Grothendieck rings}\label{sectqGRs}
	
We will recall in this section the definition of the quantum Grothendieck ring of the category $\CZ$, introduce that of the category $\CZm$, and study those rings. 

Let $t$ be an indeterminate. The quantum Grothendieck rings of the categories $\CZ$ and $\CZm$ are non-commutative $t$-deformations of the Grothendieck rings.  
		
		\subsection{Quantum torus}	
	
Let $\bf Y_t$ be the $\Z[t^{\pm 1}]$-algebra generated by the variables $Y_{i,r}^{\pm 1}$, for $(i,r)\in \I$, and the $t$-commutations relations:
\begin{equation}\label{relY}
Y_{i,r}\ast Y_{j,s} = t^{\mathcal{N}_{i,j}(r-s)}Y_{j,s}\ast Y_{i,r},
\end{equation}
where $\mathcal{N}_{i,j} : \Z \to \Z$ is the antisymmetrical map:
\begin{equation}\label{eqCN}
\mathcal{N}_{i,j}(m)={\bf C}_{i,j}(m+1) - {\bf C}_{i,j}(m-1), \quad \forall m\in \Z,
\end{equation}
using the notations from Section \ref{sectQCM}. . 

\begin{rem}\label{remtores}
Here we work with the quantum torus of \citep{AAqtC} and \citep{QGR}, which is slightly different from the original quantum torus used to define the quantum Grothendieck ring in \citep{QVtA} and \citep{PSqGR}.
\end{rem}

\begin{ex}\label{exsl2-2}
If we continue Example \ref{exsl2-1}, for $\mathfrak{g}=\mathfrak{sl}_2$, in this case, $\I=(1,2\Z)$, for $r\in\Z$, one has
\begin{equation}\label{relsl2Y}
Y_{1,2r}\ast Y_{1,2s} = t^{2(-1)^{s-r}}Y_{1,2s}\ast Y_{1,2r}, \quad \forall s>r>0.
\end{equation}
\end{ex}	
	
The $\mathbb{Z}(t)$-algebra $\bf Y_t$ is viewed as a quantum torus of infinite rank.

We extend this quantum torus by adjoining a fixed square root $t^{1/2}$ of $t$:
\begin{equation}
\Y_t:=\Z[t^{1/2}]\otimes_{\Z[t]}\bf Y_t.
\end{equation}

Let $\Y_t^-$ be the quantum torus defined exactly the same way, except by only taking as generators the $Y_{i,r}^{\pm 1}$, for $(i,r)\in\mI$.

Let us denote by
\begin{equation}\label{pi}
\pi: \Y_t \to \Y_t^-,
\end{equation}
the projection of $\Y_t$ onto $\Y_t^-$,
\begin{equation}
\pi(Y_{i,r}) = 0 \text{ if } (i,r) \in \I\setminus\mI.
\end{equation}

\begin{rem}\label{remsfield}
Even if $\Y_t^-$ is an infinite rank quantum torus, it can be seen as a limit of finite rank quantum tori. As finite rank quantum tori are of polynomial growth, they are Ore domains (see \citep{RSP}). Moreover, the Ore condition being local (any pair of elements of $\Y_t^-$ belongs to some sufficiently larger finite rank quantum torus), $\Y_t^-$ is an Ore domain. Hence we can consider its skew field of fractions $\mathcal{F}_t$.
\end{rem}

		\subsection{Commutative monomials}\label{sectcomm}

For a family of integers with finitely many non-zero components $(u_{i,r})_{(i,r)\in \hat{I}}$, define the \emph{commutative monomial} $\prod_{(i,r)\in \hat{I}} Y_{i,r}^{u_{i,r}}$ as 
\begin{equation}
\prod_{(i,r)\in \hat{I}} Y_{i,r}^{u_{i,r}} :=t^{\frac{1}{2}\sum_{(i,r)<(j,s)}u_{i,r}u_{j,s}\mathcal{N}_{i,j}(r,s)} \overrightarrow{\bigast}_{(i,r)\in \hat{I}} Y_{i,r}^{u_{i,r}},
\end{equation}
where on the right-hand side an order on $\hat{I}$ is chosen so as to give meaning to the sum, and the product $\ast$ is ordered by it (notice that the result does not depend on the order chosen).

The commutative monomials form a basis of the $\Z[t^{1/2}]$-vector space $\mathcal{Y}_t$. 

The non-commutative product of two commutative monomials $m_1$ and $m_2$ in $\Yt$ is given by:
\begin{equation}
m_1 \ast m_2 = t^{D(m_1,m_2)}m_2 \ast m_1 = t^{\frac{1}{2}D(m_1,m_2)}m_1m_2,
\end{equation}
where $m_1m_2$ denotes the commutative product of the monomials, and
\begin{equation}
D(m_1,m_2)=\sum_{(i,r),(j,s)\in \hat{I}}u_{i,r}(m_1)u_{j,s}(m_2)\mathcal{N}_{i,j}(r,s).
\end{equation}

		\subsection{Quantum Grothendieck ring $K_t(\CZ)$}\label{sectqGR}
	
We define the quantum Grothendieck ring $K_t(\CZ)$ of the category $\CZ$ as in \citep[Section 5.4]{QGR} (see Remark~\ref{remtores} for original references).

For all $(i,r)\in\hat{J}$, let $A_{i,r}$ denote the commutative monomial in $\Y_t$ defined as in (\ref{A}):
\begin{equation*}
A_{i,r} := Y_{i,r-1}Y_{i,r+1}\left( \prod_{j\sim i} Y_{j,r}\right)^{-1}.
\end{equation*}

For all $i\in I$, define $K_{i,t}$ the subring of $\Y_t$ generated by  the
\begin{equation}
Y_{i,r}\left(1+A_{i,r+1}^{-1}\right), \quad Y_{j,s}^{\pm 1} \quad \left( (i,r),(j,s)\in \I, j\neq i\right).
\end{equation}

In \citep{tAOE}, the $K_{i,t}$ are defined as kernels of $t$-deformed screening operators, motivated by the results in \citep{QCRQAA}. Let us detail this, as it will be important in the proof of the main result. For all $i\in I$, define the free $\mathcal{Y}_t$-modules
\begin{equation*}
\mathcal{Y}_{t,i}^{l} := \sum_{r\in\Z\mid (i,r)\in\I}\mathcal{Y}_t\cdot S_{i,r}.
\end{equation*}
Then let $\mathcal{Y}_{t,i}$ be the quotient of $\mathcal{Y}_{t,i}^{l}$ by the left-$\mathcal{Y}_t$-module generated by the elements
\begin{equation*}
Q_{i,r}:= A_{i,r+1}^{-1}S_{i,r+2} - t^2S_{i,r}, \quad \forall (i,r)\in\I.
\end{equation*}
\begin{lem}\label{lemYit}
For all $i\in I$, the module $\mathcal{Y}_{t,i}$ is free.
\end{lem}

\begin{proof}
The elements $Q_{i,r}$ are linearly independent and for all $r_0$ such that $(i,r_0)\in \I$ fixed, the set $\left\lbrace Q_{i,r}, S_{i,r_0} \mid (i,r) \in I \right\rbrace $ forms a basis of $\mathcal{Y}_{t,i}^{l}$.

Hence $\mathcal{Y}_{t,i}$ is a quotient of a free module by a submodule generated by elements of a basis, thus it is free.
\end{proof}

From \citep{tAOE}, for all $i\in I$, there exists a $\Z[t^{\pm 1/2}]$-linear map 
\begin{equation}
S_{i,t} : \mathcal{Y}_t \to \mathcal{Y}_{i,t},
\end{equation}
which is a derivation and such that
\begin{equation}
K_{i,t} = \ker(S_{i,t}).
\end{equation}

Finally, let
\begin{equation}
K_t(\CZ) := \bigcap_{i\in I}K_{i,t}.
\end{equation}

From \citep{tAqC} and \citep{AAqtC} we know that for all dominant monomials $m\in \M$, there exists a unique element $F_{t}(m) \in  K_t(\CZ)$ such that $m$ occurs in $F_{t}(m)$ with multiplicity 1 and no other dominant monomial occurs in $F_{t}(m)$. Thus, all elements of $K_t(\CZ)$ are characterized by the coefficients of their dominant monomials. 

The $F_{t}(m)$ linearly generate $K_t(\CZ)$. 

\begin{rem}\label{remKtgen}
For all $(i,r)\in\mI$,
\begin{equation}\label{FtL}
F_t(Y_{i,r}) = [L(Y_{i,r})]_t.
\end{equation}
The $[L(Y_{i,r})]_t$ generate $K_t(\CZm)$ algebraically.
\end{rem}

		\subsection{The $(q,t)$-characters}

For a dominant monomial $m\in\M$, write it as a commutative monomial in $\Y_t$:
\begin{equation}
m = \prod_{(i,r)\in\I}Y_{i,r}^{n_{i,r}(m)} \quad \in \Y_t.
\end{equation}

Define 
\begin{equation}
[M(m)]_t := t^{\alpha(m)} \overleftarrow{\bigast}_{r\in \Z}F\left( \prod_{i\in I} Y_{i,r}^{n_{i,r}(m)} \right)\quad \in K_t(\CZ),
\end{equation}
where $\alpha(m)\in\frac{1}{2}\Z$ is fixed such that $m$ occurs with multiplicity one in the expansion of $[M(m)]_t$ on the basis of the commutative monomials of $\Y_t$, and the product $\overleftarrow{\bigast}$ is taken with decreasing $r\in\Z$. 

In particular, from (\ref{FtL}), for all $(i,r)\in\I$,
\begin{equation}
[L(Y_{i,r})]_t = [M(Y_{i,r})]_t.
\end{equation}

One has, for all $m\in\M$,
\begin{equation}\label{MMt}
[M(m)]_t  \xrightarrow{t=1} \chi_q(M(m)).
\end{equation}
This result is a direct consequence of the definition of $[M(m)]_t$, as it is satisfied for the fundamental modules $[L(Y_{i,r})]_t=F_t(Y_{i,r})$. Thus $[M(m)]_t$ is called the \emph{$\qt$-character} of the standard module $M(m)$.

As in \citep{QVtA}, we consider the $\Z$-algebra anti-automorphism $\overline{\phantom{A}}$ of $\Y_t$ defined by:
\begin{equation}\label{bar}
\overline{t^{1/2}} = t^{-1/2}, \quad \overline{Y_{i,r}} = Y_{i,r}, \quad \left( (i,r) \in\I \right).
\end{equation}
This map is called the \emph{bar-involution}.

\begin{theo}\citep{QVtA}\label{theoN}
There exists a unique family $\left\lbrace [L(m)]_t \right\rbrace_{m\in\M}$ of elements of $K_t(\CZ)$ such that, for all $m\in\M$,
\begin{itemize}
	\item[•] $\overline{[L(m)]_t}=[L(m)]_t$,
	\item[•] $[L(m)]_t \in [M(m)]_t + \sum_{m'<m}t^{-1}\Z[t^{-1}][M(m')]_t$, where $m'<m$ for Nakajima's partial order (\ref{orderN}).
\end{itemize}
\end{theo}
The following Theorem extends (\ref{MMt}), but more importantly gives an algorithm, similar to the Kazhdan-Lusztig algorithm, to compute the $\qt$-characters (and thus the $q$-characters) of the simple modules.

\begin{theo}\citep[Corollary 3.6]{QVtA}\label{theoNpos}
The evaluation at $t=1$ of the $\qt$-characters recovers the $q$-characters. For all $m\in\M$,
\begin{align*}
[L(m)]_t & \xrightarrow{t=1} \chi_q(L(m)) \quad \in \Y.
\end{align*}
Moreover, the coefficients of the expansion of $[L(m)]_t$ as a linear combination of Laurent monomials in the variables $(Y_{i,r})_{(i,r)\in\I}$ belong to $\N[t^{\pm 1}]$.
\end{theo}
Note that the positivity result of this Theorem has only been proven for ADE types as yet.

		\subsection{Truncated $(q,t)$-characters and quantum Grothendieck ring $K_t(\CZm)$}
		
As in Section~\ref{sectqcharac}, one can define truncated versions of the $(q,t)$-characters. 

For all dominant monomials $m$ in $\Mm$, let $[L(m)]_t^-$ be the Laurent polynomial obtained from $[L(m)]_t$ by removing any term in which a variable $Y_{i,r}$, with $(i,r)\in\I\setminus\mI$ occurs:
\begin{equation}
[L(m)]_t^- = \pi\left([L(m)]_t\right)\quad \in \Y_t^-,
\end{equation}
where $\pi$ is the projection defined in (\ref{pi}).

Define $K_t(\CZm)$ as the $\Z[t^{1/2}]$-submodule of $\Y_t^-$ generated by the truncated $\qt$-characters $[L(m)]_t^-$ of the simple finite-dimensional modules $L(m)$ in the category $\CZm$.

\begin{lem}\label{lemCZm}
The quantum Grothendieck ring $K_t(\CZm)$ is actually a subalgebra of $\Y_t^-$. Moreover, it is algebraically generated by the truncated $\qt$-characters of the fundamental modules:
\begin{equation}
K_t(\CZm) = \left\langle [L(Y_{i,r})]_t^- \mid (i,r) \in \mI \right\rangle.
\end{equation}
\end{lem}

\begin{proof}
For every dominant monomials $m_1,m_2 \in\M$, one can write:
\begin{equation}\label{Lm1Lm2}
[L(m_1)]_t\ast [L(m_2)]_t = \sum_{m\in\M}c_{m1,m2}^m(t^{1/2})[L(m)]_t.
\end{equation}
Hence the image of (\ref{Lm1Lm2}) by the projection $\pi$ of (\ref{pi}) is:
\begin{equation*}
[L(m_1)]_t^-\ast [L(m_2)]_t^- = \sum_{m\in\M}c_{m1,m2}^m(t^{1/2})[L(m)]_t^-.
\end{equation*}
Thus $K_t(\CZm)$ is stable by products.

By definition the truncated $\qt$-characters of the fundamental modules $L(Y_{i,r})$, for $(i,r)\in \mI$ belong to $K_t(\CZm)$.

Reciprocally, the $\qt$-characters of the fundamental modules $L(Y_{i,r})$, for all $(i,r)\in \I$, algebraically generate the quantum Grothendieck ring $K_t(\CZ)$ (see remark~\ref{remKtgen}). Hence the truncated $\qt$-characters $[L(Y_{i,r})]_t^-$, for all $(i,r)\in \I$, algebraically generate $K_t(\CZm)$. 

From Proposition~\ref{prop1+} and Theorem~\ref{theoN}, for all dominant monomials $m\in\M$, the $\qt$-character of the simple representation $L(m)$ is of the form
\begin{equation*}
[L(m)]_t = m\left( 1 +  \sum_p M_p \right),
\end{equation*}
where $M_p$ is a monomial in the variables $(A_{i,r}^{-1})_{i,r)\in\hat{J}}$, with coefficients in $\Z[t^{\pm 1}]$. Thus, 
\begin{equation*}
\pi\left([L(Y_{i,r})]_t\right) = 0, \text{ if } (i,r) \in \I\setminus \mI.
\end{equation*}
Hence, $K_t(\CZm)$ is algebraically generated by the $[L(Y_{i,r})]_t^-$, with $(i,r)\in \mI$.
\end{proof}

$K_t(\CZm)$ is a $t$-deformed version of the Grothendieck ring of the category $\CZm$, in the sense that the evaluation $[L(m)]_t^-  \xrightarrow{t=1} \chi_q^-(L(m))$ extends to a ring homomorphism
\begin{equation}
K_t(\CZm)  \xrightarrow{t=1} K_0(\CZm),
\end{equation}
where $ K_0(\CZm)$ is identified with its image under the truncated $q$-character (\ref{chim}), which is an injective map.

		\subsection{Quantum $T$-systems}\label{sectTsys}

The $\qt$-characters of the Kirillov-Reshetikhin modules satisfy some algebraic relations called \emph{quantum $T$-systems}. Those are $t$-deformed versions of the $T$-system relations, which are satisfied by the $q$-characters of the KR-modules \citep{FRSLM,QVtA,KRC}.

\begin{prop}\citep{tAqCKR}\citep[Proposition 5.6]{QGR}
For all $(i,r)\in\I$ and $k\in\Z_{>0}$, the following relation holds in $K_t(\CZ)$:
\begin{equation}\label{Tsys}
[W_{k,r}^{(i)}]_t \ast [W_{k,r+2}^{(i)}]_t = t^{\alpha(i,k)} [W_{k-1,r+2}^{(i)}]_t \ast [W_{k+1,r}^{(i)}]_t + t^{\gamma(i,k)} \bigast_{j\sim i} [W_{k,r+1}^{(j)}]_t,
\end{equation}
where 
\begin{equation}\label{ag}
\alpha(i,k) = -1 + \frac{1}{2}\left( \tilde{C}_{ii}(2k-1) + \tilde{C}_{ii}(2k+1)\right), \quad \gamma(i,k) = \alpha(i,k) +1.
\end{equation}
\end{prop}

\begin{rem}\label{remtcomTsys}
First of all, one notices that the dominant monomials of $W_{k-1,r+2}^{(i)}$ and $W_{k+1,r}^{(i)}$ commute:
\begin{equation}
m_{k-1,r+2}^{(i)}\ast m_{k+1,r}^{(i)} =  m_{k+1,r}^{(i)} \ast m_{k-1,r+2}^{(i)}.
\end{equation}
Moreover, the tensor product of the KR-modules $W_{k-1,r+2}^{(i)}\otimes W_{k+1,r}^{(i)}$ is irreducible (this result is proved in \citep{BGAT} and also by explicit computation of its $\qt$-character in \citep{tAqCKR}). Thus their respective $\qt$-characters $t$-commute (see \citep[Corollary 5.5]{QGR}). As their dominant monomials commute, these $\qt$-characters in fact commute and their product can be written as a commutative product, as in Section~\ref{sectcomm}.

By the same arguments, for $j\sim i$, the $\qt$-characters $[W_{k,r+1}^{(j)}]_t$ commute so the order of the factors in $\bigast_{j\sim i}$ in (\ref{Tsys}) does not matter.
\end{rem}

By taking the image of (\ref{Tsys}) through the projection $\pi$ of (\ref{pi}), one obtains the following relation in $K_t(\CZm)$.
For all $(i,r)\in\I$ and $k\in\Z_{>0}$,
\begin{equation}\label{Tsysm}
[W_{k,r}^{(i)}]_t^- \ast [W_{k,r+2}^{(i)}]_t^- = t^{\alpha(i,k)} [W_{k-1,r+2}^{(i)}]_t^- \cdot [W_{k+1,r}^{(i)}]_t^- + t^{\gamma(i,k)} \prod_{j\sim i} [W_{k,r+1}^{(j)}]_t^-,
\end{equation}
where $\alpha(i,k)$ and $\gamma(i,k)$ are defined in (\ref{ag}). 

Note that in (\ref{Tsysm}), the products appearing on the left-hand side are commutative products, which are well-defined from Remark~\ref{remtcomTsys}. Hence the change of notations since (\ref{Tsys}).

	\section{Quantum cluster algebra structure}\label{sectqCAstructure}
We define in this section the quantum cluster algebra structure built within the quantum torus $\Y_t^-$.

		\subsection{A compatible pair}

For all $(i,r)\in\mI$, the variable $u_{i,r}$, written as in (\ref{uY}), can be seen as commutative monomial in $\Y_t^-$. Define
\begin{equation*}
U_{i,r} := \prod_{\substack{k\geq 0 \\ r+2k\leq 0}}Y_{i,r+2k}\quad \in\Y_t^-.
\end{equation*}
They satisfy the following $t$-commutation relations. For all $((i,r),(j,s))\in(\mI)^2$,
\begin{equation}\label{tcommU}
U_{i,r}\ast U_{j,s} = t^{L\left((i,r),(j,s)\right)}U_{j,s}\ast U_{i,r},
\end{equation}
where
\begin{equation}
L\left((i,r),(j,s)\right) = \sum_{\substack{k\geq 0\\ r+2k\leq 0}} \sum_{\substack{l\geq 0\\ s+2l\leq 0}}\mathcal{N}_{ij}(s+2l-r-2k).
\end{equation}

Let $B_-$ be the $\mI\times\mI$-matrix encoding the quiver $G^-$, for all $\left((i,r),(j,s)\right)\in(\mI)^2$:
\begin{equation}
B_-\left((i,r),(j,s)\right)= |\{ \text{ arrows } (i,r) \to (j,s) \text{ in } G^- \}| - |\{ \text{ arrows } (j,s) \to (i,r) \text{ in } G^- \}|.
\end{equation}
Let $L$ be the $\mI\times\mI$ skew-symmetric matrix
\begin{equation}
L := \left(L\left((i,r),(j,s)\right)\right)_{((i,r),(j,s))\in(\mI)^2}.
\end{equation}

The pair of $\mI\times\mI$-matrices $(L,B_-)$ forms a \emph{compatible pair}, in the sense of quantum cluster algebras.

More precisely, we prove the following.

\begin{prop}\label{propcompCZm}
For all $\left((i,r),(j,s)\right)\in(\mI)^2$,
\begin{equation}\label{Bm}
\left(B_-^T L\right)\left((i,r),(j,s)\right) = \left\lbrace\begin{array}{ll}
-2 & \text{ if } (i,r) = (j,s) \\
0 & \text{ otherwise.}
\end{array}\right.
\end{equation}
\end{prop}

\begin{rem}
In \citep{qCA}, by definition a pair of $J\times J$-matrices $(\Lambda,B)$ forms a compatible pair if $^TBL$ is a diagonal matrix with positive integer coefficients. But as explained in \citep{LEA2}, quantum cluster algebras can be built exactly the same way given as data a pair $(\Lambda,B)$ such that $^TBL$ is a diagonal matrix with integer coefficients with constant signs.
\end{rem}

\begin{proof}
Fix $\left((i,r),(j,s)\right)\in(\mI)^2$, there are different cases to consider.
\begin{itemize}
	\item[•] If $r\leq -2$, one has:
\begin{align*}
\left(B_-^T L\right)\left((i,r),(j,s)\right) & = L\left((i,r-2),(j,s)\right) - L\left((i,r+2),(j,s)\right) \\
& + \sum_{k\sim i}\left(L\left((k,r+1),(j,s)\right) - L\left((k,r-1),(j,s)\right) \right).
\end{align*}
One has 
\begin{align*}
L\left((i,r-2),(j,s)\right) - L\left((i,r+2),(j,s)\right) & =  - {\bf C}_{ij}(s-r-1) - {\bf C}_{ij}(s-r+1) \\ 
& +{\bf C}_{ij}(-r + 3-\xi_j) + {\bf C}_{ij}(-r +1 -\xi_j),
\end{align*}
where $\xi: I \to \{ 0,1\}$ is the height function on the Dynkin diagram of $\g$ fixed in Section~\ref{sectheightfun}.

On the other hand, for all $k\sim i$, one has
\begin{align*}
L\left((k,r+1),(j,s)\right) - L\left((k,r-1),(j,s)\right) & = {\bf C}_{kj}(s-r) -  {\bf C}_{kj}(-r+2-\xi_j).
\end{align*}
Thus, with the reformulation (\ref{Clem}) of Lemma~\ref{lemCtildeC},
\begin{equation}\label{Bm1}
\left(B_-^T L\right)\left((i,r),(j,s)\right) = \left\lbrace \begin{array}{ll}
-2\delta_{i,j} & \text{ if } s=r,\\
0 & \text{ otherwise.}
\end{array}\right.
\end{equation}
	\item[•] If $r=-1$, one has:
\begin{align*}
\left(B_-^T L\right)\left((i,-1),(j,s)\right) & = L\left((i,-3),(j,s)\right) \\
& + \sum_{k\sim i}\left(L\left((k,0),(j,s)\right) - L\left((k,-2),(j,s)\right) \right).
\end{align*}
However,
\begin{align*}
L\left((i,-3),(j,s)\right) & = \sum_{\substack{l\geq 0\\ s+2l\leq 0}}\left( \mathcal{N}_{ij}(s+2l+3) + \mathcal{N}_{ij}(s+2l+1)\right),\\
& = {\bf C}_{ij}(4-\xi_j) + {\bf C}_{ij}(2-\xi_j) - {\bf C}_{ij}(s) - {\bf C}_{ij}(s+2).
\end{align*}
And, for all $k\sim i$,
\begin{align*}
L\left((k,0),(j,s)\right) - L\left((k,-2),(j,s)\right) & = {\bf C}_{ij}(s+1) - {\bf C}_{ij}(3-\xi_j).
\end{align*}
Thus, with relation (\ref{Clem}),
\begin{equation}\label{Bm2}
\left(B_-^T L\right)\left((i,-1),(j,s)\right) = \left\lbrace \begin{array}{ll}
-2\delta_{i,j} & \text{ if } s=-1,\\
0 & \text{ otherwise.}
\end{array}\right.
\end{equation}
	\item[•] If $r=0$, one has
	\begin{align*}
\left(B_-^T L\right)\left((i,0),(j,s)\right) & = L\left((i,-2),(j,s)\right)  -\sum_{k\sim i}  L\left((k,-1),(j,s)\right) .
\end{align*}
However,
\begin{align*}
L\left((i,-2),(j,s)\right)  = {\bf C}_{ij}(3-\xi_j) + {\bf C}_{ij}(1-\xi_j) - {\bf C}_{ij}(s+1) - {\bf C}_{ij}(s-1).
\end{align*}
And, for all $k\sim i$,
\begin{align*}
 L\left((k,-4),(j,s)\right)  = -{\bf C}_{ij}(s) + {\bf C}_{ij}(2-\xi_j).
\end{align*}
Thus, with relation (\ref{Clem}),
\begin{equation}\label{Bm3}
\left(B_-^T L\right)\left((i,0),(j,s)\right) = \left\lbrace \begin{array}{ll}
-2\delta_{i,j} & \text{ if } s=0,\\
0 & \text{ otherwise.}
\end{array}\right.
\end{equation}
\end{itemize}
The combination of the results (\ref{Bm1}),(\ref{Bm2}) and (\ref{Bm3}) gives the general expression (\ref{Bm}). 
\end{proof}

		\subsection{The quantum cluster algebra $\A_t$}

\begin{defi}
Let $T$ be the based quantum torus with generators $\{~u_{i,r}\mid (i,r)\in\mI\}$ satisfying the quasi-commutation relations (\ref{tcommU}):
\begin{equation*}
u_{i,r}\ast u_{j,s} = t^{L\left((i,r),(j,s)\right)}u_{j,s}\ast u_{i,r}.
\end{equation*}
Let $\F$ be the skew-field of fractions of $T$.
\end{defi}

As $(L,B_-)$ forms a compatible pair, it defines a quantum seed in $\F$. Let $\mathcal{S}$ be the mutation equivalence class of the quantum seed $(L,B_-)$.

\begin{defi}
Let $\A_t$ be the quantum cluster algebra defined by the quantum seed $\mathcal{S}$, as in \citep{qCA}.
\end{defi}

By definition, $\A_t$ is a $\Z[t^{\pm 1/2}]$-subalgebra of $\F$. However, by the quantum Laurent phenomenon, $\A_t$ is actually a $\Z[t^{\pm 1/2}]$-subalgebra of the quantum torus $T$.

\begin{lem}\label{lemTY}
The map
\begin{equation}\label{eta}
\eta : \begin{array}{rcl}
T &  \longrightarrow & \Y_t^-\\
u_{i,r} & \longmapsto & U_{i,r},
\end{array}
\end{equation}
where the variables $U_{i,r}$ are defined in (\ref{uY}), is an isomorphism of quantum tori.
\end{lem}

\begin{proof}
First of all, this map is well-defined because the variables $u_{i,r}$ $t$-commute exactly as the variables $U_{i,r}$, by definition of the matrix $L$ (\ref{tcommU}).

Secondly, this map is invertible, with inverse:
\begin{equation*}
\eta^{-1}:\begin{array}{rcl}
\Y_t^- & \longrightarrow & T\\
Y_{i,r} & \longmapsto  & \left\lbrace  \begin{array}{ll}
	u_{i,r}u_{i,r+2}^{-1} & \text{ if } r+2\leq 0,\\
	u_{i,r} & \text{otherwise,}
\end{array}\right.
\end{array}.
\end{equation*}
\end{proof}

With this lemma, we know that the quantum cluster algebra $\A_t$ belongs to the quantum torus $\Y_t^-$.
The following result is the main result of this paper, it extends Theorem~\ref{theoclusterstr} to the quantum setting.

\begin{theo}\label{theo1}
The image of the quantum cluster algebra $\A_t$ by the injective ring morphism $\eta$ is the quantum Grothendieck ring of the category $\CZm$,
\begin{equation}
\eta\restriction_{\A_t} : \A_t \xrightarrow{\sim} K_t(\CZm).
\end{equation}
Moreover, the truncated $\qt$-characters of the Kirillov-Reshetikhin modules which are in $\CZm$ are obtained as quantum cluster variables.
\end{theo}

The proof of this Theorem will be developed in the following section. It is mainly based on Proposition~\ref{propyirm}.

\begin{rem}\label{remQin}
In \citep[Theorem 8.4.3]{QIN}, for certain subcategories of $\C$ generated by a finite number of fundamental modules, Qin proved that there existed an isomorphism between the quantum Grothendieck ring of the category and a quantum cluster algebra, which identifies classes of Kirillov-Reshetikhin modules to cluster variables. It is our understanding that this identification coincides with the truncated $\qt$-characters in our work. Here, the isomorphism is given explicitly, and we obtain directly the truncated $\qt$-characters.
\end{rem}


	\section{Quantum cluster algebras and quantum Grothendieck ring}\label{sectqCAqGR}
	
We prove in this section that the quantum Grothendieck ring of the category $\CZm$ is isomorphic to the quantum cluster algebra we have just defined.

		\subsection{A note on the bar-involution}\label{sectbar}

The bar-involution $\overline{\phantom{A}}$, as defined in (\ref{bar}), is a $\Z$-algebra anti-automorphism of the quantum torus $\Y_t$. The commutative monomials are invariant under this involution.

On the other hand, the quantum cluster algebra $\A_t$ is also equipped with a $\Z$-linear bar-involution morphism $\overline{\phantom{A}}$ on its quantum torus $T$ (see \citep[Section 6]{qCA}), which satisfies
\begin{align*}
\overline{t^{1/2}} & = t^{-1/2}, \\
\overline{u_{i,r}} & = u_{i,r}.
\end{align*}

As noted in \citep[Section 7.1]{LEA2}, these definitions are compatible. In our case, they define exactly the same involution on $\Y_t^-$; the following diagram is commutative:
\begin{equation}
\xymatrix{
\Y_t^-  \ar[r]^{\overline{\phantom{A}}} \ar[d]_{\eta^{-1}} & \Y_t^- \\
T \ar[r]_{\overline{\phantom{A}}} & T\ar[u]^{\eta}
}.
\end{equation}

From \citep[Remark 6.4]{qCA}, all cluster variables are invariant under the bar-involution. Thus, the images of the quantum cluster variables in $\A_t$ are bar-invariant elements of $\Y_t^-$.

We will use the terminology "commutative products", as in Section~\ref{sectcomm} for bar-invariant elements of the quantum torus $T$.

		\subsection{A sequence of vertices}\label{sectsequ}

In \citep{ACAA} Hernandez and Leclerc exhibited a particular sequence of mutations in the cluster algebra $\mathcal{A}(\pmb u, \Gm)$ (see Section~\ref{sectCAs}) in order to obtain the truncated $q$-characters of all the KR-modules, up to a shift of spectral parameter.

The key idea we used was that at each step of this sequence, the exchange relation was a $T$-system equation.

We will recall this sequence of mutations and show that, if applied to the quantum cluster algebra $\A_t$, the quantum exchange relations at each step are in fact quantum $T$-systems relations, as in (\ref{sectTsys}). 

Recall the height function $\xi : I\to \{0,1\}$ fixed on the Dynkin diagram of $\g$ in Section~\ref{sectheightfun}.

First, fix an order on the columns of $G^-$:
\begin{equation}\label{ordercol}
i_1,i_2,\ldots,i_n,
\end{equation}
such that if $k\leq l$ then $\xi_{i_k} \leq \xi_{i_l}$ (select first the vertices $i$ such that $\xi_i=0$ then the others).

Then, the sequence $\mathscr{S}$ is defined by reading each column, from top to bottom, in this order.

\begin{ex}
We follow Examples \ref{exsl4} and \ref{exsl42}, $\g=\mathfrak{sl}_4$ and 
\begin{equation*}
\mI =  (1,2\Z_{\leq 0}) \cup  (2,2\Z_{\leq 0}-1) \cup (3,2\Z_{\leq 0}).
\end{equation*}
We fix the following order on the columns: $1, \smallskip 3, \smallskip 3$. Then the sequence $\mathscr{S}$ is
\begin{equation}
\mathscr{S} = (1,0),(1,-2),(1,-4),\ldots,(3,0),(3,-2),(3,-4),\ldots,(2,-1),(2,-3),(2,-5),\ldots .
\end{equation}
\end{ex}

		\subsection{Truncated $\qt$-characters as quantum cluster variables}\label{secttruncq}

As in \citep{ACAA}, let $\mu_\mathscr{S}$ be the sequence of quantum cluster mutations in $\A_t$ indexed by the sequence of vertices $\mathscr{S}$.

For all $m\geq 1$, let $u_{i,r}^{(m)}$ be the quantum cluster variable obtained at vertex $(i,r)$ after applying $m$ times the sequence of mutations $\mu_\mathscr{S}$ to the quantum cluster algebra $\A_t$ with initial seed $\{~u_{i,r}\mid (i,r)\in \mI\}$. By the quantum Laurent phenomenon, the $u_{i,r}^{(m)}$ belong to the quantum torus $T$. Let
\begin{equation}
w_{i,r}^{(m)} := \eta(u_{i,r}^{(m)}) \quad \in \Y_t^-,
\end{equation}
where $\eta: T\to \Y_t^-$ is the isomorphism defined in Lemma~\ref{lemTY}.

The following result is an extension of Theorem 3.1 from \citep{ACAA} to the quantum setting.

\begin{prop}\label{propyirm}
For all $(i,r)\in\mI$ and $m\geq 0$, 
\begin{equation}\label{eqprop}
w_{i,r}^{(m)}  = [W_{k_{i,r},r-2m}^{(i)}]_t^-,
\end{equation}
where $k_{i,r}$ is defined in (\ref{kir}).

In particular, if $2m\geq h$, this truncated $\qt$-character is equal to its $\qt$-character and
\begin{equation}\label{eqprop2}
w_{i,r}^{(m)}  = [W_{k_{i,r},r-2m}^{(i)}]_t.
\end{equation}
\end{prop}

\begin{rem}
The sequence of vertices $\mathscr{S}$ is infinite. However, in order to compute one fixed truncated $\qt$-character, one only has to compute a finite number of mutations in the infinite sequence $\mu_\mathscr{S}$. In \citep[Section 7.2]{LEA2}, the exact finite sequence needed to compute the $\qt$-characters of the fundamental representations $V_{i,r}$ is given explicitly, we will also recall it in Section~\ref{sectapp}.
\end{rem}

\begin{proof}
We prove this Proposition by induction on $m$, the number of times the mutation sequence $\mu_\mathscr{S}$ is applied on the initial quantum cluster variables $\{~u_{i,r}\mid (i,r)\in\mI\}$.

The base step is given noting, as in \citep{ACAA}, that the images by the isomorphism $\eta$ of the initial quantum cluster variables are indeed truncated $\qt$-characters. For all $(i,r)\in \mI$,
\begin{align*}
w_{i,r}^{(0)} & = \eta(u_{i,r}) = U_{i,r} = \prod_{\substack{k\geq 0 \\ r+2k\leq 0}}Y_{i,r+2k}\quad \in \Y_t^-.
\end{align*}
Thus
\begin{equation}\label{yW}
w_{i,r}^{(0)} =  [W_{k_{i,r},r}^{(i)}]_t^-.
\end{equation}

Let $m\geq 0$ and $(i,r)\in\mI$. Supposed we have applied $m$ times the mutation sequence $\mu_\mathscr{S}$, and a $(m+1)$th time on all vertices preceding $(i,r)$ in the sequence $\mathscr{S}$, and that all those previous vertices satisfy (\ref{eqprop}). 

We want to write the quantum exchange relation corresponding to the mutation at vertex $(i,r)$. From the proof of Theorem 3.1 in \citep{ACAA}, we know the shape of the quiver just before this mutation ($\A_t$ and $\A(\pmb u, \Gm)$ are defined on the same initial quiver and the mutation process on the quiver is the same for classical and quantum cluster algebras). 

As explained in \citep[Section 3.2.3]{ACAA}, for a general simply laced Lie algebra $\g$, the mutation process takes place at vertices $(i,r)$ having two (or one if $r=-\xi_i$) in-going arrows from $(i,r\pm 2)$ and outgoing arrows to vertices $(j,s)$, with $j\sim i$. Thus the effect of the mutation sequence $\mu_\mathscr{S}$ on two fixed columns of the quiver is the same as the effect of an iteration of the mutation sequence on the corresponding quiver of rank 2. 

Let us recall the mutation process on the quiver when $\g$ is of type $A_2$.

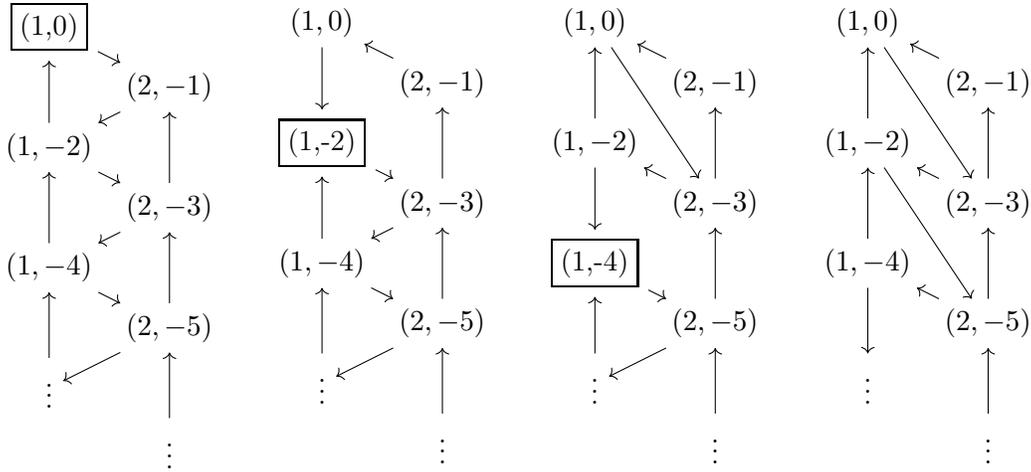
\begin{minipage}{3.5cm}
\begin{tikzpicture}[scale=0.8]
\node (16) at (0,0) {\boxed{$(1,0)$}};
\node (14) at (0,-2) {$(1,-2)$};
\node (12) at (0,-4) {$(1,-4)$};
\node (10) at (0,-6) {$\vdots$};

\node (15) at (2,-1) {$(2,-1)$};
\node (13) at (2,-3) {$(2,-3)$};
\node (11) at (2,-5) {$(2,-5)$};
\node (9) at (2,-7) {$\vdots$};

\draw [->] (10) edge (12);
\draw [->] (12) edge (14);
\draw [->] (14) edge (16);

\draw [->] (9) edge (11);
\draw [->] (11) edge (13);
\draw [->] (13) edge (15);

\draw [->] (16) edge (15);
\draw [->] (14) edge (13);
\draw [->] (12) edge (11);

\draw [->] (15) edge (14);
\draw [->] (13) edge (12);
\draw [->] (11) edge (10);
\end{tikzpicture}
\end{minipage}
\begin{minipage}{3.5cm}
\begin{tikzpicture}[scale=0.8]
\node (16) at (0,0) {$(1,0)$};
\node (14) at (0,-2) {\boxed{$(1,-2)$}};
\node (12) at (0,-4) {$(1,-4)$};
\node (10) at (0,-6) {$\vdots$};

\node (15) at (2,-1) {$(2,-1)$};
\node (13) at (2,-3) {$(2,-3)$};
\node (11) at (2,-5) {$(2,-5)$};
\node (9) at (2,-7) {$\vdots$};

\draw [->] (10) edge (12);
\draw [->] (12) edge (14);
\draw [->] (16) edge (14);

\draw [->] (9) edge (11);
\draw [->] (11) edge (13);
\draw [->] (13) edge (15);

\draw [->] (15) edge (16);
\draw [->] (14) edge (13);
\draw [->] (12) edge (11);

\draw [->] (13) edge (12);
\draw [->] (11) edge (10);

\end{tikzpicture}
\end{minipage}
\begin{minipage}{3.5cm}
\begin{tikzpicture}[scale=0.8]
\node (16) at (0,0) {$(1,0)$};
\node (14) at (0,-2) {$(1,-2)$};
\node (12) at (0,-4) {\boxed{$(1,-4)$}};
\node (10) at (0,-6) {$\vdots$};

\node (15) at (2,-1) {$(2,-1)$};
\node (13) at (2,-3) {$(2,-3)$};
\node (11) at (2,-5) {$(2,-5)$};
\node (9) at (2,-7) {$\vdots$};

\draw [->] (10) edge (12);
\draw [->] (14) edge (12);
\draw [->] (14) edge (16);

\draw [->] (9) edge (11);
\draw [->] (11) edge (13);
\draw [->] (13) edge (15);

\draw [->] (15) edge (16);
\draw [->] (13) edge (14);
\draw [->] (12) edge (11);

\draw [->] (11) edge (10);

\draw [->] (16) edge (13);
\end{tikzpicture}
\end{minipage}
\begin{minipage}{3.5cm}
\begin{tikzpicture}[scale=0.8]
\node (16) at (0,0) {$(1,0)$};
\node (14) at (0,-2) {$(1,-2)$};
\node (12) at (0,-4) {$(1,-4)$};
\node (10) at (0,-6) {$\vdots$};

\node (15) at (2,-1) {$(2,-1)$};
\node (13) at (2,-3) {$(2,-3)$};
\node (11) at (2,-5) {$(2,-5)$};
\node (9) at (2,-7) {$\vdots$};

\draw [->] (12) edge (10);
\draw [->] (12) edge (14);
\draw [->] (14) edge (16);

\draw [->] (9) edge (11);
\draw [->] (11) edge (13);
\draw [->] (13) edge (15);

\draw [->] (15) edge (16);
\draw [->] (13) edge (14);
\draw [->] (11) edge (12);


\draw [->] (16) edge (13);
\draw [->] (14) edge (11);
\end{tikzpicture}
\end{minipage}

After an infinite number of mutations (or a sufficiently large one), we start mutating on the second column.

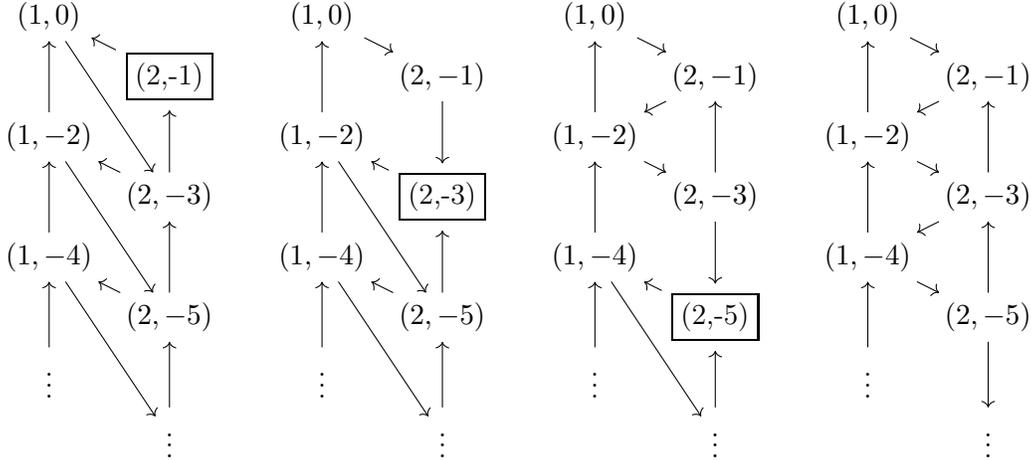
\begin{minipage}{3.5cm}
\begin{tikzpicture}[scale=0.8]
\node (16) at (0,0) {$(1,0)$};
\node (14) at (0,-2) {$(1,-2)$};
\node (12) at (0,-4) {$(1,-4)$};
\node (10) at (0,-6) {$\vdots$};

\node (15) at (2,-1) {\boxed{$(2,-1)$}};
\node (13) at (2,-3) {$(2,-3)$};
\node (11) at (2,-5) {$(2,-5)$};
\node (9) at (2,-7) {$\vdots$};

\draw [->] (10) edge (12);
\draw [->] (12) edge (14);
\draw [->] (14) edge (16);

\draw [->] (9) edge (11);
\draw [->] (11) edge (13);
\draw [->] (13) edge (15);

\draw [->] (15) edge (16);
\draw [->] (13) edge (14);
\draw [->] (11) edge (12);


\draw [->] (16) edge (13);
\draw [->] (14) edge (11);
\draw [->] (12) edge (9);
\end{tikzpicture}
\end{minipage}
\begin{minipage}{3.5cm}
\begin{tikzpicture}[scale=0.8]
\node (16) at (0,0) {$(1,0)$};
\node (14) at (0,-2) {$(1,-2)$};
\node (12) at (0,-4) {$(1,-4)$};
\node (10) at (0,-6) {$\vdots$};

\node (15) at (2,-1) {$(2,-1)$};
\node (13) at (2,-3) {\boxed{$(2,-3)$}};
\node (11) at (2,-5) {$(2,-5)$};
\node (9) at (2,-7) {$\vdots$};

\draw [->] (10) edge (12);
\draw [->] (12) edge (14);
\draw [->] (14) edge (16);

\draw [->] (9) edge (11);
\draw [->] (11) edge (13);
\draw [->] (15) edge (13);

\draw [->] (16) edge (15);
\draw [->] (13) edge (14);
\draw [->] (11) edge (12);


\draw [->] (14) edge (11);
\draw [->] (12) edge (9);
\end{tikzpicture}
\end{minipage}
\begin{minipage}{3.5cm}
\begin{tikzpicture}[scale=0.8]
\node (16) at (0,0) {$(1,0)$};
\node (14) at (0,-2) {$(1,-2)$};
\node (12) at (0,-4) {$(1,-4)$};
\node (10) at (0,-6) {$\vdots$};

\node (15) at (2,-1) {$(2,-1)$};
\node (13) at (2,-3) {$(2,-3)$};
\node (11) at (2,-5) {\boxed{$(2,-5)$}};
\node (9) at (2,-7) {$\vdots$};

\draw [->] (10) edge (12);
\draw [->] (12) edge (14);
\draw [->] (14) edge (16);

\draw [->] (9) edge (11);
\draw [->] (13) edge (11);
\draw [->] (13) edge (15);

\draw [->] (16) edge (15);
\draw [->] (14) edge (13);
\draw [->] (11) edge (12);

\draw [->] (15) edge (14);

\draw [->] (12) edge (9);
\end{tikzpicture}
\end{minipage}
\begin{minipage}{3.5cm}
\begin{tikzpicture}[scale=0.8]
\node (16) at (0,0) {$(1,0)$};
\node (14) at (0,-2) {$(1,-2)$};
\node (12) at (0,-4) {$(1,-4)$};
\node (10) at (0,-6) {$\vdots$};

\node (15) at (2,-1) {$(2,-1)$};
\node (13) at (2,-3) {$(2,-3)$};
\node (11) at (2,-5) {$(2,-5)$};
\node (9) at (2,-7) {$\vdots$};

\draw [->] (10) edge (12);
\draw [->] (12) edge (14);
\draw [->] (14) edge (16);

\draw [->] (11) edge (9);
\draw [->] (11) edge (13);
\draw [->] (13) edge (15);

\draw [->] (16) edge (15);
\draw [->] (14) edge (13);
\draw [->] (12) edge (11);

\draw [->] (15) edge (14);
\draw [->] (13) edge (12);

\end{tikzpicture}
\end{minipage}

Thus, in general, the quantum exchange relation has the form:
\begin{equation}\label{1mut}
u_{i,r}^{(m+1)} = u_{i,r+2}^{(m+1)}u_{i,r-2}^{(m)} \left(u_{i,r}^{(m)}\right)^{-1} + \prod_{j\sim i_0}u_{j,r-\epsilon_i}^{(m+\xi_i)}\left(u_{i,r}^{(m)}\right)^{-1} ,
\end{equation}
where $u_{i,r+2}^{(m)}=1$ if $r+2\geq 0$ and $\epsilon_i$ is defined in (\ref{epsilon}), and both terms are commutative products. This relation can also be written:

\begin{equation}\label{mut}
u_{i,r}^{(m+1)}\ast u_{i,r}^{(m)} = t^\alpha u_{i,r+2}^{(m+1)}u_{i,r-2}^{(m)} + t^\beta \prod_{j\sim i_0}u_{j,r-\epsilon_i}^{(m+\xi_i)},
\end{equation}
where $\alpha, \beta\in\frac{1}{2}\Z$.

If we apply $\eta$ to (\ref{mut}), and use the induction hypothesis, we get the following relation in $\Y_t^-$:
\begin{multline}
\eta(u_{i,r}^{(m+1)})\ast [W_{k_{i,r},r-2m}^{(i)}]_t^- = t^\alpha [W_{k_{i,r+2},r-2m}^{(i)}]_t^-[W_{k_{i,r-2},r-2-2m}^{(i)}]_t^- \\
+ t^\beta \prod_{j\sim i} [W_{k_{j,r-\epsilon_i},r-1-2m}^{(j)}]_t^-.
\end{multline}
Whereas, the corresponding (truncated) quantum $T$-system relation (\ref{Tsysm}) is
\begin{multline}\label{1Tsys}
[W_{k_{i,r},r-2m-2}^{(i)}]_t^- \ast [W_{k_{i,r},r-2m}^{(i)}]_t^- = t^{\alpha'} [W_{k_{i,r}-1,r-2m}^{(i)}]_t^-[W_{k_{i,r}+1,r-2-2m}^{(i)}]_t^- \\
+ t^{\beta'} \prod_{j\sim i} [W_{k_{i,r},r-1-2m}^{(j)}]_t^-,
\end{multline}
where $\alpha', \beta'\in\frac{1}{2}\Z$ are given in (\ref{ag}). Note that one has indeed
\begin{equation}
k_{i,r+2}=k_{i,r}-1,\quad k_{i,r-2}=k_{i,r}+1, \quad \text{and } k_{j,r-\epsilon_i}=k_{i,r}, \text{ for } j\sim i.
\end{equation}

Let $k=k_{i,r}$ and $r'=r-2m$. Let us precise how to obtain the coefficients $\alpha$ and $\beta$. From (\ref{1mut}) and (\ref{mut}), $\alpha$ and $\beta$ are such that the terms
\begin{equation}\label{talphaWWW}
t^\alpha [W_{k-1,r'}^{(i)}]_t^-[W_{k+1,r'-2}^{(i)}]_t^- \ast \left( [W_{k,r'}^{(i)}]_t^-\right)^{-1},
\end{equation}
and
\begin{equation}\label{tbetaWWW}
t^\beta \prod_{j\sim i} [W_{k,r'-1}^{(j)}]_t^- \ast \left( [W_{k,r'}^{(i)}]_t^-\right)^{-1}, 
\end{equation}
are bar-invariant. Thus, if one takes only the dominant monomials of (\ref{talphaWWW}) and (\ref{tbetaWWW}), they are bar-invariant:
\begin{align*}
t^\alpha m_{k-1,r'}^{(i)}m_{k+1,r'-2}^{(i)}\ast \left(m_{k,r'}^{(i)}\right)^{-1} & = t^{-\alpha} \left(m_{k,r'}^{(i)}\right)^{-1} \ast m_{k-1,r'}^{(i)}m_{k+1,r'-2}^{(i)},\\
t^\beta \prod_{j\sim i}m_{k,r'-1}^{(j)} \ast \left(m_{k,r'}^{(i)}\right)^{-1} & = t^{-\beta} \left(m_{k,r'}^{(i)}\right)^{-1} \ast \prod_{j\sim i}m_{k,r'-1}^{(j)}.
\end{align*}
This enables us to compute $\alpha$ and $\beta$.
\begin{align*}
\alpha & = \frac{1}{2}\sum_{l=0}^{k-1}\left( \sum_{p=0}^{k-2}\mathcal{N}_{i,i}(2l-2p) +  \sum_{p=0}^{k}\mathcal{N}_{i,i}(2l-2p+2) \right)\\
&= \frac{1}{2}\sum_{l=0}^{k-1}\left( {\bf C}_{ii}(2l+1) -{\bf C}_{ii}(2l-2k+3) + {\bf C}_{ii}(2l+3) - {\bf C}_{ii}(2l-2k+1) \right) \\
& =  \frac{1}{2}\left( {\bf C}_{ii}(2k+1) + {\bf C}_{ii}(2k-1)\right) - {\bf C}_{ii}(1).
\end{align*}
Thus $\alpha = \alpha(i,k) = -1 + \frac{1}{2}\left( \tilde{C}_{ii}(2k-1) + \tilde{C}_{ii}(2k+1)\right)$. And
\begin{align*}
\beta & = \frac{1}{2}\sum_{j\sim i}\sum_{l=0}^{k-1}\sum_{p=0}^{k-1}\mathcal{N}_{i,j}(2l-2p+1)
= \frac{1}{2}\sum_{j\sim i}\sum_{l=0}^{k-1} \left( {\bf C}_{ij}(2l+2)- {\bf C}_{ij}(2l-2k+2)\right) \\
& = \frac{1}{2}\sum_{j\sim i}\left( {\bf C}_{ij}(2k) - \underbrace{{\bf C}_{ij}(0)}_{=0} \right) =  \frac{1}{2}\left( {\bf C}_{ii}(2k+1) + {\bf C}_{ii}(2k-1)\right), \quad \text{ using (\ref{Clem})}.
\end{align*}
Hence $\beta = \gamma(i,k)= \frac{1}{2}\left( \tilde{C}_{ii}(2k-1) + \tilde{C}_{ii}(2k+1)\right)$.

Thus $\alpha=\alpha'$ and $\beta=\beta'$, and
\begin{equation}
\eta(u_{i,r}^{(m+1)})\ast [W_{k,r'}^{(i)}]_t^- = [W_{k,r'-2}^{(i)}]_t^- \ast [W_{k,r'}^{(i)}]_t^-.
\end{equation}

However, $[W_{k,r'}^{(i)}]_t^-$ is invertible in the skew-field of fractions $\mathcal{F}_t$ of the quantum torus $\Y_t^-$ (see Remark~\ref{remsfield}). Thus
\begin{equation}
\eta(u_{i,r}^{(m+1)}) = [W_{k_{i,r},r-2-2m}^{(i)}]_t^-.
\end{equation}
This concludes the induction.

Finally, from \citep[Corollary 6.14]{CqC}, we know that for all $(i,r)\in \mI$, the monomials $m$ occurring in the $q$-character $\chi_q(W_{k,r}^{(i)})$ of the KR-module $W_{k,r}^{(i)}$ are products of $Y_{j,r+s}^{\pm 1}$, with $0\leq s\leq 2k + h$. Moreover, from Theorem~\ref{theoN}, if one writes the $\qt$-character $[W_{k,r}^{(i)}]_t$ of this KR-module as a linear combination of Laurent monomials in the variables $Y_{i,s}$, with coefficients in $\Z[t^{\pm 1}]$, all monomials which occur in this expansion also occur in its $q$-character. Thus, if $r+2k+h \leq 0$, the truncated $\qt$-character $[W_{k,r}^{(i)}]_t^-$ is equal to the $\qt$-character $[W_{k,r}^{(i)}]_t$.
In particular, for $m\geq h$, 
\begin{equation*}
w_{i,r}^{(m)}  = [W_{k_{i,r},r-2m}^{(i)}]_t.
\end{equation*}
\end{proof}

		\subsection{Proof of Theorem~\ref{theo1}}\label{sectproof}

We can now prove Theorem~\ref{theo1}. This proof is a quantum analog of the proof of \citep[Theorem 5.1]{ACAA}. Naturally, there are technical difficulties brought forth by the non-commutative quantum tori structure. For example, in our situation, the quantum cluster algebra $\A_t$ is isomorphic to the truncated quantum Grothendieck ring $K_t(\CZm)$ only via the isomorphism of quantum tori $\eta$ (\ref{eta}). In particular, we need the following result.

\begin{lem}\label{lemeta'}
The identification 
\begin{equation}
\eta' : u_{i,r}  \longmapsto  [W_{k_{i,r},r}^{(i)}]_t.
\end{equation}
extends to a well-defined injective $\Z[t^{\pm 1/2}]$-algebras morphism 
\begin{equation*}
\eta': T \to \mathcal{F}_t,
\end{equation*}
where $\mathcal{F}_t$ is the skew-field of fractions of $\mathcal{Y}_t$ (see Remark~\ref{remsfield}).

Moreover, the restriction of $\eta'$ to the quantum cluster algebra $\A_t$ has its image in the quantum torus $\mathcal{Y}_t$  and the $\Z[t^{\pm 1/2}]$-algebra morphisms $\eta,\eta'$ and $\pi$ satisfy the following commutative diagram:
\begin{equation}\label{diagetaeta'}
\xymatrix@C=1em@R=0.7em{
\A_t \ar[rr]^{\eta} \ar[rd]_{\eta'}& & \mathcal{Y}_t^- \\
&  \mathcal{Y}_t \ar[ru]_{\pi} &.
}
\end{equation}
\end{lem}

\begin{proof}
From Proposition~\ref{propyirm}, for all $(i,r)\in\mI$, the \emph{full} $\qt$-character $[W_{k_{i,r},r-2h}^{(i)}]_t$ is obtained as the image of the cluster variable sitting at vertex $(i,r)$ after applying $h$ times the mutation sequence $\mathscr{S}$ (which is locally a finite sequence of mutations):
\begin{equation}
\eta(u_{i,r}^{(h)}) = [W_{k_{i,r},r-2h}^{(i)}]_t.
\end{equation}
In particular, for any two vertices $(i,r),(j,s)$, the variables $u_{i,r}^{(h)}$ and $u_{j,s}^{(h)}$ belong to a common cluster and $t$-commute. Thus the $\qt$-characters $[W_{k_{i,r},r-2h}^{(i)}]_t$ $t$-commute. As the quantum torus $\mathcal{Y}_t$ is invariant by shift of quantum parameters ($Y_{j,s}\mapsto Y_{j,s+2}$), the $\qt$-characters $[W_{k_{i,r},r}^{(i)}]_t$ also $t$-commute for the product $\ast$. Their $t$-commutation relations are determined by their dominant monomials, which are $\eta(u_{i,r})$. Thus the $[W_{k_{i,r},r}^{(i)}]_t$ satisfy exactly the same $t$-commutations relations are the $u_{i,r}$. This proves the first part of the lemma.

Let $X$ be a cluster variable of $\A_t$ obtained from the initial seed $\pmb u =\{u_{i,r}\}$ via finite sequence of mutations $\sigma$. We want to show that $\eta'(X)\in \mathcal{Y}_t$. As the sequence of mutations $\sigma$ is finite, it will only involve a finite number of cluster variables. Now apply $h$ times the mutation sequence $\mu_\mathscr{S}$ to the initial seed so as to replace each cluster variable considered by $u_{i,r}^{(h)}$ (again, we only need a finite number of mutations). Let us summarize:
\begin{align*}
\eta'(u_{i,r}) = [W_{k_{i,r},r}^{(i)}]_t, \quad \eta(u_{i,r}^{(h)}) = [W_{k_{i,r},r-2h}^{(i)}]_t.
\end{align*} 
Let $X'$ be the cluster variable obtained by applying to this new seed the sequence of mutations $\sigma$. By construction, $\eta(X')$ is equal to $\eta'(X)$, up to the downward shift of spectral parameters by $2h$: every variable $Y_{j,s}^{\pm 1}$ is replaced by $Y_{j,s-2h}^{\pm 1}$. In particular, $\eta'(X)\in\mathcal{Y}_t$.

Next, the commutation of diagram (\ref{diagetaeta'}) is verified as it is satisfied on the initial seed $\pmb u =\{u_{i,r}\}$.
\end{proof}

Let $R_t$ be the image of the quantum cluster algebra $\A_t$
\begin{equation}
R_t:=\eta(\A_t)\quad \in \Y_t^-.
\end{equation}

The inclusion $K_t(\CZm) \subset R_t$ is essentially contained in Proposition~\ref{propyirm}. For the reverse inclusion, the main idea is to use the characterization of the quantum Grothendieck ring as the intersection of kernel of operators, called deformed screening operators. We show by induction on the length of a sequence of mutations that the images of all cluster variables belong to those kernels. The images of the initial cluster variables $u_{i,r}$ clearly belong to the quantum Grothendieck ring, and the screening operators being derivations, the exchange relations force the newly created cluster variables to be in the intersection of the kernels too. let us prove this in details.

\begin{proof}
Recall from Lemma~\ref{lemCZm} that the quantum Grothendieck ring $K_t(\CZm)$ is algebraically generated by the truncated $\qt$-characters of the fundamental modules:
\begin{equation*}
K_t(\CZm) = \left\langle [L(Y_{i,r})]_t^- \mid (i,r)\in\mI \right\rangle.
\end{equation*}
By Proposition~\ref{propyirm}, for all $(i,r)\in\mI$,
\begin{equation}
[L(Y_{i,r})]_t^- = w_{i,\xi_i}^{(\frac{r-\xi_i}{2})} = \eta\left(u_{i,\xi_i}^{(\frac{r-\xi_i}{2})} \right) \quad \in\eta(\A_t).
\end{equation}
Which proves the first inclusion:
\begin{equation}
K_t(\CZm) \subset R_t.
\end{equation}

We prove the reverse inclusion as explained just above.

As explained in Section~\ref{sectqGR}, Hernandez proved in \citep{tAOE} that for all $i\in I$ there exists operators $S_{i,t} : \Y_t \to \Y_{i,t}$, where $\Y_{i,t}$ is a $\mathcal{Y}_t$-module, which are $\Z[t^{\pm 1}]$-linear and derivations, such that
\begin{equation}\label{capSit}
\bigcap_{i\in I}\ker(S_{i,t}) = K_t(\CZ).
\end{equation}

Notice that these operators characterize the quantum Grothendieck ring $K_t(\CZ)$ and not $K_t(\CZm)$. Hence the need for Lemma~\ref{lemeta'}.

Let us prove by induction that all cluster variables $Z$ in $\A_t$ satisfy $\eta'(Z) \in K_t(\CZ)$.

Let $Z$ be a quantum cluster variable in $\A_t$. If $Z$ belongs to the initial cluster variables, $Z=u_{i,r}$ and 
\begin{equation}
\eta'(Z)=\eta'(u_{i,r}) =  [W_{k_{i,r},r}]_t \quad \in K_t(\CZ).
\end{equation}
If not, then by induction on the length of the sequence $\mu$, one can assume that $Z$ is obtained via a quantum exchange relation 
\begin{equation}\label{ZZ1}
Z\ast Z_1 = t^\alpha M_1 + t^\beta M_2,
\end{equation}
where $Z_1$ is a quantum cluster variable of $\A_t$, $M_1$ and $M_2$ are quantum cluster monomials of $\A_t$ and 
\begin{equation*}
\eta'(Z_1), \eta'(M_1), \eta'(M_2) \in K_t(\C). 
\end{equation*}
 Apply $\eta'$ to (\ref{ZZ1}):
\begin{equation}
\eta'(Z)\ast \eta'(Z_1 )= t^\alpha\eta'( M_1) + t^\beta \eta'(M_2).
\end{equation}
For all $i\in I$, apply the derivation $S_{i,t}$:
\begin{align*}
S_{i,t}\left( \eta'(Z)\ast \eta'(Z_1 ) \right) & = S_{i,t}\left( \eta'(Z)\right)\ast \eta'(Z_1 ) +\eta'(Z)\ast S_{i,t}\left(\eta'(Z_1 ) \right) \\
	& = t^\alpha S_{i,t}\left(\eta'( M_1)\right) + t^\beta S_{i,t}\left(\eta'( M_2)\right).
\end{align*}
However, by hypothesis, 
\begin{flalign*}
S_{i,t}\left( \eta'(Z_1)\right)  = 0, &\\
S_{i,t}\left( \eta'(M_1)\right)  = 0,\\
S_{i,t}\left( \eta'(M_1)\right)  = 0.
\end{flalign*}
Moreover, $\eta'(Z_1 )\neq 0$ and the images of the screening operator is in a free module over $\mathcal{Y}_t$ by Lemma~\ref{lemYit}. Thus $S_{i,t}\left( \eta'(Z)\right) =0$, for all $i\in I$. Hence
\begin{equation*}
\eta'(Z) \in K_t(\CZm),
\end{equation*}
which concludes the induction. We have proven 
\begin{equation*}
\eta'(\A_t) \subset K_t(\CZ).
\end{equation*}
Then, by the commutation of the diagram (\ref{diagetaeta'}) in Lemma~\ref{lemeta'},
\begin{equation*}
R_t = \eta(\A_t) \subset K_t(\CZm).
\end{equation*}
Which concludes the proof of the theorem.


\end{proof}

	\section{Application to the proof of an inclusion conjecture}\label{sectapp}

In this section, we use the quantum cluster algebra structure of the Grothendieck ring $\CZm$ to prove that the quantum Grothendieck ring $K_t(\mathcal{O}^+_\Z)$ defined in \citep{LEA2} contains $K_t(\CZm)$. In other words, we prove Conjecture~1 in \citep{LEA2}. The result was already proven in that paper in type $A$, but the core argument used was different.

		\subsection{The quantum Grothendieck ring $K_t(\mathcal{O}^+_\Z)$}
The quantum Grothendieck ring $K_t(\mathcal{O}^+_\Z)$ is defined as a quantum cluster algebra on the full infinite quiver $\Gamma$, of which the semi-infinite quiver $\Gm$ is a subquiver.

Let us recall some notations. For all $i,j\in I$, $\mathcal{F}_{ij}: \Z \to \Z$ is a anti-symmetrical map such that, for all $m\geq 0$,
\begin{equation}
\mathcal{F}_{ij}(m) = - \sum_{\substack{k\geq 1 \\ m \geq 2k-1}}\tilde{C}_{ij}(m -2k+1).
\end{equation}
Let $\mathcal{T}_t$ be the quantum torus defined as the $\Z[t^{\pm 1/2}]$-algebra generated by the variables $z_{i,r}^\pm$, for $(i,r)\in\hat{I}$, with a non-commutative product $\ast$, and the $t$-commutations relations
\begin{equation}
z_{i,r}\ast z_{j,s} = t^{\F_{ij}(s-r)} z_{j,s}\ast z_{i,r}, \quad \left( (i,r),(j,s)\in\hat{I} \right).
\end{equation}
Recall also from \citep[Proposition 5.2.2]{LEA2} the inclusion of quantum tori $\mathcal{J}$ (with a slight shift of parameters on the $z_{i,r}$): 
\begin{equation}\label{inclJ}
\mathcal{J}: \left\lbrace\begin{array}{rcl}
				\mathcal{Y}_t & \longrightarrow & \mathcal{T}_t,\\
				Y_{i,r} & \longmapsto & z_{i,r}\left(z_{i,r+2}\right)^{-1}.
\end{array}\right. .
\end{equation}

Let $\Lambda$ be the infinite skew-symmetric $\hat{I}\times\hat{I}$-matrix:
\begin{equation}\label{Lambdatot}
\Lambda_{(i,r),(j,s)} = \F_{ij}(s-r), \quad \left( (i,r),(j,s)\in\hat{I}\right).
\end{equation}
From \citep{LEA2}, the quiver $\Gamma$ and the skew-symmetric matrix $\Lambda$ form a compatible pair. Let $\A_t(\Gamma,\Lambda)$ be the associated quantum cluster algebra. Then, as in \citep[Definition 6.3.5]{LEA2},
\begin{equation}
K_t(\mathcal{O}^+_\Z) := \A_t(\Gamma,\Lambda)\hat{\otimes}\mathcal{E},
\end{equation}
where $\mathcal{E}$ is a commutative ring and the completion allows for certain countable sums.

		\subsection{Intermediate quantum cluster algebras}

The general idea is to see the quantum cluster algebra $\A_t$ as a "sub-quantum cluster algebra" of $\A_t(\Gamma, \Lambda)$ (this term is not well-defined). However, as in Section~\ref{sectproof} and contrary to the aforementioned proof, as we are dealing with quantum cluster algebras in our setting, this is not done trivially. Mainly, one notices that the map
\begin{equation*}
\begin{array}{rcl}
	T & \longrightarrow & \mathcal{T}_t \\
	u_{i,r} & \longmapsto & z_{i,r},
\end{array}
\end{equation*}
is not a well-defined inclusion of quantum tori, as the generators $u_{i,r}$ and $z_{i,r}$ do not satisfy the same $t$-commutation relations.

First, consider the subquiver $\Gamma^-$ of $\Gamma$ of index set 
\begin{equation}
\hat{I}^-_{\leq 2} := \I\cap\{ (i,r) \mid i\in I, r\leq 2\},
\end{equation}
such that the vertices $(i,r)$, with $r>0$ are frozen.

To summarize, for $\xi_i=0$ and $j\sim i$:

$\Gamma=$
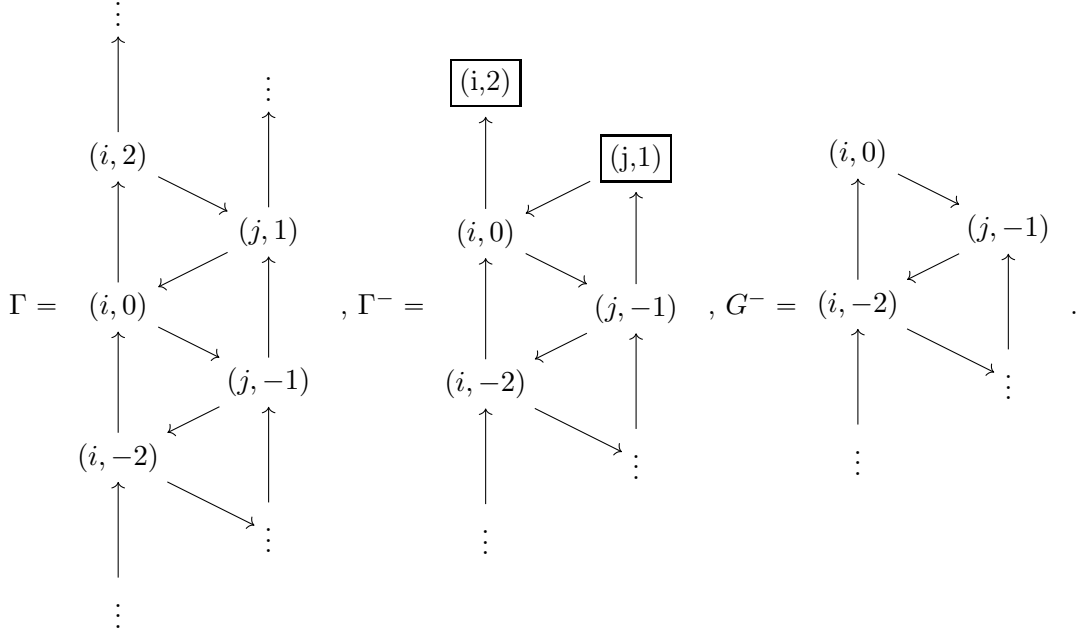
\begin{minipage}[c]{3.5cm}
\begin{tikzpicture}
\node (18) at (0,2) {$\vdots$};
\node (16) at (0,0) {$(i,2)$};
\node (14) at (0,-2) {$(i,0)$};
\node (12) at (0,-4) {$(i,-2)$};
\node (10) at (0,-6) {$\vdots$};

\node (17) at (2,1) {$\vdots$};
\node (15) at (2,-1) {$(j,1)$};
\node (13) at (2,-3) {$(j,-1)$};
\node (11) at (2,-5) {$\vdots$};

\draw [->] (10) edge (12);
\draw [->] (12) edge (14);
\draw [->] (14) edge (16);
\draw [->] (16) edge (18);

\draw [->] (11) edge (13);
\draw [->] (13) edge (15);
\draw [->] (15) edge (17);

\draw [->] (16) edge (15);
\draw [->] (14) edge (13);
\draw [->] (12) edge (11);

\draw [->] (15) edge (14);
\draw [->] (13) edge (12);
\end{tikzpicture}
\end{minipage}
, $\Gamma^-=$
\begin{minipage}[c]{3.5cm}
\begin{tikzpicture}
\node (16) at (0,0) {\boxed{$(i,2)$}};
\node (14) at (0,-2) {$(i,0)$};
\node (12) at (0,-4) {$(i,-2)$};
\node (10) at (0,-6) {$\vdots$};

\node (15) at (2,-1) {\boxed{$(j,1)$}};
\node (13) at (2,-3) {$(j,-1)$};
\node (11) at (2,-5) {$\vdots$};

\draw [->] (10) edge (12);
\draw [->] (12) edge (14);
\draw [->] (14) edge (16);

\draw [->] (11) edge (13);
\draw [->] (13) edge (15);

\draw [->] (14) edge (13);
\draw [->] (12) edge (11);

\draw [->] (15) edge (14);
\draw [->] (13) edge (12);
\end{tikzpicture}
\end{minipage}
, $\Gm=$
\begin{minipage}[c]{3.5cm}
\begin{tikzpicture}
\node (14) at (0,-2) {$(i,0)$};
\node (12) at (0,-4) {$(i,-2)$};
\node (10) at (0,-6) {$\vdots$};

\node (13) at (2,-3) {$(j,-1)$};
\node (11) at (2,-5) {$\vdots$};

\draw [->] (10) edge (12);
\draw [->] (12) edge (14);

\draw [->] (11) edge (13);

\draw [->] (14) edge (13);
\draw [->] (12) edge (11);

\draw [->] (13) edge (12);
\end{tikzpicture}
\end{minipage}.

The quivers $\Gamma$ and $\Gamma^-$ are only connected by coefficients (as in \citep[Definition 4.1]{GRA1}), thus by \citep[Theorem 4.5]{GRA1} the inclusion of seeds $\Gamma^-,\Lambda \subset \Gamma,\Lambda$ induces an inclusion of the quantum cluster algebra $\A_t(\Gamma^-,\Lambda)$ into the quantum cluster algebra of $\A_t(\Gamma,\Lambda)$.

Now we need to link the quantum cluster algebras $\A_t$ and $\A_t(\Gamma^-,\Lambda)$. 

In order to do that, we use a result from \citep{GRA2}, which deals with \emph{graded} cluster algebras.

\begin{defi}\citep{GRA2}
A quantum cluster seed $(B, \Lambda)$ is \emph{graded} if there exists  an integer column vector $G$ such that, for all mutable indices $k$, the $k$th row of $B$, $B_j$ satisfies $B_j G=0$.

Then, the $G$-degree of the initial variables are set by the vector $G$: for all cluster variables $X_k$ in the initial cluster $\bar{X}$, $\deg_G(X)=G_i$. 

The grading condition is equivalent to the following, for all mutable variables $X_k$, the sum of the degrees of all variables with arrows to $X_k$ is equal to the sum of the degrees of all variables with arrows coming from $X_k$, i.e. exchange relations are homogeneous. Hence, each cluster variable has a well-defined degree.
\end{defi}

Let us start by considering the quantum cluster algebra $\A_t(\Gamma^-, L)$, built on the quiver $\Gamma^-$, with coefficients 1 on the frozen vertices. This quantum cluster algebra is clearly isomorphic to $\A_t$, and is a graded quantum cluster algebra of grading $G=0$.

Next, for all $i\in I$, we apply the process of \citep[Theorem 4.6]{GRA2} to add coefficients $f_i$, while twisting the $t$-commutation relations.

Let 
\begin{equation}
\underline{u}_i(j,s) = \delta_{i,j}, \quad \left( (j,s) \in \mI_{\leq 2} \right)
\end{equation}
and 
\begin{equation}
\underline{t}_i(j,s) = -\mathcal{F}_{ij}(2-s-\xi_i), \quad \left( (j,s) \in \mI_{\leq 2} \right),
\end{equation}
with $\xi : I\to \{0,1\}$ the height function, fixed in Section~\ref{sectheightfun}.
Then 
\begin{lem}
For all $i\in I$, $\underline{u}_i$ and $\underline{t}_i$ are gradings for the ice quiver $\Gamma^-$.
\end{lem}

\begin{proof}
For all $(j,s) \in\mI$, let $B_{(j,s)}$ be the $(j,s)$th "row" of the $\It\times \It$-skew-symmetric matrix encoding the adjacency of the ice quiver $\Gamma^-$.

For all $i\in I$, 
\begin{align*}
	B_{(j,s)}\underline{u}_i & = \sum_{(k,r)\in\It}B_{(j,s)}(k,r)\underline{u}_i(k,r) \\
	& = \underline{u}_i(j,s+2) - \underline{u}_i(j,s-2) + \sum_{k\sim j}\left( \underline{u}_i(k,s-1) - \underline{u}_i(k,s+1)\right) \\
	& = \delta_{i,j} - \delta_{i,j} + \sum_{k\sim j}\left( \delta_{i,k} -\delta_{i,k} \right) =0.
\end{align*}
And 
\begin{align*}
	B_{(j,s)}\underline{t}_i & = \sum_{(k,r)\in\It}B_{(j,s)}(k,r)\underline{t}_i(k,r) \\
	& = \underline{t}_i(j,s+2) - \underline{t}_i(j,s-2) + \sum_{k\sim j}\left( \underline{t}_i(k,s-1) - \underline{t}_i(k,s+1)\right) \\
	& =  -\F_{ij}(-s-\xi_i) + \F_{ij}(4-s-\xi_i) + \sum_{k\sim j}\left( -\F_{ik}(3-s-\xi_i) + \F_{ik}(1-s-\xi_i) \right) \\
	& = -\tilde{C}_{ij}(3-s-\xi_i) - \tilde{C}_{ij}(1-s-\xi_i) + \sum_{k\sim j}\tilde{C}_{ik}(2-s-\xi_i)= 0,
\end{align*}
from Lemma~\ref{lemCtildeC}.
\end{proof}
Then, by \citep[Theorem 4.6]{GRA2} (applied $n$ times), the following is a valid set of initial data for a (multi-)graded quantum cluster algebra $\A_t(\tilde{\pmb u},\tilde{B},\tilde{L},\tilde{G})$, where
\begin{itemize}
	\item[•] or all $(i,r)\in\It$, 
\begin{equation}\label{newtildeu}
	\tilde{u}_{i,r} = \left\lbrace\begin{array}{ll}
				u_{i,r}f_i & \text{ if } r\leq 0,\\
				f_i & \text{ otherwise}
	\end{array}\right. .
\end{equation}
	and $\tilde{\pmb u}=\{\tilde{u}_{i,r}\}_{(i,r)\in\It}$.
	\item[•] $\tilde{B}=B$, the $\It\times\It$-skew-symmetric adjacency matrix of the quiver $\Gamma^-$.
	\item[•]  $\tilde{L}$ encodes the $t$-commutations relations, such that, for all $(i,r)\in\mI$, and $j\in I$,
\begin{equation}\label{newtcommfj}
	f_j\ast u_{i,r} = t^{\underline{t}_j(i,r)} u_{i,r}\ast f_j,
\end{equation}
and the $f_j$ pairwise commute.
	\item[•] $\tilde{G}$ is a multi-grading, i.e. instead of being an integer column vector, each entry in $\tilde{G}$ is in the lattice $\Z^I$. It is defined by, for all $(i,r)\in\It$,
\begin{equation}\label{tildeG}
	\tilde{G}(\tilde{u}_{i,r})=e_i \quad\in \Z^I,
\end{equation}
or $\deg_i=\underline{u}_i$, for all $i\in I$.

This is indeed the construction of \citep{GRA2}, with the initial grading on $\A_t(\Gamma^-,L)$ being $G\equiv 0$. The new quantum cluster algebra is denoted by $\A_t^{\underline{u},\underline{t}}(\Gamma^-,L)$, to show that it is a twisted version of $\A_t(\Gamma^-,L)$.
\end{itemize}

\begin{prop}\label{propisoquant}
The quantum cluster algebra $\A_t(\tilde{\pmb u},\tilde{B},\tilde{L})$ is isomorphic to the quantum cluster algebra $\A_t(\Gamma^-,\Lambda)$.
\end{prop}

\begin{proof}
The rest of the data being the same, we only have to check that the $\It\times \It$-skew-symmetric matrices $\tilde{L}$ and $\Lambda_{|\It}$ are equal.

From (\ref{Lambdatot}), for all $\left( (i,r),(j,s)\in\It\right)$,
\begin{equation*}
\Lambda\left( (i,r),(j,s)\right) =  \F_{ij}(s-r).
\end{equation*}
And $\tilde{L}$ is defined as 
\begin{equation}
\tilde{u}_{i,r}\ast \tilde{u}_{j,s} = t^{\tilde{L}\left( (i,r),(j,s)\right)}\tilde{u}_{j,s}\ast\tilde{u}_{i,r}.
\end{equation}
Hence,
\begin{equation*}
\tilde{L}\left( (i,r),(j,s)\right) = \left\lbrace \begin{array}{ll}
		L\left( (i,r),(j,s)\right) + \underline{t}_i(j,s) - \underline{t}_j(i,r) & ,  \text{ if } (i,r),(j,s)\in\mI,\\
		\underline{t}_i(j,s)  & ,  \text{ if } \left[ \begin{array}{l}
						(i,r)=(i,-\xi_i+2)\\
						(j,s)\in\mI
\end{array}\right.,\\
		0 & , \text{ if } \left[\begin{array}{l}
				(i,r) =(i,-\xi_i+2)\\
				(j,s) =(j,-\xi_j+2)
		\end{array}\right. .
\end{array}\right. 
\end{equation*}
First, notice that, for all $i,j\in I$, $m\in\Z$,
\begin{equation}
\mathcal{N}_{ij}(m) = 2\F_{ij}(m) - \F_{ij}(m+2) - \F_{ij}(m-2).
\end{equation}
This result is proven in \citep{LEA2}, in the course of the proof of Proposition 5.2.2.

Thus, for all $(i,r),(j,s)\in\mI$,
\begin{align*}
	L\left( (i,r),(j,s)\right) & = \sum_{\substack{k\geq 0\\ r+2k\leq 0}} \sum_{\substack{l\geq 0\\ s+2l\leq 0}}\mathcal{N}_{ij}(s+2l-r-2k) \\
	& = \sum_{\substack{k\geq 0\\ r+2k\leq 0}} \left( \F_{ij}(s-r-2k) - \F_{ij}(s-r-2k-2) \right. \\
	& \left. + \F_{ij}(-\xi_i - r -2k) - \F_{ij}(-\xi_j-r-2k+2) \right)\\
	& = \F_{ij}(s-r) - \F_{ij}(s+ \xi_i-2) + \underbrace{\F_{ij}(-\xi_j+\xi_i)}_{=0} - \F_{ij}(-\xi_j-r+2)\\
	& = \F_{ij}(s-r) -\underline{t}_i(j,s) + \underline{t}_j(i,r).
\end{align*}
And of course, for all $i\in I$, $(j,s)\in\mI$,
\begin{equation*}
\Lambda\left( (i,-\xi_i+2),(j,s)\right) = \F_{ij}(s+\xi_i -2) = \underline{t}_i(j,s).
\end{equation*}
Thus, one had indeed, 
\begin{equation}
\tilde{L} = \Lambda_{|\It}.
\end{equation}
\end{proof}

From now on, we will use the notations $\pmb z =\{z_{i,r}, f_j\}_{(i,r)\in\mI,j\in I}$ for the initial clusters variables of both $\A_t(\Gamma^-,\Lambda)$ and $\A_t^{\underline{u},\underline{t}}(\Gamma^-,L)$.

\begin{rem}
\begin{itemize}
	\item[•] This result is natural, if we look at what the different initial cluster variables mean in terms of $\ell$-weights. From (\ref{uY}), for all $(i,r)\in\mI$, the cluster variable $u_{i,r}$ can be identified with the (commutative) dominant monomial $U_{i,r}=\prod_{\substack{k\geq 0 \\ r+2k\leq 0}}Y_{i,r+2k}$. Whereas, the quantum tori $\mathcal{Y}_t$ and $\mathcal{T}_t$ are compared via the inclusion $\mathcal{J}$ of (\ref{inclJ})
\begin{equation*}
	\mathcal{J} : Y_{i,r} \longmapsto z_{i,r}\left(z_{i,r+2}\right)^{-1}, \quad\forall (i,r)\in \I.
\end{equation*} 
Thus, the link between the variables $u_{i,r}$ and $z_{i,r}$ is the following:
\begin{equation}\label{uzf}
	u_{i,r} \equiv z_{i,r}\left(f_i\right)^{-1},
\end{equation}
with the previous convention $f_i=z_{i,-\xi_i+2}$. More precisely, there are different maps of quantum tori:
\begin{equation}\label{diagTTYY}
\xymatrix{T \ar@{.>}[r] \ar[d]_{\eta}^{\sim} & \mathcal{T}_t\\
\mathcal{Y}_t^- \ar@{^{(}->}[r] & \mathcal{Y}_t \ar[u]^{\mathcal{J}}
},
\end{equation}
where identification (\ref{uzf}) is the resulting dotted map, which we will denote by $\rho$:
\begin{equation}\label{rho}
\rho : T\to \mathcal{T}_t.
\end{equation} 
The quantum cluster algebra $\A_t^{\underline{u},\underline{t}}(\Gamma^-,L)$ was built with this identification in mind.

	\item[•] This process could also be seen as a quantum version of the multi-grading homogenization process of the seed $(\pmb u,\Gm)$, as in \citep[Lemma 7.1]{GRA3}, where the multi-grading is defined by (\ref{tildeG}).
\end{itemize}
\end{rem}

		 \subsection{Inclusion of quantum Grothendieck rings}

In the section, we prove that the quantum Grothendieck ring $K_t(\cO^+_\Z)$, or more precisely, the quantum cluster algebra $\A_t(\Gamma,\Lambda)$ contains the quantum Grothendieck ring $K_t(\CZ)$, which is the statement of Conjecture 1 in \citep{LEA2}. Recalled that in \citep{LEA2}, the ring $K_t(\cO^+_\Z)$ was defined as a completion of the quantum cluster algebra $\A_t(\Gamma,\Lambda)$, but the aim was to see it as a quantum Grothendieck ring for the category of representations $\cO^+_\Z$ from \citep{ARDRF} and \citep{CABS}. As this category contains the category $\CZ$, it was expected for the quantum Grothendieck ring $K_t(\cO^+_\Z)$ to contain $K_t(\CZ)$.

In order to prove this result we will actually prove Conjecture~2 of the same paper, which is a stronger result. We state it as follows.
\begin{theo}\label{theoconj2}
The $\qt$-characters of all fundamental representations in $\CZ$ are obtained as quantum cluster variables in the quantum cluster algebra $\A_t(\Gamma,\Lambda)$.
\end{theo}

More precisely, we show that for all $i\in I$, there exists a specific finite sequence of mutations $S_i$ in $\A_t(\Gamma,\Lambda)$ such that, if applied to the initial seed $\{z_{j,s}\}_{(j,s)\in\I}$, the cluster variable sitting at vertex $(i,-\xi_i)$ is the image by $\mathcal{J}$ (of (\ref{inclJ})) of the $\qt$-character of the fundamental module $[L(Y_{i,-\xi_i-2h'})]_t$, where 
\begin{equation}
h' = \lceil \frac{h}{2} \rceil,
\end{equation} 
with $h$ the Coxeter number of the simple Lie algebra $\g$.


Let us define the sequence $S_i$. Let $(i_1,i_2,\ldots,j_n)$ be an ordering of the columns of $\Gamma$ as in (\ref{ordercol}), such that $i_1=i$ (take first all columns $j$ such that $\xi_j=\xi_i$). The sequence $S_i$ is a sequence of vertices of $\Gamma$, and more precisely of $\Gamma^-$, defined as follows. First read all vertices $(i_1,r)$ for $-2h'+2\leq r \leq 0$, from top to bottom, then all $(i_2,r)$, with $-2h'\leq r \leq 0$, and so on, then read again all vertices $(i_1,r)$ for $-2h'+4\leq r \leq 0$, and continue browsing the columns successively, until at the last step you only read the vertex $(i,-\xi_i)$.

Note that applying this sequence $S_i$ of mutations on the quivers $\Gamma^-$ or $\Gm$ has exactly the same effect on the cluster variable sitting at vertex $(i,-\xi_i)$ than applying $h'$ times the infinite sequence $\mathscr{S}$ from Section~\ref{sectsequ}.

\begin{ex}\label{exseqD4}
For $\g$ of type $D_4$, $h=6$, then $h'=3$. Let us give explicitly the sequence $S_2$ starting on column 2. For simplicity of notations, we assume that the height function is chosen such that $\xi_2=0$. Then the sequence $S_2$ has 15 steps:
\begin{multline}\label{D4steps}
S_2= (2,0),(2,-2),(2,-4),(1,-1),(1,-3),(3,-1),(3,-3),(4,-1),(4,-3),\\
(2,0),(2,-2),(1,-1)(3,-1)(4,-1)(2,0).
\end{multline}
\end{ex}
	
Let $r_0=-\xi_i-2h'$. Let $\tilde{\chi}_{i,r_0} \in \mathcal{T}_t$ be the quantum cluster variable obtained at vertex $(i,-\xi_i)$ after applying the sequence of mutations $S_i$ to the quantum cluster variable $\A_t(\Gamma^-,\Lambda)$ with initial seed $\{z_{j,s},f_k\}_{(j,s)\in\mI, k\in I}$.
\begin{prop}\label{propfinale}
As an element of the quantum torus $\mathcal{T}_t$, $\tilde{\chi}_{i,r_0}$ belongs to the image of the inclusion morphism $\mathcal{J}$. 

Moreover,
\begin{equation}
\tilde{\chi}_{i,r_0} = \mathcal{J}\left( [L(Y_{i,r_0})]_t\right).
\end{equation}
\end{prop}

\begin{proof}
The cluster variable $\tilde{\chi}_{i,r_0}$ is a variable of the quantum cluster algebra $\A_t(\Gamma^-,\Lambda)$, which is isomorphic to $\A_t^{\underline{u},\underline{t}}(\Gamma^-,L)$ from Proposition~\ref{propisoquant}. By \citep[Corollary 4.7]{GRA2}, there is a bijection between the quantum cluster variables of $\A_t^{\underline{u},\underline{t}}(\Gamma^-,L)$ and those of $\A_t(\Gamma^-,L)$.

With notations from Section~\ref{secttruncq}, $u_{i,-\xi_i}^{(h')}$ is the cluster variable of $\A_t(\Gamma^-,L)$ obtained at vertex $(i,-\xi_i)$ after applying the mutations of the sequence $S_i$. Also from \citep[Corollary 4.7]{GRA2}, we know that there exists integers $a_j\in \Z$ such that
\begin{equation}\label{tildechiuf}
\tilde{\chi}_{i,r_0} =\rho\left(u_{i,-\xi_i}^{(h')}\right)\prod_{j\in I}f_j^{a_j},
\end{equation}
written as a commutative product (both $\tilde{\chi}_{i,r_0}$ and $u_{i,-\xi_i}^{(h')}$ are bar-invariant), with $\rho$ defined in (\ref{rho}). The term $u_{i,-\xi_i}^{(h')}$ is a Laurent polynomial in the variables $u_{j,s}$, which satisfy $\rho(u_{j,s})=z_{j,s}(f_j)^{-1}$ from (\ref{uzf}). Thus expression (\ref{tildechiuf}) is a way of writing $\tilde{\chi}_{i,r_0}$ as a Laurent polynomial in the initial variables $\{z_{j,s},f_k\}$. However, one can write $\tilde{\chi}_{i,r_0}=N/D$, where $N$ is the Laurent polynomial in the cluster variables $\{z_{j,s},f_k\}$, with coefficients in $\Z[t^{\pm 1/2}]$ and not divisible by any of the $f_k$, and $D$ is a monomial in the non-frozen variables $\{z_{j,s}\}$. Thus $\prod_{j\in I}f_j^{a_j}$ is the smallest monomial such that 
\begin{equation*}
\rho\left(u_{i,-\xi_i}^{(h')}\right)\prod_{j\in I}f_j^{a_j}
\end{equation*}
contains only non-negative powers of the frozen variables $f_k$.

Moreover, from Proposition~\ref{propyirm}, 
\begin{equation}
\eta(u_{i,-\xi_i}^{(h')}) = w_{i,-\xi_i}^{(h')} = [L(Y_{i,r_0})]_t.
\end{equation}

However, all Laurent monomials occurring in $[L(Y_{i,r_0})]_t$ already occurred in the $q$-character $\chi_q(L(Y_{i,r_0}))$, as the $\qt$-character $[L(Y_{i,r_0})]_t$ has positive coefficients. Indeed, the $\qt$-characters of fundamental modules have been explicitly computed and have been found to have non-negative coefficients (in \citep{NtqAD} for types $A$ and $D$ and \citep{NE8} for type $E$).

From \citep{CqC}, all monomials in $\chi_q(L(Y_{i,r_0}))$ are products of $Y_{j,s}^{\pm 1}$, with $s\leq r_0+h$, but by definition of $r_0$, $r_0+h \leq 0$, and the term with the highest quantum parameter being the anti-dominant monomial $Y_{\overline{i},r_0+h}^{-1}$, where $\overline{\phantom{A}}: I\to I$ is the involutive map such that $\omega_0(\alpha_j)=-\alpha_{\overline{j}}$, with $\omega_0$ the longest element of the Weyl group of $\g$ (no relation with the bar-involution of Section~\ref{sectbar}).

Consider the change of variables, for $(j,s)\in \mI$,
\begin{equation}
y_{j,s} = \eta^{-1}(Y_{j,s})= \left\lbrace\begin{array}{ll}
		\frac{u_{j,s}}{u_{j,s+2}}, & \text{if } s+2 \leq 0,\\
		u_{j,s}, & \text{otherwise}
\end{array}\right..
\end{equation}
Thus 
\begin{equation}
\rho(y_{j,s}) = \left\lbrace\begin{array}{ll}
		\frac{z_{j,s}}{z_{j,s+2}}, & \text{if } s+2 \leq 0,\\
		\frac{z_{j,s}}{f_j}, & \text{otherwise}
\end{array}\right..
\end{equation}
All monomials occurring in $[L(Y_{i,r_0})]_t$ are commutative monomials in the variables $Y_{j,s}^{\pm 1}$, with $s\leq r_0 +h \leq 0$. Moreover, the only monomials in which the variables $Y_{j,s}^{\pm 1}$, with $s+2 \geq 0$, occur are the anti-dominant monomial $Y_{\overline{i},r_0+h}^{-1}$, and any possible monomial in which some variable $Y_{j,r_0+h-1}$ occurs. But for such a monomial $m$, the variable $Y_{j,s}^{\pm 1}$ with the highest $s$ in $m$ occurs with a negative power in $m$ (the monomial $m$ is "right-negative", as from \citep[Lemma 6.5]{CqC}), thus the variable $Y_{j,r_0+h-1}$ also occurs with a negative power.

The image by $\rho\circ\eta^{-1}$ of any monomial in the variables $Y_{j,s}^{\pm 1}$, with $s+2\leq 0$ is a monomial in the variables $\{z_{j,s}^{\pm 1}\}$ (without frozen variables). Thus, the image
\begin{equation*}
\rho\left(u_{i,-\xi_i}^{(h')}\right)= \rho\circ\eta^{-1}\left( [L(Y_{i,r_0})]_t\right)
\end{equation*}
is a Laurent polynomial with only positive powers of the variables $f_j$. 

Necessarily, $\prod_{j\in I}f_j^{a_j}=1$, and (\ref{tildechiuf}) becomes
\begin{equation*}
\tilde{\chi}_{i,r_0} =\rho\left(u_{i,-\xi_i}^{(h')}\right).
\end{equation*}
Finally, from the diagram (\ref{diagTTYY}), 
\begin{equation*}
\tilde{\chi}_{i,r_0} =\rho\left(u_{i,-\xi_i}^{(h')}\right) = \mathcal{J}\left( \eta\left(u_{i,-\xi_i}^{(h')}\right)\right) = \mathcal{J}\left( [L(Y_{i,r_0})]_t\right),
\end{equation*}
which concludes the proof.
\end{proof}

\begin{rem}
At some point in the proof we used the fact that the $\qt$-characters of the fundamental modules had non-negative coefficients. Note that this part of the proof could easily be extended to non-simply laced types, as the $\qt$-characters of their fundamental representations have also be explicitly computed, and also have non-negative coefficients (in types $B$ and $C$, the $\qt$-characters of all fundamental representations are equal to their respective $q$-character, and all coefficients are actually equal to 1 \citep[Proposition 7.2]{MqtC}, and see \citep{AAqtC} for type $G_2$ and \citep{MqtC} for type $F_4$). 
\end{rem}

\begin{cor}\label{corconj}
For all $(i,r)\in \I$ there exists a quantum cluster variable $\tilde{\chi}_{i,r}$ in the quantum cluster algebra $\A_t(\Gamma, \Lambda)$ such that 
\begin{equation*}
\tilde{\chi}_{i,r} = \mathcal{J}\left( [L(Y_{i,r})]_t\right).
\end{equation*}
\end{cor}

\begin{proof}
For all $(i,r)\in \I$, let $\tilde{\chi}_{i,r}$ be the cluster variable of the quantum cluster algebra $\A_t(\Gamma, \Lambda)$ obtained at vertex $(i,r+2h')$ after applying the sequence of mutations $S_i$, but starting at vertex $(i,r+2h')$ instead of $(i,-\xi_i)$. 

Consider the change of variables in $\A_t(\Gamma, \Lambda)$:
\begin{equation}
	s: z_{j,s} \longmapsto z_{j,s+r_0-r} , \quad \forall (j,s)\in \I.
\end{equation}
The quantum cluster algebra $\A_t(\Gamma, \Lambda)$ is invariant under this shift $s$, and this change of variables is clearly invertible ($s^{-1}(z_{j,s})=z_{j,s+r-r_0}$). One has $s(z_{i,r}) = z_{i,r_0}$, and $s(z_{i,r+2h'})=z_{i,-\xi_i}$, thus $s\left(\tilde{\chi}_{i,r}\right)= \tilde{\chi}_{i,r_0}$, and from Proposition~\ref{propfinale},
\begin{equation*}
\tilde{\chi}_{i,r} = s^{-1}\left( \tilde{\chi}_{i,r_0} \right)= s^{-1}\left( \mathcal{J}\left( [L(Y_{i,r_0})]_t\right) \right). 
\end{equation*}
However, from the definition of the map $\mathcal{J}$ in (\ref{inclJ}), the shift $s$ also acts as a change of variables in the quantum torus $\mathcal{Y}_t$, $s : Y_{j,u} \longmapsto Y_{j,u+r_0-r}$. Hence,
\begin{equation*}
\tilde{\chi}_{i,r} = \mathcal{J}\left( [L(Y_{i,r})]_t\right).
\end{equation*}
 \end{proof}

Thus we have proven Theorem~\ref{theoconj2}.

\begin{rem}\label{remcomplexity}
\begin{itemize}
	\item[•] When Conjecture~2 was formulated in \citep{LEA2}, a recent positivity result of Davison \citep{PQCA} was mentioned there. This work proves the so-called "positivity conjecture" for quantum cluster algebras, which states that the coefficients of the Laurent polynomials into which the cluster variables decompose from the Laurent phenomenon are in fact non-negative. This is an important result, but also a difficult one, and it is not actually needed in order to obtain our result.
	
	\item[•] One can note from this proof that we know a close bound on the number of mutations needed in order to compute the $\qt$-character of a fundamental module. For $\g$ a simple-laced simple Lie algebra of rank $n$ and of (dual) Coxeter number $h$, if $h'=\lceil h/2\rceil$, then the number of steps is lower than
\begin{equation}\label{boundsteps}
		n\frac{h'(h'+1)}{2}.
\end{equation}
We have chosen to go into details on the number of steps required for this process because it made sense from an algorithmic point of view to know its complexity. 

One can compare this algorithm to Frenkel-Mukhin \citep{CqC} to compute $q$-characters. As explained in \citep{NE8}, when trying to compute $q$-characters of fundamental representations of large dimension (for example, in type $E_8$, the $q$-character of the $5$th fundamental representation has approximately $6.4\times 2^{30}$ monomials), one encounters memory issues. Indeed, this algorithm has to keep track of all the previously computed terms. This advantage of the cluster algebra approach is that one only had to keep the seed in memory.
\end{itemize}
\end{rem}

	\section{Explicit computation in type $D_4$}\label{sectexplD4}

For $\g$ of type $D_4$, with the height function $\xi_2 = 0, \xi_1=\xi_3=\xi_4=1$, we compute the $\qt$-character of the fundamental representation $L(Y_{2,-6})$ as a quantum cluster variable of the quantum cluster algebra $\A_t(\Gamma^-,\Lambda)$, using the algorithm presented in the previous Section.

From Example~\ref{exseqD4}, the sequence of mutation we have to apply to the initial seed is
\begin{multline*}
S_2= (2,0),(2,-2),(2,-4),(1,-1),(1,-3),(3,-1),(3,-3),(4,-1),(4,-3),\\
(2,0),(2,-2),(1,-1)(3,-1)(4,-1)(2,0).
\end{multline*} 
One can notice that the required number of steps is indeed lower than 24, which was the bound given in (\ref{boundsteps}).

Let us give explicitly the mutations on the quiver $\Gamma^-$ (which encodes more than $\Gm$), as well as the quantum cluster variables obtained at (almost every) step. We give the quantum cluster variables as Laurent polynomials in the variables $\{z_{i,r},f_j\}_{(i,r)\in \I, j\in I}$, as well as in the form (\ref{tildechiuf}). For completeness, we also compute the multi-degrees of the quantum cluster variables.

This computation was done thanks to the latest version (as of January 2019) of Bernhard Keller's \href{https://webusers.imj-prg.fr/~bernhard.keller/quivermutation/}{\color{blue} wonderful quiver mutation applet}.
\vspace*{2em}

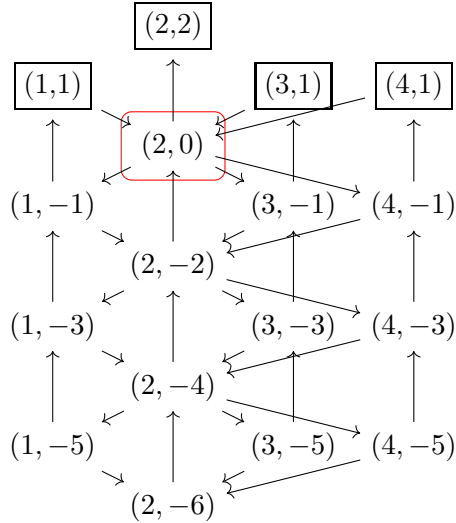
\begin{minipage}[c]{10cm}
\begin{center}
\begin{tikzpicture}[scale=0.8]
\path[use as bounding box] (-2, 2.5) rectangle (7, -6.5);
\node (1) at (2,2) {\boxed{$(2,2)$}};
\node (2) at (0,1) {\boxed{$(1,1)$}};
\node (3) at (4,1) {\boxed{$(3,1)$}};
\node (4) at (6,1) {\boxed{$(4,1)$}};

\node (5) at (2,0) {$(2,0)$};
\node (6) at (0,-1) {$(1,-1)$};
\node (7) at (4,-1) {$(3,-1)$};
\node (8) at (6,-1) {$(4,-1)$};

\node (9) at (2,-2) {$(2,-2)$};
\node (10) at (0,-3) {$(1,-3)$};
\node (11) at (4,-3) {$(3,-3)$};
\node (12) at (6,-3) {$(4,-3)$};

\node (13) at (2,-4) {$(2,-4)$};
\node (14) at (0,-5) {$(1,-5)$};
\node (15) at (4,-5) {$(3,-5)$};
\node (16) at (6,-5) {$(4,-5)$};

\node (17) at (2,-6) {$(2,-6)$};
\node (box) [draw=red,rounded corners,fit = (5)] {};

\draw [->] (5) edge (1);
\draw [->] (9) edge (5);
\draw [->] (13) edge (9);
\draw [->] (17) edge (13);

\draw [->] (6) edge (2);
\draw [->] (10) edge (6);
\draw [->] (14) edge (10);

\draw [->] (7) edge (3);
\draw [->] (11) edge (7);
\draw [->] (15) edge (11);

\draw [->] (8) edge (4);
\draw [->] (12) edge (8);
\draw [->] (16) edge (12);

\draw [->] (5) edge (6);
\draw [->] (5) edge (7);
\draw [->] (5) edge (8);

\draw [->] (9) edge (10);
\draw [->] (9) edge (11);
\draw [->] (9) edge (12);

\draw [->] (13) edge (14);
\draw [->] (13) edge (15);
\draw [->] (13) edge (16);

\draw [->] (2) edge (5);
\draw [->] (3) edge (5);
\draw [->] (4) edge (5);

\draw [->] (6) edge (9);
\draw [->] (7) edge (9);
\draw [->] (8) edge (9);

\draw [->] (10) edge (13);
\draw [->] (11) edge (13);
\draw [->] (12) edge (13);

\draw [->] (14) edge (17);
\draw [->] (15) edge (17);
\draw [->] (16) edge (17);
\end{tikzpicture}
\end{center}
\end{minipage}	
\begin{minipage}[c]{4cm}
Mutation at $(2,0)$.

\begin{equation*}
\deg\left( z_{2,0}^{(1)} \right) = e_1 + e_3 + e_4 .
\end{equation*}

\end{minipage}

\begin{align*}
z_{2,0}^{(1)} &  = z_{2,-2}\left(z_{2,0}\right)^{-1}f_1 f_3 f_4  + z_{1,-1}z_{3,-1}z_{4,-1}\left(z_{2,0}\right)^{-1}f_2,\\
= &  \mathcal{J}\left(Y_{2,-2} + Y_{1,-1}Y_{3,-1}Y_{4,-1}Y_{2,0}^{-1}\right)f_1 f_3 f_4.
\end{align*}

\vspace*{1em}

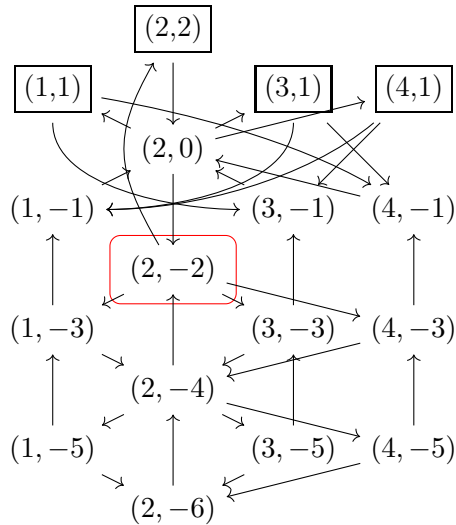
\begin{minipage}{10cm}
\begin{center}
\begin{tikzpicture}[scale=0.8]
\path[use as bounding box] (-2, 2.2) rectangle (7, -6.5);
\node (1) at (2,2) {\boxed{$(2,2)$}};
\node (2) at (0,1) {\boxed{$(1,1)$}};
\node (3) at (4,1) {\boxed{$(3,1)$}};
\node (4) at (6,1) {\boxed{$(4,1)$}};

\node (5) at (2,0) {$(2,0)$};
\node (6) at (0,-1) {$(1,-1)$};
\node (7) at (4,-1) {$(3,-1)$};
\node (8) at (6,-1) {$(4,-1)$};

\node (9) at (2,-2) {$(2,-2)$};
\node (10) at (0,-3) {$(1,-3)$};
\node (11) at (4,-3) {$(3,-3)$};
\node (12) at (6,-3) {$(4,-3)$};

\node (13) at (2,-4) {$(2,-4)$};
\node (14) at (0,-5) {$(1,-5)$};
\node (15) at (4,-5) {$(3,-5)$};
\node (16) at (6,-5) {$(4,-5)$};

\node (17) at (2,-6) {$(2,-6)$};

\node (box) [draw=red,rounded corners,fit = (9)] {};

\draw [->] (1) edge (5);
\draw [->] (5) edge (9);
\draw [->] (13) edge (9);
\draw [->] (17) edge (13);

\draw [->] (10) edge (6);
\draw [->] (14) edge (10);

\draw [->] (11) edge (7);
\draw [->] (15) edge (11);

\draw [->] (12) edge (8);
\draw [->] (16) edge (12);

\draw [->] (6) edge (5);
\draw [->] (7) edge (5);
\draw [->] (8) edge (5);

\draw [->] (9) edge (10);
\draw [->] (9) edge (11);
\draw [->] (9) edge (12);

\draw [->] (13) edge (14);
\draw [->] (13) edge (15);
\draw [->] (13) edge (16);

\draw [->] (5) edge (2);
\draw [->] (5) edge (3);
\draw [->] (5) edge (4);


\draw [->] (10) edge (13);
\draw [->] (11) edge (13);
\draw [->] (12) edge (13);

\draw [->] (14) edge (17);
\draw [->] (15) edge (17);
\draw [->] (16) edge (17);

\draw [->] (9) to[out=120, in=240, looseness=1.2] (1);

\draw [->] (2) to[out=270, in=180, looseness=0.8] (7);
\draw [->] (2) to[out=350, in=150, looseness=0.8] (8);

\draw [->] (3) to[out=270, in=0, looseness=0.8] (6);
\draw [->] (3) to (8);

\draw [->] (4) to[out=220, in=0, looseness=0.8] (6);
\draw [->] (4) to (7);
\end{tikzpicture}
\end{center}
\end{minipage}	
\begin{minipage}[c]{4.2cm}
Mutation at $(2,-2)$.
\vspace*{1em}

\begin{equation*}
\deg\left( z_{2,-2}^{(1)} \right) = e_1 + e_3 + e_4.
\end{equation*}
\end{minipage}

\begin{align*}
z_{2,-2}^{(1)} & = z_{2,-4}\left(z_{2,0}\right)^{-1}f_1 f_3 f_4  + z_{1,-1}z_{3,-1}z_{4,-1}z_{2,-4}\left(z_{2,-2}\right)^{-1} \left(z_{2,0}\right)^{-1} f_2 \\
 & + z_{1,-3}z_{3,-3}z_{4,-3}\left(z_{2,-2}\right)^{-1}  f_2,  \\
z_{2,-2}^{(1)} & = \mathcal{J}\left(Y_{2,-4}Y_{2,-2} + Y_{1,-1}Y_{2,-4}\left(Y_{2,0}\right)^{-1}Y_{3,-1}Y_{4,-1} \right. \\
 & \left.+ Y_{1,-3}Y_{1,-1}\left(Y_{2,-2}\right)^{-1}\left(Y_{2,0}\right)^{-1}Y_{3,-3}Y_{3,-1}Y_{4,-3}Y_{4,-1}\right)f_1 f_3 f_4.
\end{align*}

\vspace*{5em}

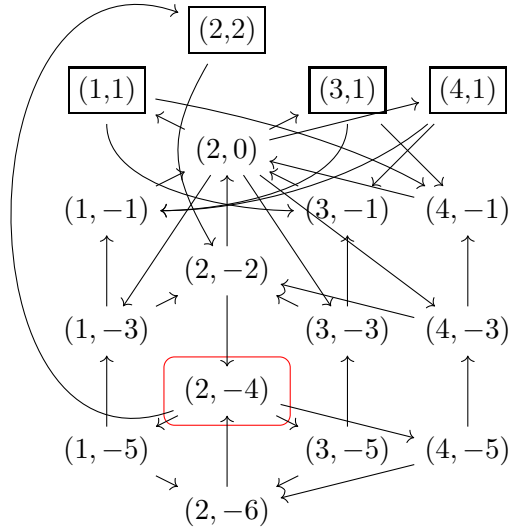
\begin{minipage}{10cm}
\begin{center}
\begin{tikzpicture}[scale=0.8]
\path[use as bounding box] (-2, 2.5) rectangle (7, -6);
\node (1) at (2,2) {\boxed{$(2,2)$}};
\node (2) at (0,1) {\boxed{$(1,1)$}};
\node (3) at (4,1) {\boxed{$(3,1)$}};
\node (4) at (6,1) {\boxed{$(4,1)$}};

\node (5) at (2,0) {$(2,0)$};
\node (6) at (0,-1) {$(1,-1)$};
\node (7) at (4,-1) {$(3,-1)$};
\node (8) at (6,-1) {$(4,-1)$};

\node (9) at (2,-2) {$(2,-2)$};
\node (10) at (0,-3) {$(1,-3)$};
\node (11) at (4,-3) {$(3,-3)$};
\node (12) at (6,-3) {$(4,-3)$};

\node (13) at (2,-4) {$(2,-4)$};
\node (14) at (0,-5) {$(1,-5)$};
\node (15) at (4,-5) {$(3,-5)$};
\node (16) at (6,-5) {$(4,-5)$};

\node (17) at (2,-6) {$(2,-6)$};

\node (box) [draw=red,rounded corners,fit = (13)] {};

\draw [->] (9) edge (5);
\draw [->] (9) edge (13);
\draw [->] (17) edge (13);

\draw [->] (10) edge (6);
\draw [->] (14) edge (10);

\draw [->] (11) edge (7);
\draw [->] (15) edge (11);

\draw [->] (12) edge (8);
\draw [->] (16) edge (12);

\draw [->] (6) edge (5);
\draw [->] (7) edge (5);
\draw [->] (8) edge (5);

\draw [->] (10) edge (9);
\draw [->] (11) edge (9);
\draw [->] (12) edge (9);

\draw [->] (13) edge (14);
\draw [->] (13) edge (15);
\draw [->] (13) edge (16);

\draw [->] (5) edge (2);
\draw [->] (5) edge (3);
\draw [->] (5) edge (4);



\draw [->] (14) edge (17);
\draw [->] (15) edge (17);
\draw [->] (16) edge (17);

\draw [->] (1) to[out=240, in=120, looseness=1.2] (9);

\draw [->] (2) to[out=270, in=180, looseness=0.8] (7);
\draw [->] (2) to[out=350, in=150, looseness=0.8] (8);

\draw [->] (3) to[out=270, in=0, looseness=0.8] (6);
\draw [->] (3) to (8);

\draw [->] (4) to[out=220, in=0, looseness=0.8] (6);
\draw [->] (4) to (7);

\draw [->] (5) to (10);
\draw [->] (5) to (11);
\draw [->] (5) to (12);

\draw [->] (13) to[out=200, in=160, looseness=1.5] (1);
\end{tikzpicture}
\end{center}
\end{minipage}
\begin{minipage}[c]{4.2cm}
Mutation at $(2,-4)$.
\vspace*{1em}

\begin{equation}
\deg\left( z_{2,-4}^{(1)} \right) = e_1 + e_3 + e_4.
\end{equation}
\end{minipage}

\vspace*{1em}

\begin{align*}
	z_{2,-4}^{(1)} & = z_{2,-6}\left(z_{2,0}\right)^{-1}f_1 f_3 f_4 + z_{1,-1}z_{3,-1}z_{4,-1}z_{2,-6}\left(z_{2,-2}\right)^{-1} \left(z_{2,0}\right)^{-1} f_2\\
 	& + z_{1,-3}z_{3,-3}z_{4,-3}z_{2,-6}\left(z_{2,-4}\right)^{-1}\left(z_{2,-2}\right)^{-1}  f_2 + z_{1,-5}z_{3,-5}z_{4,-5}\left(z_{2,-4}\right)^{-1}f_2,\\
 z_{2,-4}^{(1)} & = \mathcal{J}\left( Y_{2,-6}Y_{2,-4}Y_{2,-2} + Y_{1,-1}Y_{2,-6}Y_{2,-4}\left(Y_{2,0}\right)^{-1}Y_{3,-1}Y_{4,-1} \right. \\
 	& + Y_{1,-3}Y_{1,-1}Y_{2,-6}\left(Y_{2,-2}\right)^{-1}\left(Y_{2,0}\right)^{-1}Y_{3,-3}Y_{3,-1}Y_{4,-3}Y_{4,-1}  \\
 	& \left. + Y_{1,-5}Y_{1,-3}Y_{1,-1}\left(Y_{2,-4}\right)^{-1}\left(Y_{2,-2}\right)^{-1}Y_{3,-5}Y_{3,-3}Y_{3,-1}Y_{4,-5}Y_{4,-3}Y_{4,-1}\right)f_1 f_3 f_4.
\end{align*}


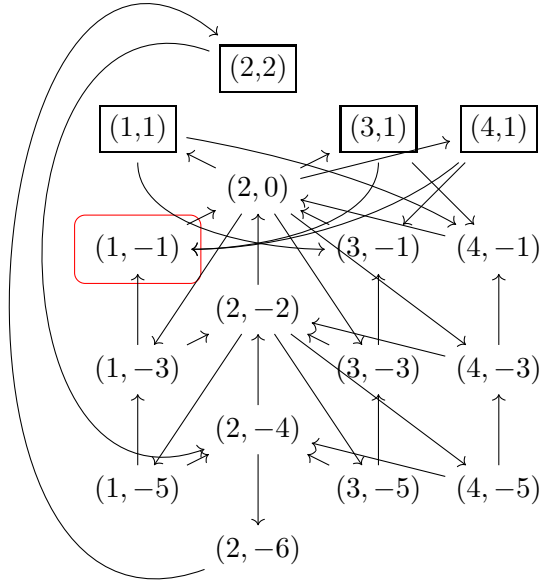
\begin{minipage}{10cm}
\begin{center}
\begin{tikzpicture}[scale=0.8]
\path[use as bounding box] (-2, 3) rectangle (7, -7);
\node (1) at (2,2) {\boxed{$(2,2)$}};
\node (2) at (0,1) {\boxed{$(1,1)$}};
\node (3) at (4,1) {\boxed{$(3,1)$}};
\node (4) at (6,1) {\boxed{$(4,1)$}};

\node (5) at (2,0) {$(2,0)$};
\node (6) at (0,-1) {$(1,-1)$};
\node (7) at (4,-1) {$(3,-1)$};
\node (8) at (6,-1) {$(4,-1)$};

\node (9) at (2,-2) {$(2,-2)$};
\node (10) at (0,-3) {$(1,-3)$};
\node (11) at (4,-3) {$(3,-3)$};
\node (12) at (6,-3) {$(4,-3)$};

\node (13) at (2,-4) {$(2,-4)$};
\node (14) at (0,-5) {$(1,-5)$};
\node (15) at (4,-5) {$(3,-5)$};
\node (16) at (6,-5) {$(4,-5)$};

\node (17) at (2,-6) {$(2,-6)$};

\node (box) [draw=red,rounded corners,fit = (6)] {};

\draw [->] (9) edge (5);
\draw [->] (13) edge (9);
\draw [->] (13) edge (17);

\draw [->] (10) edge (6);
\draw [->] (14) edge (10);

\draw [->] (11) edge (7);
\draw [->] (15) edge (11);

\draw [->] (12) edge (8);
\draw [->] (16) edge (12);

\draw [->] (6) edge (5);
\draw [->] (7) edge (5);
\draw [->] (8) edge (5);

\draw [->] (10) edge (9);
\draw [->] (11) edge (9);
\draw [->] (12) edge (9);

\draw [->] (14) edge (13);
\draw [->] (15) edge (13);
\draw [->] (16) edge (13);

\draw [->] (5) edge (2);
\draw [->] (5) edge (3);
\draw [->] (5) edge (4);




\draw [->] (9) edge (14);
\draw [->] (9) edge (15);
\draw [->] (9) edge (16);


\draw [->] (2) to[out=270, in=180, looseness=0.8] (7);
\draw [->] (2) to[out=350, in=150, looseness=0.8] (8);

\draw [->] (3) to[out=270, in=0, looseness=0.8] (6);
\draw [->] (3) to (8);

\draw [->] (4) to[out=220, in=0, looseness=0.8] (6);
\draw [->] (4) to (7);

\draw [->] (5) to (10);
\draw [->] (5) to (11);
\draw [->] (5) to (12);

\draw [->] (1) to[out=160, in=200, looseness=1.5] (13);
\draw [->] (17) to[out=200, in=140, looseness=1.5] (1);
\end{tikzpicture}
\end{center}
\end{minipage}
\begin{minipage}[c]{4cm}
Mutation at $(1,-1)$.
\vspace*{1em}

\begin{equation*}
 \deg\left( z_{1,-1}^{(1)} \right) = e_3 + e_4.
\end{equation*}
\end{minipage}

\begin{align*}
	z_{1,-1}^{(1)} &  = z_{1,-3}\left(z_{1,-1}\right)^{-1}f_3 f_4 \\		& + \left(z_{1,-1}\right)^{-1}z_{2,-2}\left(z_{2,0}\right)^{-1}f_1 f_3 f_4 + z_{3,-1}z_{4,-1}\left(z_{2,0}\right)^{-1}f_2 \\
z_{1,-1}^{(1)} 	& = \mathcal{J}\left( Y_{1,-3} + \left(Y_{1,-1}\right)^{-1}Y_{2,-2} + \left(Y_{2,0}\right)^{-1}Y_{3,-1}Y_{4,-1}\right)f_3 f_4.
 \end{align*}

\vspace*{2em}

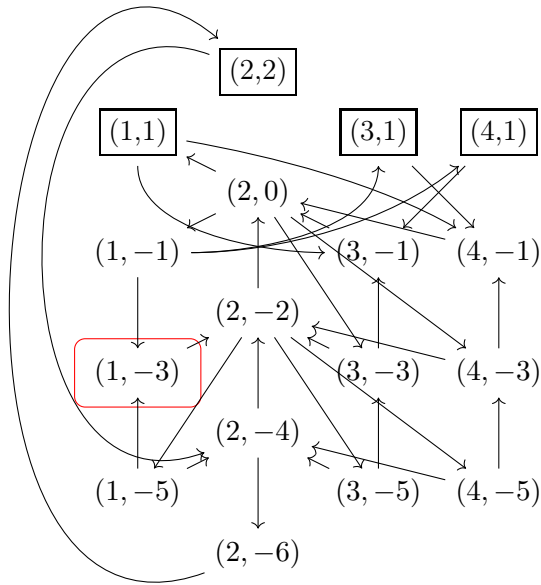
\begin{minipage}{10cm}
\begin{center}
\begin{tikzpicture}[scale=0.8]
\path[use as bounding box] (-2, 3) rectangle (7, -7);
\node (1) at (2,2) {\boxed{$(2,2)$}};
\node (2) at (0,1) {\boxed{$(1,1)$}};
\node (3) at (4,1) {\boxed{$(3,1)$}};
\node (4) at (6,1) {\boxed{$(4,1)$}};

\node (5) at (2,0) {$(2,0)$};
\node (6) at (0,-1) {$(1,-1)$};
\node (7) at (4,-1) {$(3,-1)$};
\node (8) at (6,-1) {$(4,-1)$};

\node (9) at (2,-2) {$(2,-2)$};
\node (10) at (0,-3) {$(1,-3)$};
\node (11) at (4,-3) {$(3,-3)$};
\node (12) at (6,-3) {$(4,-3)$};

\node (13) at (2,-4) {$(2,-4)$};
\node (14) at (0,-5) {$(1,-5)$};
\node (15) at (4,-5) {$(3,-5)$};
\node (16) at (6,-5) {$(4,-5)$};

\node (17) at (2,-6) {$(2,-6)$};

\node (box) [draw=red,rounded corners,fit = (10)] {};

\draw [->] (9) edge (5);
\draw [->] (13) edge (9);
\draw [->] (13) edge (17);

\draw [->] (6) edge (10);
\draw [->] (14) edge (10);

\draw [->] (11) edge (7);
\draw [->] (15) edge (11);

\draw [->] (12) edge (8);
\draw [->] (16) edge (12);

\draw [->] (5) edge (6);
\draw [->] (7) edge (5);
\draw [->] (8) edge (5);

\draw [->] (10) edge (9);
\draw [->] (11) edge (9);
\draw [->] (12) edge (9);

\draw [->] (14) edge (13);
\draw [->] (15) edge (13);
\draw [->] (16) edge (13);

\draw [->] (5) edge (2);




\draw [->] (9) edge (14);
\draw [->] (9) edge (15);
\draw [->] (9) edge (16);


\draw [->] (2) to[out=270, in=180, looseness=0.8] (7);
\draw [->] (2) to[out=350, in=150, looseness=0.8] (8);

\draw [->] (6) to[out=0, in=270, looseness=0.8] (3);
\draw [->] (3) to (8);

\draw [->] (6) to[out=0, in=220, looseness=0.8] (4);
\draw [->] (4) to (7);

\draw [->] (5) to (11);
\draw [->] (5) to (12);

\draw [->] (1) to[out=160, in=200, looseness=1.5] (13);
\draw [->] (17) to[out=200, in=140, looseness=1.5] (1);
\end{tikzpicture}
\end{center}
\end{minipage}
\begin{minipage}[c]{4cm}
Mutation at $(1,-3)$.
\vspace*{1em}

\begin{equation*}
\deg\left( z_{1,-3}^{(1)} \right) = e_3 + e_4.
\end{equation*}
\end{minipage}

\begin{align*}
	z_{1,-3}^{(1)} &  = z_{1,-5}\left(z_{1,-1}\right)^{-1}f_3 f_4  + z_{1,-5}\left(z_{1,-3}\right)^{-1}\left(z_{1,-1}\right)^{-1}z_{2,-2}\left(z_{2,0}\right)^{-1}f_1 f_3 f_4 \\
	& + z_{1,-5}\left(z_{1,-3}\right)^{-1}\left(z_{2,0}\right)^{-1}z_{3,-1}z_{4,-1}f_2 + \left(z_{1,-3}\right)^{-1}z_{2,-4}\left(z_{2,0}\right)^{-1}f_1 f_3 f_4 \\
	& + \left(z_{1,-3}\right)^{-1}z_{1,-1}\left(z_{2,-2}\right)^{-1} \left(z_{2,0}\right)^{-1}z_{3,-1}z_{4,-1}z_{2,-4} f_2 + \left(z_{2,-2}\right)^{-1}z_{3,-3}z_{4,-3}  f_2,\\
z_{1,-3}^{(1)}	& = \mathcal{J}\left(Y_{1,-5}Y_{1,-3} + Y_{1,-5}\left(Y_{1,-1}\right)^{-1}Y_{2,-2} + Y_{1,-5}\left(Y_{2,0}\right)^{-1}Y_{3,-1}Y_{4,-1} \right.\\
	& + \left(Y_{1,-3}\right)^{-1}\left(Y_{1,-1}\right)^{-1}Y_{2,-4}Y_{2,0} + \left(Y_{1,-3}\right)^{-1}Y_{2,-4}\left(Y_{2,0}\right)^{-1}Y_{3,-1}Y_{4,-1} \\
	& \left. + \left(Y_{2,-2}\right)^{-1}\left(Y_{2,0}\right)^{-1}Y_{3,-3}Y_{3,-1}Y_{4,-3}Y_{4,-1}\right)f_3 f_4.
\end{align*}

\vspace*{2em}

The next steps of the mutation process are the mutations $(3,-1)$,$(3,-3)$ and $(4,-1)$, $(4,-3)$. However, the resulting cluster variables are equal $z_{1,-1}^{(1)}$ $z_{1,-3}^{(1)}$, up to the exchanges $1 \longleftrightarrow 3$ and $1 \longleftrightarrow 4$, so we will not give their expressions. We only show the mutation of the quiver, and the multi-degrees.

\vspace*{5em}

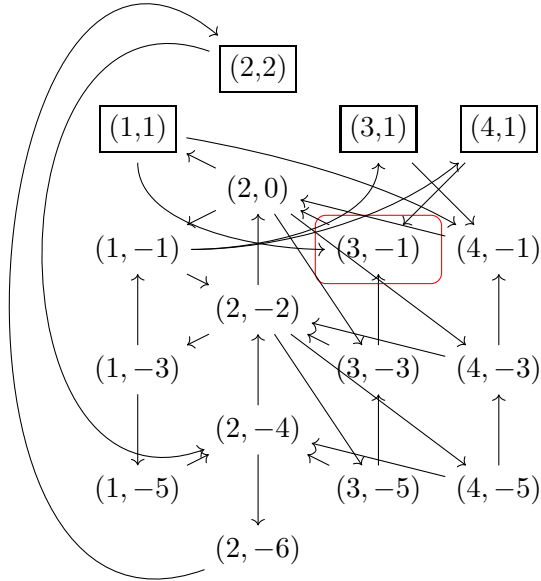
\begin{minipage}{10cm}
\begin{center}
\begin{tikzpicture}[scale=0.8]
\path[use as bounding box] (-2, 3) rectangle (7, -7);
\node (1) at (2,2) {\boxed{$(2,2)$}};
\node (2) at (0,1) {\boxed{$(1,1)$}};
\node (3) at (4,1) {\boxed{$(3,1)$}};
\node (4) at (6,1) {\boxed{$(4,1)$}};

\node (5) at (2,0) {$(2,0)$};
\node (6) at (0,-1) {$(1,-1)$};
\node (7) at (4,-1) {$(3,-1)$};
\node (8) at (6,-1) {$(4,-1)$};

\node (9) at (2,-2) {$(2,-2)$};
\node (10) at (0,-3) {$(1,-3)$};
\node (11) at (4,-3) {$(3,-3)$};
\node (12) at (6,-3) {$(4,-3)$};

\node (13) at (2,-4) {$(2,-4)$};
\node (14) at (0,-5) {$(1,-5)$};
\node (15) at (4,-5) {$(3,-5)$};
\node (16) at (6,-5) {$(4,-5)$};

\node (17) at (2,-6) {$(2,-6)$};

\node (box) [draw=red,rounded corners,fit = (7)] {};

\draw [->] (9) edge (5);
\draw [->] (13) edge (9);
\draw [->] (13) edge (17);

\draw [->] (10) edge (6);
\draw [->] (10) edge (14);

\draw [->] (11) edge (7);
\draw [->] (15) edge (11);

\draw [->] (12) edge (8);
\draw [->] (16) edge (12);

\draw [->] (5) edge (6);
\draw [->] (7) edge (5);
\draw [->] (8) edge (5);

\draw [->] (9) edge (10);
\draw [->] (11) edge (9);
\draw [->] (12) edge (9);

\draw [->] (14) edge (13);
\draw [->] (15) edge (13);
\draw [->] (16) edge (13);

\draw [->] (5) edge (2);

\draw [->] (6) edge (9);



\draw [->] (9) edge (15);
\draw [->] (9) edge (16);


\draw [->] (2) to[out=270, in=180, looseness=0.8] (7);
\draw [->] (2) to[out=350, in=150, looseness=0.8] (8);

\draw [->] (6) to[out=0, in=270, looseness=0.8] (3);
\draw [->] (3) to (8);

\draw [->] (6) to[out=0, in=220, looseness=0.8] (4);
\draw [->] (4) to (7);

\draw [->] (5) to (11);
\draw [->] (5) to (12);

\draw [->] (1) to[out=160, in=200, looseness=1.5] (13);
\draw [->] (17) to[out=200, in=140, looseness=1.5] (1);
\end{tikzpicture}
\end{center}
\end{minipage}
\begin{minipage}[c]{4cm}
Mutation at $(3,-1)$.
\vspace*{1em}

\begin{equation*}
\deg\left( z_{3,-1}^{(1)} \right) = e_1 + e_4.
\end{equation*}
\end{minipage}



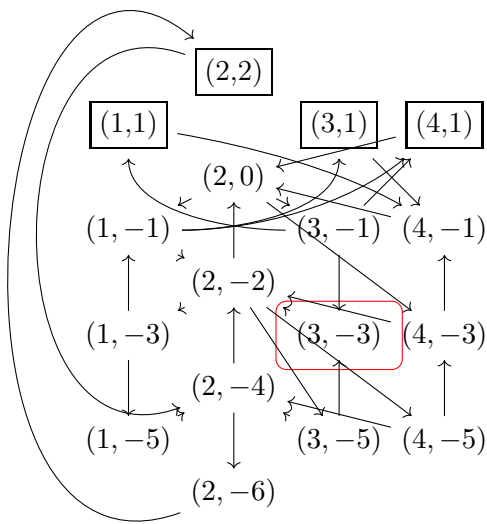
\begin{minipage}{10cm}
\begin{center}
\begin{tikzpicture}[scale=0.7]
\path[use as bounding box] (-2, 3) rectangle (7, -7);
\node (1) at (2,2) {\boxed{$(2,2)$}};
\node (2) at (0,1) {\boxed{$(1,1)$}};
\node (3) at (4,1) {\boxed{$(3,1)$}};
\node (4) at (6,1) {\boxed{$(4,1)$}};

\node (5) at (2,0) {$(2,0)$};
\node (6) at (0,-1) {$(1,-1)$};
\node (7) at (4,-1) {$(3,-1)$};
\node (8) at (6,-1) {$(4,-1)$};

\node (9) at (2,-2) {$(2,-2)$};
\node (10) at (0,-3) {$(1,-3)$};
\node (11) at (4,-3) {$(3,-3)$};
\node (12) at (6,-3) {$(4,-3)$};

\node (13) at (2,-4) {$(2,-4)$};
\node (14) at (0,-5) {$(1,-5)$};
\node (15) at (4,-5) {$(3,-5)$};
\node (16) at (6,-5) {$(4,-5)$};

\node (17) at (2,-6) {$(2,-6)$};

\node (box) [draw=red,rounded corners,fit = (11)] {};

\draw [->] (9) edge (5);
\draw [->] (13) edge (9);
\draw [->] (13) edge (17);

\draw [->] (10) edge (6);
\draw [->] (10) edge (14);

\draw [->] (7) edge (11);
\draw [->] (15) edge (11);

\draw [->] (12) edge (8);
\draw [->] (16) edge (12);

\draw [->] (5) edge (6);
\draw [->] (5) edge (7);
\draw [->] (8) edge (5);

\draw [->] (9) edge (10);
\draw [->] (11) edge (9);
\draw [->] (12) edge (9);

\draw [->] (14) edge (13);
\draw [->] (15) edge (13);
\draw [->] (16) edge (13);

\draw [->] (4) edge (5);

\draw [->] (6) edge (9);



\draw [->] (9) edge (15);
\draw [->] (9) edge (16);


\draw [->] (7) to[out=180, in=270, looseness=0.8] (2);
\draw [->] (2) to[out=350, in=150, looseness=0.8] (8);

\draw [->] (6) to[out=0, in=270, looseness=0.8] (3);
\draw [->] (3) to (8);

\draw [->] (6) to[out=0, in=220, looseness=0.8] (4);
\draw [->] (7) to (4);

\draw [->] (5) to (12);

\draw [->] (1) to[out=160, in=200, looseness=1.5] (13);
\draw [->] (17) to[out=200, in=140, looseness=1.5] (1);
\end{tikzpicture}
\end{center}
\end{minipage}
\begin{minipage}[c]{4cm}
Mutation at $(3,-3)$.
\vspace*{1em}

\begin{equation*}
\deg\left( z_{3,-3}^{(1)} \right) = e_1 + e_4.
\end{equation*}
\end{minipage}



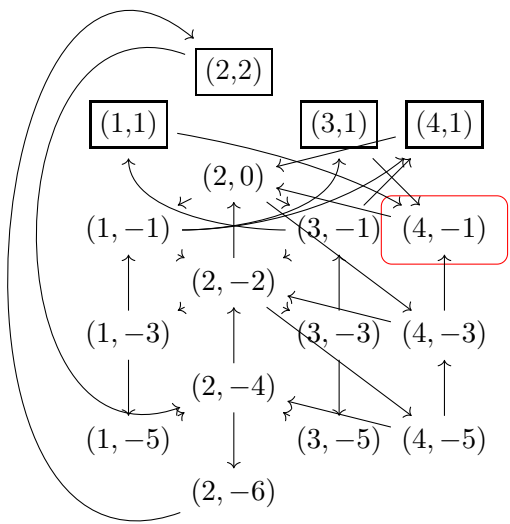
\begin{minipage}{10cm}
\begin{center}
\begin{tikzpicture}[scale=0.7]
\path[use as bounding box] (-2, 3) rectangle (7, -7);
\node (1) at (2,2) {\boxed{$(2,2)$}};
\node (2) at (0,1) {\boxed{$(1,1)$}};
\node (3) at (4,1) {\boxed{$(3,1)$}};
\node (4) at (6,1) {\boxed{$(4,1)$}};

\node (5) at (2,0) {$(2,0)$};
\node (6) at (0,-1) {$(1,-1)$};
\node (7) at (4,-1) {$(3,-1)$};
\node (8) at (6,-1) {$(4,-1)$};

\node (9) at (2,-2) {$(2,-2)$};
\node (10) at (0,-3) {$(1,-3)$};
\node (11) at (4,-3) {$(3,-3)$};
\node (12) at (6,-3) {$(4,-3)$};

\node (13) at (2,-4) {$(2,-4)$};
\node (14) at (0,-5) {$(1,-5)$};
\node (15) at (4,-5) {$(3,-5)$};
\node (16) at (6,-5) {$(4,-5)$};

\node (17) at (2,-6) {$(2,-6)$};

\node (box) [draw=red,rounded corners,fit = (8)] {};

\draw [->] (9) edge (5);
\draw [->] (13) edge (9);
\draw [->] (13) edge (17);

\draw [->] (10) edge (6);
\draw [->] (10) edge (14);

\draw [->] (11) edge (7);
\draw [->] (11) edge (15);

\draw [->] (12) edge (8);
\draw [->] (16) edge (12);

\draw [->] (5) edge (6);
\draw [->] (5) edge (7);
\draw [->] (8) edge (5);

\draw [->] (9) edge (10);
\draw [->] (9) edge (11);
\draw [->] (12) edge (9);

\draw [->] (14) edge (13);
\draw [->] (15) edge (13);
\draw [->] (16) edge (13);

\draw [->] (4) edge (5);

\draw [->] (6) edge (9);
\draw [->] (7) edge (9);



\draw [->] (9) edge (16);


\draw [->] (7) to[out=180, in=270, looseness=0.8] (2);
\draw [->] (2) to[out=350, in=150, looseness=0.8] (8);

\draw [->] (6) to[out=0, in=270, looseness=0.8] (3);
\draw [->] (3) to (8);

\draw [->] (6) to[out=0, in=220, looseness=0.8] (4);
\draw [->] (7) to (4);

\draw [->] (5) to (12);

\draw [->] (1) to[out=160, in=200, looseness=1.5] (13);
\draw [->] (17) to[out=200, in=140, looseness=1.5] (1);
\end{tikzpicture}
\end{center}
\end{minipage}
\begin{minipage}[c]{4cm}
Mutation at $(4,-1)$.
\vspace*{1em}

\begin{equation*}
\deg\left( z_{4,-1}^{(1)} \right) = e_1 + e_3.
\end{equation*}
\end{minipage}




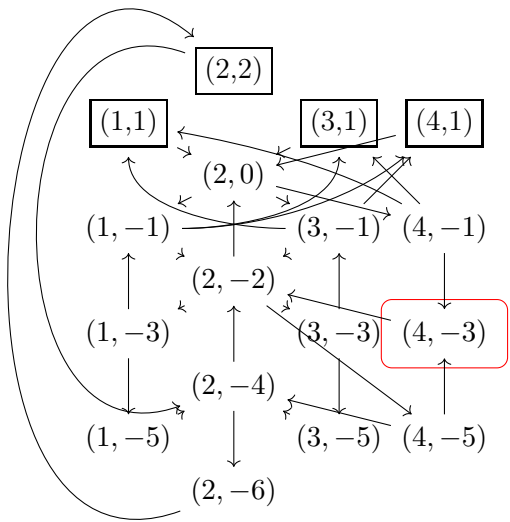
\begin{minipage}{10cm}
\begin{center}
\begin{tikzpicture}[scale=0.7]
\path[use as bounding box] (-2, 3) rectangle (7, -7);
\node (1) at (2,2) {\boxed{$(2,2)$}};
\node (2) at (0,1) {\boxed{$(1,1)$}};
\node (3) at (4,1) {\boxed{$(3,1)$}};
\node (4) at (6,1) {\boxed{$(4,1)$}};

\node (5) at (2,0) {$(2,0)$};
\node (6) at (0,-1) {$(1,-1)$};
\node (7) at (4,-1) {$(3,-1)$};
\node (8) at (6,-1) {$(4,-1)$};

\node (9) at (2,-2) {$(2,-2)$};
\node (10) at (0,-3) {$(1,-3)$};
\node (11) at (4,-3) {$(3,-3)$};
\node (12) at (6,-3) {$(4,-3)$};

\node (13) at (2,-4) {$(2,-4)$};
\node (14) at (0,-5) {$(1,-5)$};
\node (15) at (4,-5) {$(3,-5)$};
\node (16) at (6,-5) {$(4,-5)$};

\node (17) at (2,-6) {$(2,-6)$};

\node (box) [draw=red,rounded corners,fit = (12)] {};

\draw [->] (9) edge (5);
\draw [->] (13) edge (9);
\draw [->] (13) edge (17);

\draw [->] (10) edge (6);
\draw [->] (10) edge (14);

\draw [->] (11) edge (7);
\draw [->] (11) edge (15);

\draw [->] (8) edge (12);
\draw [->] (16) edge (12);

\draw [->] (5) edge (6);
\draw [->] (5) edge (7);
\draw [->] (5) edge (8);

\draw [->] (9) edge (10);
\draw [->] (9) edge (11);
\draw [->] (12) edge (9);

\draw [->] (14) edge (13);
\draw [->] (15) edge (13);
\draw [->] (16) edge (13);

\draw [->] (2) edge (5);
\draw [->] (3) edge (5);
\draw [->] (4) edge (5);

\draw [->] (6) edge (9);
\draw [->] (7) edge (9);



\draw [->] (9) edge (16);


\draw [->] (7) to[out=180, in=270, looseness=0.8] (2);
\draw [->] (8) to[out=150, in=350, looseness=0.8] (2);

\draw [->] (6) to[out=0, in=270, looseness=0.8] (3);
\draw [->] (8) to (3);

\draw [->] (6) to[out=0, in=220, looseness=0.8] (4);
\draw [->] (7) to (4);


\draw [->] (1) to[out=160, in=200, looseness=1.5] (13);
\draw [->] (17) to[out=200, in=140, looseness=1.5] (1);
\end{tikzpicture}
\end{center}
\end{minipage}
\begin{minipage}[c]{4cm}
Mutation at $(4,-3)$.
\vspace*{1em}

\begin{equation*}
\deg\left( z_{4,-3}^{(1)} \right) = e_1 + e_3.
\end{equation*}
\end{minipage}



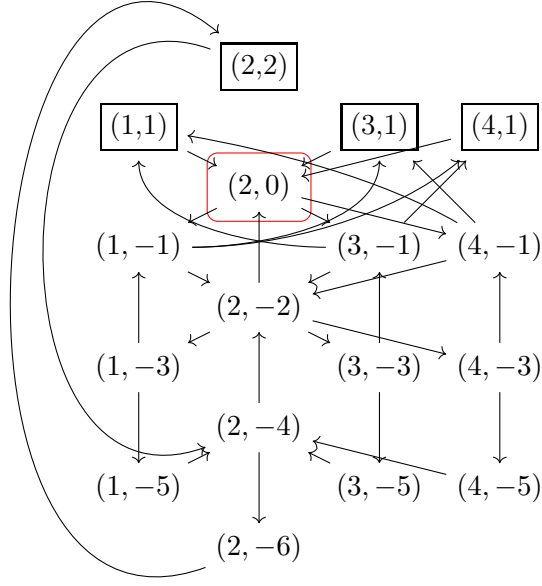
\begin{minipage}{10cm}
\begin{center}
\begin{tikzpicture}[scale=0.8]
\path[use as bounding box] (-2, 3) rectangle (7, -7);
\node (1) at (2,2) {\boxed{$(2,2)$}};
\node (2) at (0,1) {\boxed{$(1,1)$}};
\node (3) at (4,1) {\boxed{$(3,1)$}};
\node (4) at (6,1) {\boxed{$(4,1)$}};

\node (5) at (2,0) {$(2,0)$};
\node (6) at (0,-1) {$(1,-1)$};
\node (7) at (4,-1) {$(3,-1)$};
\node (8) at (6,-1) {$(4,-1)$};

\node (9) at (2,-2) {$(2,-2)$};
\node (10) at (0,-3) {$(1,-3)$};
\node (11) at (4,-3) {$(3,-3)$};
\node (12) at (6,-3) {$(4,-3)$};

\node (13) at (2,-4) {$(2,-4)$};
\node (14) at (0,-5) {$(1,-5)$};
\node (15) at (4,-5) {$(3,-5)$};
\node (16) at (6,-5) {$(4,-5)$};

\node (17) at (2,-6) {$(2,-6)$};

\node (box) [draw=red,rounded corners,fit = (5)] {};

\draw [->] (9) edge (5);
\draw [->] (13) edge (9);
\draw [->] (13) edge (17);

\draw [->] (10) edge (6);
\draw [->] (10) edge (14);

\draw [->] (11) edge (7);
\draw [->] (11) edge (15);

\draw [->] (12) edge (8);
\draw [->] (12) edge (16);

\draw [->] (5) edge (6);
\draw [->] (5) edge (7);
\draw [->] (5) edge (8);

\draw [->] (9) edge (10);
\draw [->] (9) edge (11);
\draw [->] (9) edge (12);

\draw [->] (14) edge (13);
\draw [->] (15) edge (13);
\draw [->] (16) edge (13);

\draw [->] (2) edge (5);
\draw [->] (3) edge (5);
\draw [->] (4) edge (5);

\draw [->] (6) edge (9);
\draw [->] (7) edge (9);
\draw [->] (8) edge (9);





\draw [->] (7) to[out=180, in=270, looseness=0.8] (2);
\draw [->] (8) to[out=150, in=350, looseness=0.8] (2);

\draw [->] (6) to[out=0, in=270, looseness=0.8] (3);
\draw [->] (8) to (3);

\draw [->] (6) to[out=0, in=220, looseness=0.8] (4);
\draw [->] (7) to (4);


\draw [->] (1) to[out=160, in=200, looseness=1.5, /tikz/overlay] (13);
\draw [->] (17) to[out=200, in=140, looseness=1.5, /tikz/overlay] (1);
\end{tikzpicture}
\end{center}
\end{minipage}
\begin{minipage}[c]{4cm}
Mutation at $(2,0)$.
\vspace*{1em}

\begin{equation*}
\deg(z_{2,0}^{(2)}) = e_1 + e_3 + e_4.
\end{equation*}
\end{minipage}

\begin{align*}
	z_{2,0}^{(2)} &  = z_{2,-4}\left( z_{2,-2}\right)^{-1}f_1 f_3 f_4 \\
	& +  z_{1,-3}\left( z_{1,-1}\right)^{-1} \left( z_{2,-2}\right)^{-1} z_{2,0} z_{3,-3}\left( z_{3,-1}\right)^{-1}z_{4,-3}\left( z_{4,-1}\right)^{-1} f_1 f_3 f_4 \\
	& + \left(z_{1,-1}\right)^{-1}z_{3,-3}\left(z_{3,-1}\right)^{-1}z_{4,-3}\left(z_{4,-1}\right)^{-1}f_1^2 f_3 f_4 \\
	& + z_{1,-3}\left(z_{1,-1}\right)^{-1}\left(z_{3,-1}\right)^{-1} z_{4,-3}\left(z_{4,-1}\right)^{-1}f_1 f_3^2 f_4 \\
	& + z_{1,-3}\left(z_{1,-1}\right)^{-1}z_{3,-3}\left(z_{3,-1}\right)^{-1}\left(z_{4,-1}\right)^{-1}f_1 f_3 f_4^2 \\
	& + \left( z_{1,-1}\right)^{-1}  z_{2,-2}\left( z_{2,0}\right)^{-1}\left( z_{3,-1}\right)^{-1}  z_{4,-3}\left( z_{4,-1}\right)^{-1} z_{4,1}f_1^2 f_3^2 f_4\\
	& + \left( z_{1,-1}\right)^{-1}  z_{2,-2}\left( z_{2,0}\right)^{-1} z_{3,-3}\left( z_{3,-1}\right)^{-1} \left( z_{4,-1}\right)^{-1} f_1^2 f_3 f_4^2 \\
	& +  z_{1,-3}\left( z_{1,-1}\right)^{-1}  z_{2,-2}\left( z_{2,0}\right)^{-1}\left( z_{3,-1}\right)^{-1} \left( z_{4,-1}\right)^{-1} f_1 f_3^2 f_4^2 \\
	& + \left(z_{2,0}\right)^{-1}z_{4,-3}f_1 f_2 f_3 +  \left(z_{2,0}\right)^{-1}z_{3,-3}f_1 f_2 f_4 +  z_{1,-3}\left(z_{2,0}\right)^{-1}f_2 f_3 f_4 \\
	& +\left( z_{1,-1}\right)^{-1} z_{2,-2}^2\left( z_{2,0}\right)^{-2}\left( z_{3,-1}\right)^{-1} \left( z_{4,-1}\right)^{-1} f_1^2 f_3^2 f_4^2\\
	& + \left( {\color{red} t+t^{-1}} \right)  z_{2,-2}\left( z_{2,0}\right)^{-2} f_1 f_2 f_3 f_4 + z_{1,-1}\left(z_{2,0}\right)^{-2} z_{3,-1}z_{4,-1}f_2^2, 
\end{align*}

\begin{align*}
z_{2,0}^{(2)} & =\mathcal{J}  \left( Y_{2,-4} + Y_{1,-3}\left(Y_{2,-2}\right)^{-1}Y_{3,-3}Y_{4,-3} \right.\\
	& + \left(Y_{1,-1}\right)^{-1}Y_{3,-3}Y_{4,-3}  + Y_{1,-3}\left(Y_{3,-1}\right)^{-1}Y_{4,-3} + Y_{1,-3}Y_{3,-3}\left(Y_{4,-1}\right)^{-1} \\
	& + \left(Y_{1,-1}\right)^{-1} Y_{2,-2} \left(Y_{3,-1}\right)^{-1}Y_{4,-3} + \left(Y_{1,-1}\right)^{-1} Y_{2,-2} Y_{3,-3}\left(Y_{4,-1}\right)^{-1} \\
	&  + Y_{1,-3}Y_{2,-2} \left(Y_{3,-1}\right)^{-1} \left(Y_{4,-1}\right)^{-1} + \left(Y_{2,0}\right)^{-1}Y_{4,-3}Y_{4,-1} \\
	& + \left(Y_{2,0}\right)^{-1}Y_{3,-3}Y_{3,-1} + Y_{1,-3}Y_{1,-1}\left(Y_{2,0}\right)^{-1} \\
	& + \left(Y_{1,-1}\right)^{-1}\left(Y_{2,-2}\right)^{2}\left(Y_{3,-1}\right)^{-1} \left(Y_{4,-1}\right)^{-1} + \left( {\color{red} t+t^{-1}} \right) Y_{2,-2}\left(Y_{2,0}\right)^{-1}   \\
	&\left. + Y_{1,-1}\left(Y_{2,0}\right)^{-2}Y_{3,-1}Y_{4,-1}\right)f_1 f_3 f_4.
\end{align*}

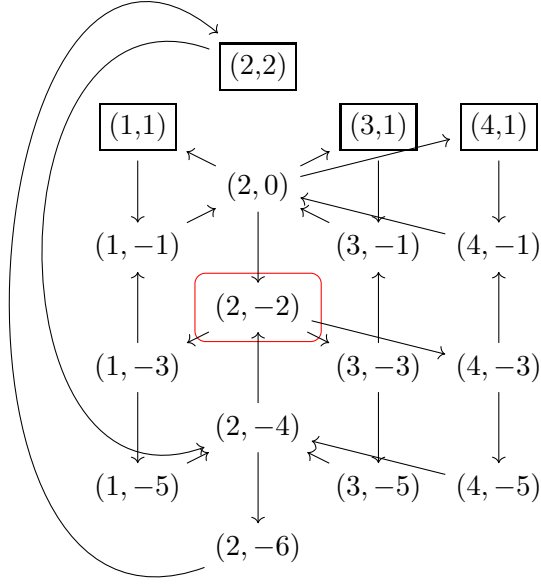
\begin{minipage}{10cm}
\begin{center}
\begin{tikzpicture}[scale=0.8]
\path[use as bounding box] (-2, 3) rectangle (7, -7);
\node (1) at (2,2) {\boxed{$(2,2)$}};
\node (2) at (0,1) {\boxed{$(1,1)$}};
\node (3) at (4,1) {\boxed{$(3,1)$}};
\node (4) at (6,1) {\boxed{$(4,1)$}};

\node (5) at (2,0) {$(2,0)$};
\node (6) at (0,-1) {$(1,-1)$};
\node (7) at (4,-1) {$(3,-1)$};
\node (8) at (6,-1) {$(4,-1)$};

\node (9) at (2,-2) {$(2,-2)$};
\node (10) at (0,-3) {$(1,-3)$};
\node (11) at (4,-3) {$(3,-3)$};
\node (12) at (6,-3) {$(4,-3)$};

\node (13) at (2,-4) {$(2,-4)$};
\node (14) at (0,-5) {$(1,-5)$};
\node (15) at (4,-5) {$(3,-5)$};
\node (16) at (6,-5) {$(4,-5)$};

\node (17) at (2,-6) {$(2,-6)$};

\node (box) [draw=red,rounded corners,fit = (9)] {};

\draw [->] (5) edge (9);
\draw [->] (13) edge (9);
\draw [->] (13) edge (17);

\draw [->] (2) edge (6);
\draw [->] (10) edge (6);
\draw [->] (10) edge (14);

\draw [->] (3) edge (7);
\draw [->] (11) edge (7);
\draw [->] (11) edge (15);

\draw [->] (4) edge (8);
\draw [->] (12) edge (8);
\draw [->] (12) edge (16);

\draw [->] (6) edge (5);
\draw [->] (7) edge (5);
\draw [->] (8) edge (5);

\draw [->] (9) edge (10);
\draw [->] (9) edge (11);
\draw [->] (9) edge (12);

\draw [->] (14) edge (13);
\draw [->] (15) edge (13);
\draw [->] (16) edge (13);

\draw [->] (5) edge (2);
\draw [->] (5) edge (3);
\draw [->] (5) edge (4);










\draw [->] (1) to[out=160, in=200, looseness=1.5] (13);
\draw [->] (17) to[out=200, in=140, looseness=1.5] (1);
\end{tikzpicture}
\end{center}
\end{minipage}
\begin{minipage}[c]{4cm}
Mutation at $(2,-2)$.
\vspace*{2em}

\begin{equation*}
\deg(z_{2,-2}^{(2)}) = e_1 + e_3 + e_4.
\end{equation*}
\end{minipage}

Here we do not write explicitly $z_{2,-2}^{(2)}$, because it has 92 terms. 

\vspace*{2em}

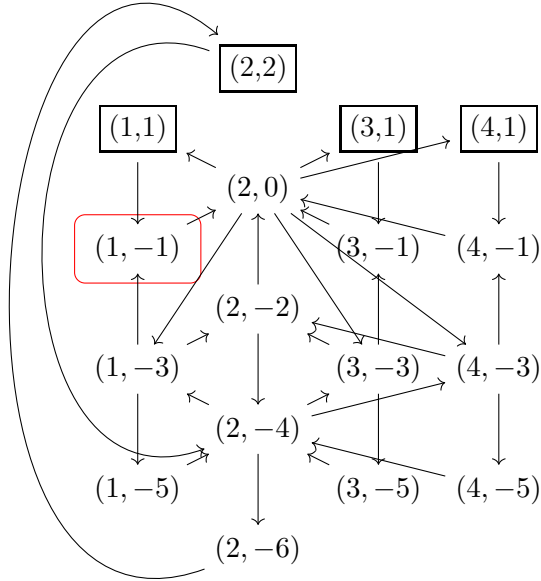
\begin{minipage}{10cm}
\begin{center}
\begin{tikzpicture}[scale=0.8]
\path[use as bounding box] (-2, 3) rectangle (7, -6);
\node (1) at (2,2) {\boxed{$(2,2)$}};
\node (2) at (0,1) {\boxed{$(1,1)$}};
\node (3) at (4,1) {\boxed{$(3,1)$}};
\node (4) at (6,1) {\boxed{$(4,1)$}};

\node (5) at (2,0) {$(2,0)$};
\node (6) at (0,-1) {$(1,-1)$};
\node (7) at (4,-1) {$(3,-1)$};
\node (8) at (6,-1) {$(4,-1)$};

\node (9) at (2,-2) {$(2,-2)$};
\node (10) at (0,-3) {$(1,-3)$};
\node (11) at (4,-3) {$(3,-3)$};
\node (12) at (6,-3) {$(4,-3)$};

\node (13) at (2,-4) {$(2,-4)$};
\node (14) at (0,-5) {$(1,-5)$};
\node (15) at (4,-5) {$(3,-5)$};
\node (16) at (6,-5) {$(4,-5)$};

\node (17) at (2,-6) {$(2,-6)$};

\node (box) [draw=red,rounded corners,fit = (6)] {};

\draw [->] (9) edge (5);
\draw [->] (9) edge (13);
\draw [->] (13) edge (17);

\draw [->] (2) edge (6);
\draw [->] (10) edge (6);
\draw [->] (10) edge (14);

\draw [->] (3) edge (7);
\draw [->] (11) edge (7);
\draw [->] (11) edge (15);

\draw [->] (4) edge (8);
\draw [->] (12) edge (8);
\draw [->] (12) edge (16);

\draw [->] (6) edge (5);
\draw [->] (7) edge (5);
\draw [->] (8) edge (5);

\draw [->] (10) edge (9);
\draw [->] (11) edge (9);
\draw [->] (12) edge (9);

\draw [->] (14) edge (13);
\draw [->] (15) edge (13);
\draw [->] (16) edge (13);

\draw [->] (5) edge (2);
\draw [->] (5) edge (3);
\draw [->] (5) edge (4);


\draw [->] (13) edge (10);
\draw [->] (13) edge (11);
\draw [->] (13) edge (12);







\draw [->] (5) to (10);
\draw [->] (5) to (11);
\draw [->] (5) to (12);

\draw [->] (1) to[out=160, in=200, looseness=1.5] (13);
\draw [->] (17) to[out=200, in=140, looseness=1.5] (1);
\end{tikzpicture}
\end{center}
\end{minipage}
\begin{minipage}[c]{4cm}
Mutation at $(1,-1)$.
\vspace*{1em}

\begin{equation*}
\deg(z_{1,-1}^{(2)}) = e_1.
\end{equation*}
\end{minipage}

\vspace*{1em}

\begin{align*}
	z_{1,-1}^{(2)} & = z_{1,-5}\left(z_{1,-3}\right)^{-1}f_1 + \left(z_{1,-3}\right)^{-1}z_{1,-1}z_{2,-4} \left( z_{2,-2}\right)^{-1}f_1 \\
	& + \left( z_{2,-2}\right)^{-1}z_{2,0}z_{3,-3}\left( z_{3,-1}\right)^{-1}z_{4,-3}\left( z_{4,-1}\right)^{-1}f_1 \\
	& + z_{3,-3}\left( z_{3,-1}\right)^{-1}\left( z_{4,-1}\right)^{-1}f_1 f_4 + \left( z_{3,-1}\right)^{-1}z_{4,-3}\left( z_{4,-1}\right)^{-1}f_1 f_3 \\
	& + z_{2,-2}\left( z_{2,0}\right)^{-1}\left( z_{3,-1}\right)^{-1}\left( z_{4,-1}\right)^{-1}f_1 f_3 f_4  + z_{1,-1}\left( z_{2,0}\right)^{-1}f_2,
\end{align*}

\begin{align*}
z_{1,-1}^{(2)}	& = \mathcal{J}\left( Y_{1,-5} + \left(Y_{1,-3}\right)^{-1} Y_{2,-4} + \left(Y_{2,-2}\right)^{-1}Y_{3,-3}Y_{4,-3} + Y_{3,-3}\left(Y_{4,-1}\right)^{-1}  \right.\\
	& \left. + \left(Y_{3,-1}\right)^{-1}Y_{4,-3}  + Y_{2,-2}\left(Y_{3,-1}\right)^{-1}\left(Y_{4,-1}\right)^{-1} + Y_{1,-1}\left(Y_{2,0}\right)^{-1}\right)f_1.
\end{align*}

\vspace*{1em}

As before, $z_{3,-1}^{(2)}$ and $z_{4,-1}^{(2)}$ are similar to $z_{1,-1}^{(2)}$ so we do not write them.

\vspace*{2em}

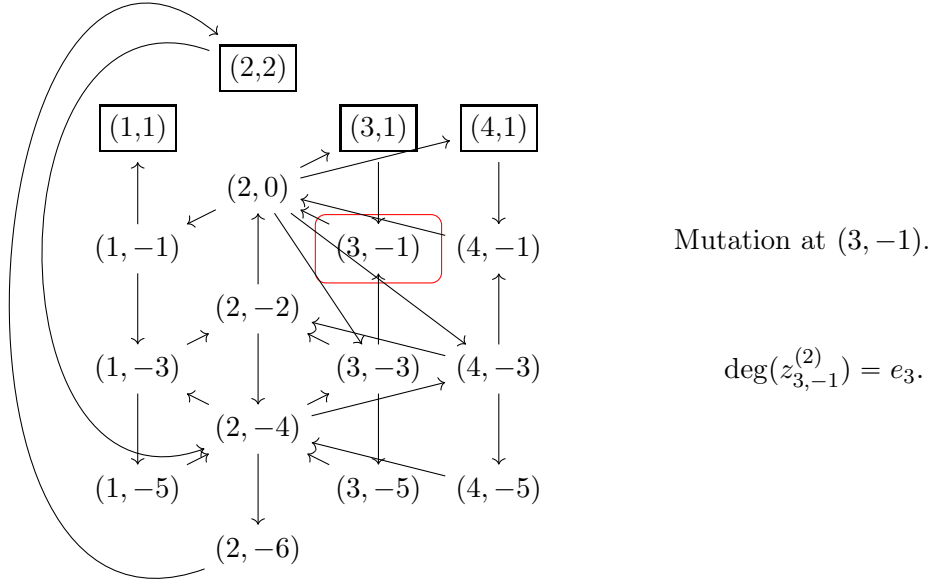
\begin{minipage}{10cm}
\begin{center}
\begin{tikzpicture}[scale=0.8]
\path[use as bounding box] (-2, 3) rectangle (7, -7);
\node (1) at (2,2) {\boxed{$(2,2)$}};
\node (2) at (0,1) {\boxed{$(1,1)$}};
\node (3) at (4,1) {\boxed{$(3,1)$}};
\node (4) at (6,1) {\boxed{$(4,1)$}};

\node (5) at (2,0) {$(2,0)$};
\node (6) at (0,-1) {$(1,-1)$};
\node (7) at (4,-1) {$(3,-1)$};
\node (8) at (6,-1) {$(4,-1)$};

\node (9) at (2,-2) {$(2,-2)$};
\node (10) at (0,-3) {$(1,-3)$};
\node (11) at (4,-3) {$(3,-3)$};
\node (12) at (6,-3) {$(4,-3)$};

\node (13) at (2,-4) {$(2,-4)$};
\node (14) at (0,-5) {$(1,-5)$};
\node (15) at (4,-5) {$(3,-5)$};
\node (16) at (6,-5) {$(4,-5)$};

\node (17) at (2,-6) {$(2,-6)$};

\node (box) [draw=red,rounded corners,fit = (7)] {};

\draw [->] (9) edge (5);
\draw [->] (9) edge (13);
\draw [->] (13) edge (17);

\draw [->] (6) edge (2);
\draw [->] (6) edge (10);
\draw [->] (10) edge (14);

\draw [->] (3) edge (7);
\draw [->] (11) edge (7);
\draw [->] (11) edge (15);

\draw [->] (4) edge (8);
\draw [->] (12) edge (8);
\draw [->] (12) edge (16);

\draw [->] (5) edge (6);
\draw [->] (7) edge (5);
\draw [->] (8) edge (5);

\draw [->] (10) edge (9);
\draw [->] (11) edge (9);
\draw [->] (12) edge (9);

\draw [->] (14) edge (13);
\draw [->] (15) edge (13);
\draw [->] (16) edge (13);

\draw [->] (5) edge (3);
\draw [->] (5) edge (4);


\draw [->] (13) edge (10);
\draw [->] (13) edge (11);
\draw [->] (13) edge (12);







\draw [->] (5) to (11);
\draw [->] (5) to (12);

\draw [->] (1) to[out=160, in=200, looseness=1.5] (13);
\draw [->] (17) to[out=200, in=140, looseness=1.5] (1);
\end{tikzpicture}
\end{center}
\end{minipage}
\begin{minipage}[c]{4cm}
Mutation at $(3,-1)$.
\vspace*{1em}

\begin{equation*}
\deg(z_{3,-1}^{(2)}) = e_3.
\end{equation*}
\end{minipage}

\begin{minipage}{10cm}
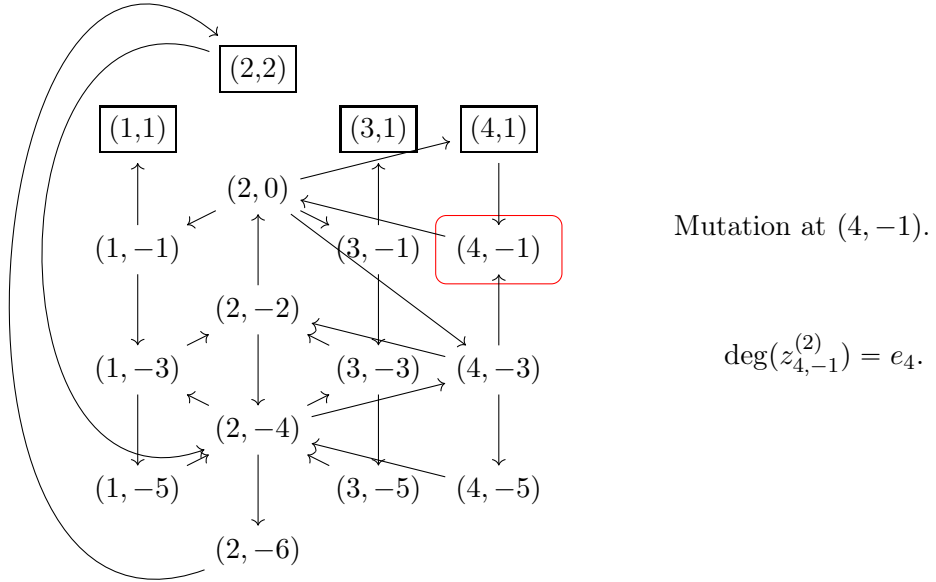
\begin{center}
\begin{tikzpicture}[scale=0.8]
\path[use as bounding box] (-2, 3) rectangle (7, -6.5);
\node (1) at (2,2) {\boxed{$(2,2)$}};
\node (2) at (0,1) {\boxed{$(1,1)$}};
\node (3) at (4,1) {\boxed{$(3,1)$}};
\node (4) at (6,1) {\boxed{$(4,1)$}};

\node (5) at (2,0) {$(2,0)$};
\node (6) at (0,-1) {$(1,-1)$};
\node (7) at (4,-1) {$(3,-1)$};
\node (8) at (6,-1) {$(4,-1)$};

\node (9) at (2,-2) {$(2,-2)$};
\node (10) at (0,-3) {$(1,-3)$};
\node (11) at (4,-3) {$(3,-3)$};
\node (12) at (6,-3) {$(4,-3)$};

\node (13) at (2,-4) {$(2,-4)$};
\node (14) at (0,-5) {$(1,-5)$};
\node (15) at (4,-5) {$(3,-5)$};
\node (16) at (6,-5) {$(4,-5)$};

\node (17) at (2,-6) {$(2,-6)$};

\node (box) [draw=red,rounded corners,fit = (8)] {};

\draw [->] (9) edge (5);
\draw [->] (9) edge (13);
\draw [->] (13) edge (17);

\draw [->] (6) edge (2);
\draw [->] (6) edge (10);
\draw [->] (10) edge (14);

\draw [->] (7) edge (3);
\draw [->] (7) edge (11);
\draw [->] (11) edge (15);

\draw [->] (4) edge (8);
\draw [->] (12) edge (8);
\draw [->] (12) edge (16);

\draw [->] (5) edge (6);
\draw [->] (5) edge (7);
\draw [->] (8) edge (5);

\draw [->] (10) edge (9);
\draw [->] (11) edge (9);
\draw [->] (12) edge (9);

\draw [->] (14) edge (13);
\draw [->] (15) edge (13);
\draw [->] (16) edge (13);

\draw [->] (5) edge (4);


\draw [->] (13) edge (10);
\draw [->] (13) edge (11);
\draw [->] (13) edge (12);







\draw [->] (5) to (12);

\draw [->] (1) to[out=160, in=200, looseness=1.5] (13);
\draw [->] (17) to[out=200, in=140, looseness=1.5] (1);
\end{tikzpicture}
\end{center}
\end{minipage}
\begin{minipage}[c]{4cm}
Mutation at $(4,-1)$.
\vspace*{1em}

\begin{equation*}
\deg(z_{4,-1}^{(2)}) = e_4.
\end{equation*}
\end{minipage}

\begin{minipage}{10cm}
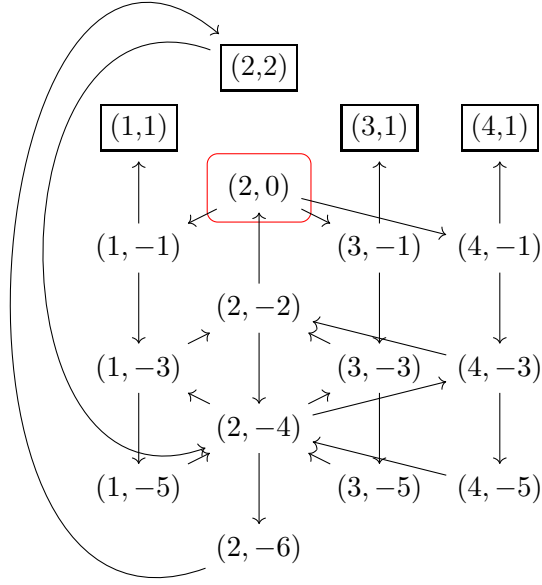
\begin{center}
\begin{tikzpicture}[scale=0.8]
\path[use as bounding box] (-2, 2.5) rectangle (7, -6.5);
\node (1) at (2,2) {\boxed{$(2,2)$}};
\node (2) at (0,1) {\boxed{$(1,1)$}};
\node (3) at (4,1) {\boxed{$(3,1)$}};
\node (4) at (6,1) {\boxed{$(4,1)$}};

\node (5) at (2,0) {$(2,0)$};
\node (6) at (0,-1) {$(1,-1)$};
\node (7) at (4,-1) {$(3,-1)$};
\node (8) at (6,-1) {$(4,-1)$};

\node (9) at (2,-2) {$(2,-2)$};
\node (10) at (0,-3) {$(1,-3)$};
\node (11) at (4,-3) {$(3,-3)$};
\node (12) at (6,-3) {$(4,-3)$};

\node (13) at (2,-4) {$(2,-4)$};
\node (14) at (0,-5) {$(1,-5)$};
\node (15) at (4,-5) {$(3,-5)$};
\node (16) at (6,-5) {$(4,-5)$};

\node (17) at (2,-6) {$(2,-6)$};

\node (box) [draw=red,rounded corners,fit = (5)] {};

\draw [->] (9) edge (5);
\draw [->] (9) edge (13);
\draw [->] (13) edge (17);

\draw [->] (6) edge (2);
\draw [->] (6) edge (10);
\draw [->] (10) edge (14);

\draw [->] (7) edge (3);
\draw [->] (7) edge (11);
\draw [->] (11) edge (15);

\draw [->] (8) edge (4);
\draw [->] (8) edge (12);
\draw [->] (12) edge (16);

\draw [->] (5) edge (6);
\draw [->] (5) edge (7);
\draw [->] (5) edge (8);

\draw [->] (10) edge (9);
\draw [->] (11) edge (9);
\draw [->] (12) edge (9);

\draw [->] (14) edge (13);
\draw [->] (15) edge (13);
\draw [->] (16) edge (13);



\draw [->] (13) edge (10);
\draw [->] (13) edge (11);
\draw [->] (13) edge (12);








\draw [->] (1) to[out=160, in=200, looseness=1.5] (13);
\draw [->] (17) to[out=200, in=140, looseness=1.5] (1);
\end{tikzpicture}
\end{center}
\end{minipage}
\begin{minipage}[c]{4cm}
Mutation at $(2,0)$.
\vspace*{1em}

\begin{equation*}
\deg(z_{2,0}^{(3)}) = 0.
\end{equation*}
\end{minipage}

\begin{align*}
	z_{2,0}^{(3)} & = z_{2,-6}\left(z_{2,-4}\right)^{-1} + z_{1,-5}\left(z_{1,-3}\right)^{-1}\left(z_{2,-4}\right)^{-1} z_{2,-2}z_{3,-5}\left(z_{3,-3}\right)^{-1}z_{4,-5}\left(z_{4,-3}\right)^{-1}   \\
	& + \left(z_{1,-3}\right)^{-1}z_{1,-1}z_{3,-5}\left(z_{3,-3}\right)^{-1}z_{4,-5}\left(z_{4,-3}\right)^{-1} \\
	& + z_{1,-5}\left(z_{1,-3}\right)^{-1}\left(z_{3,-3}\right)^{-1}z_{3,-1}z_{4,-5}\left(z_{4,-3}\right)^{-1} \\
	& + z_{1,-5}\left(z_{1,-3}\right)^{-1}z_{3,-5}\left(z_{3,-3}\right)^{-1}\left(z_{4,-3}\right)^{-1}z_{4,-1} \\
	& + \left(z_{1,-3}\right)^{-1}z_{1,-1}z_{2,-4}\left(z_{2,-2}\right)^{-1}\left(z_{3,-3}\right)^{-1}z_{3,-1}z_{4,-5}\left(z_{4,-3}\right)^{-1}   \\
	& + \left(z_{1,-3}\right)^{-1}z_{1,-1}z_{2,-4}\left(z_{2,-2}\right)^{-1}z_{3,-5}\left(z_{3,-3}\right)^{-1}\left(z_{4,-3}\right)^{-1}z_{4,-1} \\
	& + z_{1,-5}\left(z_{1,-3}\right)^{-1}z_{2,-4}\left(z_{2,-2}\right)^{-1}\left(z_{3,-3}\right)^{-1}z_{3,-1}\left(z_{4,-3}\right)^{-1}z_{4,-1}\\
	& + \left(z_{1,-3}\right)^{-1}z_{1,-1}\left(z_{2,-4}\right)^{2}\left(z_{2,-2}\right)^{-2}\left(z_{3,-3}\right)^{-1}z_{3,-1}\left(z_{4,-3}\right)^{-1}z_{4,-1} \\
	& + \left(z_{2,-2}\right)^{-1}z_{2,0}z_{4,-5}\left(z_{4,-1}\right)^{-1} + \left(z_{2,-2}\right)^{-1}z_{2,0}z_{3,-5}\left(z_{3,-1}\right)^{-1} \\
	& + z_{1,-5}\left(z_{1,-1}\right)^{-1}\left(z_{2,-2}\right)^{-1}z_{2,0}  + \left({\color{red} t + t^{-1}}\right)z_{2,-4}\left(z_{2,-2}\right)^{-2}z_{2,0} \\
	& + z_{1,-3}\left(z_{1,-1}\right)^{-1}\left(z_{2,-2}\right)^{-2}\left(z_{2,0}\right)^{2}z_{3,-3}\left(z_{3,-1}\right)^{-1}z_{4,-3}\left(z_{4,-1}\right)^{-1} \\
		& + z_{4,-5}\left(z_{4,-3}\right)^{-1}\left(z_{4,-1}\right)^{-1}f_4 + z_{3,-5}\left(z_{3,-3}\right)^{-1}\left(z_{3,-1}\right)^{-1}f_3 \\
	& + z_{1,-5}\left(z_{1,-3}\right)^{-1}\left(z_{1,-1}\right)^{-1}f_1 + \left(z_{1,-3}\right)^{-1}z_{2,-4}\left(z_{2,-2}\right)^{-1}f_1\\ 
	&  + z_{2,-4}\left(z_{2,-2}\right)^{-1} \left(z_{3,-3}\right)^{-1}f_3 + z_{2,-4}\left(z_{2,-2}\right)^{-1} \left(z_{4,-3}\right)^{-1}f_4\\
	& + \left(z_{1,-1}\right)^{-1}\left(z_{2,-2}\right)^{-1}z_{2,0}z_{3,-3}\left(z_{3,-1}\right)^{-1}z_{4,-3}\left(z_{4,-1}\right)^{-1}f_1 \\
	& + z_{1,-3}\left(z_{1,-1}\right)^{-1}\left(z_{2,-2}\right)^{-1}z_{2,0}\left(z_{3,-1}\right)^{-1}z_{4,-3}\left(z_{4,-1}\right)^{-1}f_3 \\
	& + z_{1,-3}\left(z_{1,-1}\right)^{-1}\left(z_{2,-2}\right)^{-1}z_{2,0}z_{3,-3}\left(z_{3,-1}\right)^{-1}\left(z_{4,-1}\right)^{-1}f_4 \\
	& + \left(z_{1,-1}\right)^{-1}\left(z_{3,-1}\right)^{-1}z_{4,-3}\left(z_{4,-1}\right)^{-1}f_1 f_3 \\
	& + z_{1,-3}\left(z_{1,-1}\right)^{-1}\left(z_{3,-1}\right)^{-1}\left(z_{4,-1}\right)^{-1}f_3 f_4 \\
	& + \left(z_{1,-1}\right)^{-1}z_{3,-3}\left(z_{3,-1}\right)^{-1}\left(z_{4,-1}\right)^{-1}f_1 f_4 \\
	& + \left(z_{1,-1}\right)^{-1}z_{2,-2}\left(z_{2,0}\right)^{-1}\left(z_{3,-1}\right)^{-1}\left(z_{4,-1}\right)^{-1}f_1 f_3 f_4 + \left(z_{2,0}\right)^{-1}f_2.
\end{align*}

\begin{align*}
	= & \mathcal{J}\left( Y_{2,-6} + Y_{1,-5}\left(Y_{2,-4}\right)^{-1}Y_{3,-5}Y_{4,-5} \right.\\
	& + \left(Y_{1,-3}\right)^{-1}Y_{3,-5}Y_{4,-5}  + Y_{1,-5}\left(Y_{3,-3}\right)^{-1}Y_{4,-5} + Y_{1,-5}Y_{3,-5}\left(Y_{4,-3}\right)^{-1} \\
	& + \left(Y_{1,-3}\right)^{-1}Y_{2,-4}\left(Y_{3,-3}\right)^{-1}Y_{4,-5} + \left(Y_{1,-3}\right)^{-1}Y_{2,-4}Y_{3,-5}\left(Y_{4,-3}\right)^{-1} \\
	& + Y_{1,-5}Y_{2,-4}\left(Y_{3,-3}\right)^{-1}\left(Y_{4,-3}\right)^{-1} + \left(Y_{1,-3}\right)^{-1}\left(Y_{2,-4}\right)^{2}\left(Y_{3,-3}\right)^{-1}\left(Y_{4,-3}\right)^{-1}\\
	& + \left(Y_{2,-2}\right)^{-1}Y_{4,-5}Y_{4,-3} + \left(Y_{2,-2}\right)^{-1}Y_{3,-5}Y_{3,-3} + Y_{1,-5}Y_{1,-3}\left(Y_{2,-2}\right)^{-1}\\
	&  + \left({\color{red} t + t^{-1}}\right)Y_{2,-4}\left(Y_{2,-2}\right)^{-1} +  Y_{1,-3}\left(Y_{2,-2}\right)^{-2}Y_{3,-3}Y_{4,-3}\\
	& + Y_{4,-5}\left(Y_{4,-1}\right)^{-1} + Y_{3,-5}\left(Y_{3,-1}\right)^{-1} + Y_{1,-5}\left(Y_{1,-1}\right)^{-1} \\
	& + \left(Y_{1,-3}\right)^{-1}\left(Y_{1,-1}\right)^{-1}Y_{2,-4} + Y_{2,-4}\left(Y_{3,-3}\right)^{-1}\left(Y_{3,-1}\right)^{-1} + Y_{2,-4}\left(Y_{4,-3}\right)^{-1}\left(Y_{4,-1}\right)^{-1} \\
	& + \left(Y_{1,-1}\right)^{-1}\left(Y_{2,-2}\right)^{-1}Y_{3,-3}Y_{4,-3} + Y_{1,-3}\left(Y_{2,-2}\right)^{-1}\left(Y_{3,-1}\right)^{-1}Y_{4,-3} \\
	& + Y_{1,-3}\left(Y_{2,-2}\right)^{-1}Y_{3,-3}\left(Y_{4,-1}\right)^{-1} + \left(Y_{1,-1}\right)^{-1}\left(Y_{3,-1}\right)^{-1}Y_{4,-3}\\
	& + Y_{1,-3}\left(Y_{3,-1}\right)^{-1}\left(Y_{4,-1}\right)^{-1} + \left(Y_{1,-1}\right)^{-1}Y_{3,-3}\left(Y_{4,-1}\right)^{-1} \\
	& \left.+ \left(Y_{1,-1}\right)^{-1}Y_{2,-2}\left(Y_{3,-1}\right)^{-1}\left(Y_{4,-1}\right)^{-1} + \left(Y_{2,0}\right)^{-1}\right) \quad =  \mathcal{J}\left({\color{red}[L(Y_{2,-6})]_t }\right).
\end{align*}

\vspace*{2em}

\begin{rem}\label{remfinale}
Note here that the coefficient of $Y_{2,-4}\left(Y_{2,-2}\right)^{-1}$ is $t+t^{-1}$. This coefficient is actually the only coefficient of $[L(Y_{2,-6})]_t$ to not be equal to 1. The quantum cluster variables being bar-invariant, all coefficients in the decomposition of the cluster variables into sums of commutative polynomials in the variables $Y_{i,r}^{\pm 1}$ are symmetric polynomials in $t^{\pm 1}$ ($P(t^{-1})=P(t)$). Thus this is the only coefficient that could have been different from the corresponding one in $[L(Y_{2,-6})]_t$, as it could also have been equal to $2$, or any $t^k + t^{-k}$. 
\end{rem}

\bibliographystyle{alpha}
\bibliography{article}

\begin{thebibliography}{{Bit}19b}

\bibitem[BFZ05]{CA3}
A.~Berenstein, S.~Fomin, and A.~Zelevinsky.
\newblock {Cluster algebras. {III}. {U}pper bounds and double {B}ruhat cells}.
\newblock {\em Duke Math. J.}, 126(1):1--52, 2005.

\bibitem[Bit19a]{LEA}
L.~Bittmann.
\newblock {Asymptotics of Standard Modules of Quantum Affine Algebras}.
\newblock {\em Algebras and Representation Theory}, 22(5):1209--1237, Oct 2019.

\bibitem[{Bit}19b]{LEA2}
L.~{Bittmann}.
\newblock {Quantum Grothendieck rings as quantum cluster algebras}.
\newblock {\em arXiv e-prints}, page arXiv:1902.00502, Feb 2019.

\bibitem[BZ05]{qCA}
A.~Berenstein and A.~Zelevinsky.
\newblock {Quantum cluster algebras}.
\newblock {\em Adv. Math.}, 195(2):405--455, 2005.

\bibitem[Cha02]{BGAT}
V.~Chari.
\newblock {Braid group actions and tensor products}.
\newblock {\em Int. Math. Res. Not.}, (7):357--382, 2002.

\bibitem[CP95a]{GQG}
V.~Chari and A.~Pressley.
\newblock {\em {A guide to quantum groups}}.
\newblock Cambridge University Press, Cambridge, 1995.
\newblock Corrected reprint of the 1994 original.

\bibitem[CP95b]{CP95}
V.~Chari and A.~Pressley.
\newblock {Quantum affine algebras and their representations}.
\newblock In {\em {Representations of groups ({B}anff, {AB}, 1994)}}, volume~16
  of {\em {CMS Conf. Proc.}}, pages 59--78. Amer. Math. Soc., Providence, RI,
  1995.

\bibitem[Dav18]{PQCA}
B.~Davison.
\newblock {Positivity for quantum cluster algebras}.
\newblock {\em Ann. of Math. (2)}, 187(1):157--219, 2018.

\bibitem[FM01]{CqC}
E.~Frenkel and E.~Mukhin.
\newblock {Combinatorics of {$q$}-characters of finite-dimensional
  representations of quantum affine algebras}.
\newblock {\em Comm. Math. Phys.}, 216(1):23--57, 2001.

\bibitem[FR99]{QCRQAA}
E.~Frenkel and N.~Reshetikhin.
\newblock {The {$q$}-characters of representations of quantum affine algebras
  and deformations of {$\mathscr{W}$}-algebras}.
\newblock In {\em {Recent developments in quantum affine algebras and related
  topics ({R}aleigh, {NC}, 1998)}}, volume 248 of {\em {Contemp. Math.}}, pages
  163--205. Amer. Math. Soc., Providence, RI, 1999.

\bibitem[FZ02]{CA1}
S.~Fomin and A.~Zelevinsky.
\newblock {Cluster algebras. {I}. {F}oundations}.
\newblock {\em J. Amer. Math. Soc.}, 15(2):497--529, 2002.

\bibitem[FZ03]{CA2}
S.~Fomin and A.~Zelevinsky.
\newblock {Cluster algebras. {II}. {F}inite type classification}.
\newblock {\em Invent. Math.}, 154(1):63--121, 2003.

\bibitem[FZ07]{CA4}
S.~Fomin and A.~Zelevinsky.
\newblock {Cluster algebras. {IV}. {C}oefficients}.
\newblock {\em Compos. Math.}, 143(1):112--164, 2007.

\bibitem[GG18]{GRA1}
J.~Grabowski and S.~Gratz.
\newblock {Graded quantum cluster algebras of infinite rank as colimits}.
\newblock {\em J. Pure Appl. Algebra}, 222(11):3395--3413, 2018.

\bibitem[GL14]{GRA2}
J.~Grabowski and S.~Launois.
\newblock {Graded quantum cluster algebras and an application to quantum
  {G}rassmannians}.
\newblock {\em Proc. Lond. Math. Soc. (3)}, 109(3):697--732, 2014.

\bibitem[Gra15]{GRA3}
J.~Grabowski.
\newblock {Graded cluster algebras}.
\newblock {\em J. Algebraic Combin.}, 42(4):1111--1134, 2015.

\bibitem[Her03]{tAOE}
D.~Hernandez.
\newblock {{$t$}-analogues des op{\'e}rateurs d'{\'e}crantage associ{\'e}s aux
  {$q$}-caract{\`e}res}.
\newblock {\em Int. Math. Res. Not.}, (8):451--475, 2003.

\bibitem[Her04]{AAqtC}
D.~Hernandez.
\newblock {Algebraic approach to q,t-characters}.
\newblock {\em Advances in Mathematics}, 187(1):1--52, 2004.

\bibitem[Her05]{MqtC}
D.~Hernandez.
\newblock {Monomials of {$q$} and {$q, t$}-characters for non simply-laced
  quantum affinizations}.
\newblock {\em Math. Z.}, 250(2):443--473, 2005.

\bibitem[Her06]{KRC}
D.~Hernandez.
\newblock {The {K}irillov-{R}eshetikhin conjecture and solutions of
  {$T$}-systems}.
\newblock {\em J. Reine Angew. Math.}, 596:63--87, 2006.

\bibitem[HJ12]{ARDRF}
D.~Hernandez and M.~Jimbo.
\newblock {Asymptotic representations and {D}rinfeld rational fractions}.
\newblock {\em Compos. Math.}, 148(5):1593--1623, 2012.

\bibitem[HL10]{CAQAA}
D.~Hernandez and B.~Leclerc.
\newblock {Cluster algebras and quantum affine algebras}.
\newblock {\em Duke Math. J.}, 154(2):265--341, 2010.

\bibitem[HL15]{QGR}
D.~Hernandez and B.~Leclerc.
\newblock {Quantum {G}rothendieck rings and derived {H}all algebras}.
\newblock {\em J. Reine Angew. Math.}, 701:77--126, 2015.

\bibitem[HL16a]{ACAA}
D.~Hernandez and B.~Leclerc.
\newblock {A cluster algebra approach to {$q$}-characters of
  {K}irillov-{R}eshetikhin modules}.
\newblock {\em J. Eur. Math. Soc. (JEMS)}, 18(5):1113--1159, 2016.

\bibitem[HL16b]{CABS}
D.~Hernandez and B.~Leclerc.
\newblock {Cluster algebras and category {$\mathcal{O}$} for representations of
  {B}orel subalgebras of quantum affine algebras}.
\newblock {\em Algebra Number Theory}, 10(9):2015--2052, 2016.

\bibitem[HO19]{HO19}
D.~Hernandez and H.~Oya.
\newblock {Quantum {G}rothendieck ring isomorphisms, cluster algebras and
  {K}azhdan-{L}usztig algorithm}.
\newblock {\em Adv. Math.}, 347:192--272, 2019.

\bibitem[IM94]{RSP}
K.~Iohara and F.~Malikov.
\newblock {Rings of skew polynomials and Gel'fand-Kirillov conjecture for
  quantum groups}.
\newblock {\em Comm. Math. Phys.}, 164(2):217--237, 1994.

\bibitem[Kac90]{IDLA}
V.~Kac.
\newblock {\em {Infinite-dimensional {L}ie algebras}}.
\newblock Cambridge University Press, Cambridge, third edition, 1990.

\bibitem[KNS94]{FRSLM}
A.~Kuniba, T.~Nakanishi, and J.~Suzuki.
\newblock {Functional relations in solvable lattice models. I. Functional
  relations and representation theory}.
\newblock {\em Internat. J. Modern Phys. A}, 9:5215--5266, 1994.

\bibitem[Nak01]{tAqC}
H.~Nakajima.
\newblock {{$T$}-analogue of the {$q$}-characters of finite dimensional
  representations of quantum affine algebras}.
\newblock In {\em {Physics and combinatorics, 2000 ({N}agoya)}}, pages
  196--219. World Sci. Publ., River Edge, NJ, 2001.

\bibitem[Nak03a]{tAqCKR}
H.~Nakajima.
\newblock {{$t$}-analogs of {$q$}-characters of {K}irillov-{R}eshetikhin
  modules of quantum affine algebras}.
\newblock {\em Represent. Theory}, 7:259--274, 2003.

\bibitem[Nak03b]{NtqAD}
H.~Nakajima.
\newblock {{$t$}-analogs of {$q$}-characters of quantum affine algebras of type
  {$A_n,D_n$}}.
\newblock In {\em {Combinatorial and geometric representation theory ({S}eoul,
  2001)}}, volume 325 of {\em {Contemp. Math.}}, pages 141--160. Amer. Math.
  Soc., Providence, RI, 2003.

\bibitem[Nak04]{QVtA}
H.~Nakajima.
\newblock {Quiver varieties and {$t$}-analogs of {$q$}-characters of quantum
  affine algebras}.
\newblock {\em Ann. of Math. (2)}, 160(3):1057--1097, 2004.

\bibitem[Nak10]{NE8}
H.~Nakajima.
\newblock {{$t$}-analogs of {$q$}-characters of quantum affine algebras of type
  {$E_6,E_7,E_8$}}.
\newblock In {\em {Representation theory of algebraic groups and quantum
  groups}}, volume 284 of {\em {Progr. Math.}}, pages 257--272.
  Birkh{\"a}user/Springer, New York, 2010.

\bibitem[Nak11]{QVCA}
H.~Nakajima.
\newblock {Quiver varieties and cluster algebras}.
\newblock {\em Kyoto J. Math.}, 51(1):71--126, 2011.

\bibitem[NN11]{oFMA}
W.~Nakai and T.~Nakanishi.
\newblock {On {F}renkel-{M}ukhin algorithm for {$q$}-character of quantum
  affine algebras}.
\newblock In {\em {Exploring new structures and natural constructions in
  mathematical physics}}, volume~61 of {\em {Adv. Stud. Pure Math.}}, pages
  327--347. Math. Soc. Japan, Tokyo, 2011.

\bibitem[Qin17]{QIN}
F.~Qin.
\newblock {Triangular bases in quantum cluster algebras and monoidal
  categorification conjectures}.
\newblock {\em Duke Math. J.}, 166(12):2337--2442, 2017.

\bibitem[Tur18]{Bolor}
B.~Turmunkh.
\newblock {{$(q,t)$}-characters of {K}irillov-{R}eshetikhin modules of type
  {$A_r$} as quantum cluster variables}.
\newblock {\em Electron. J. Combin.}, 25(1):Paper 1.10, 45, 2018.

\bibitem[VV02]{SMQAA}
M.~Varagnolo and E.~Vasserot.
\newblock {Standard modules of quantum affine algebras}.
\newblock {\em Duke Math. J.}, 111(3):509--533, 2002.

\bibitem[VV03]{PSqGR}
M.~Varagnolo and E.~Vasserot.
\newblock {Perverse sheaves and quantum {G}rothendieck rings}.
\newblock In {\em {Studies in memory of {I}ssai {S}chur
  ({C}hevaleret/{R}ehovot, 2000)}}, volume 210 of {\em {Progr. Math.}}, pages
  345--365. Birkh{\"a}user Boston, Boston, MA, 2003.

\end{thebibliography}

\end{document}